\documentclass[11pt,reqno]{amsart}
\usepackage{amsmath, amsthm, amscd, amssymb, amsfonts, latexsym, graphicx}
\usepackage{fullpage}
\usepackage{bm}
\usepackage[all]{xy}
\usepackage{tikz}
\usepackage{mathrsfs}
\usepackage{dsfont}

\newcommand{\kommentar}[1]{}

\newcommand{\eps}{\varepsilon}

\newcommand{\acom}[1]{{\color{blue}{Alexandra: #1}} }
\newcommand{\hcom}[1]{{\color{red}{Hung: #1}} }

\newcommand{\N}{\mathbb N}

\numberwithin{equation}{section}

\newcommand{\sumplus}{\sideset{}{^+}\sum}

\newcommand{\sumstar}{\sideset{}{^*}\sum}

\renewcommand{\pmod}[1]{\,(\mathrm{mod}\,#1)}

\newtheorem{lem}{Lemma}[section]
\newtheorem{prop}[lem]{Proposition}
\newtheorem{thm}[lem]{Theorem}

\theoremstyle{definition}

\newtheorem{rem}[lem]{Remark}

\begin{document}

\author{Hung M. Bui, Alexandra Florea and Micah B. Milinovich}
\address{Department of Mathematics, University of Manchester, Manchester M13 9PL, UK}
\email{hung.bui@manchester.ac.uk}
\address{UC Irvine, Mathematics Department, Rowland Hall, Irvine 92697, USA}
\email{floreaa@uci.edu}
\address{Department of Mathematics, University of Mississippi, University, MS 38677 USA}
\email{mbmilino@olemiss.edu}

\title{Simultaneous non-vanishing of Dirichlet $L$-functions}

\begin{abstract} In this paper, we prove the simultaneous non-vanishing of four Dirichlet $L$-functions at any point on the critical line. More precisely, let $\chi_1,\ldots,\chi_4$ be even Dirichlet characters modulo $D_1,\ldots, D_4$ respectively, where the $D_j$ are pairwise co-prime and square-free integers. Under the Generalized Riemann Hypothesis, we prove that $\prod_{j=1}^4 L(1/2+it,\chi \chi_j) \neq 0$ for a positive proportion of Dirichlet characters $\chi \pmod q$, with $q$ prime and sufficiently large in terms of the $D_j$ and $t$ (and with an explicit relationship between $D_j, t$ and $q$). Unconditionally, we also prove a simultaneous non-vanishing result for four Dirichlet $L$-functions for  infinitely many characters $\chi \pmod q$, though in this case the proportion tends to zero as $q \to \infty$.   
\end{abstract}


\allowdisplaybreaks

\maketitle

\section{Introduction}
Central values of $L$-functions play a fundamental role in number theory and have been extensively studied. 
A guiding principle is that, aside from obvious obstructions (such as the functional equation having sign $-1$), vanishing at the central point reflects deep underlying arithmetic information.
In the case of Dirichlet $L$--functions, Chowla \cite{Chowla} conjectured that $L(1/2,\chi) \neq 0$ for any primitive quadratic Dirichlet character $\chi$. This conjecture is now widely believed to hold for all primitive characters. As evidence for this conjecture, Balasubramanian and Murty \cite{BM} proved that a positive proportion of $\chi \pmod q$ satisfy $L(1/2,\chi) \neq 0$ when $q$ is a large prime. Iwaniec and Sarnak \cite{IS1} later proved that at least $1/3-\varepsilon$ of the primitive Dirichlet characters $\chi \pmod q$ satisfy $L(1/2,\chi) \neq 0$, for any $\varepsilon>0$ and any integer $q>0$, with $q$ large enough in terms of $\varepsilon$. 
The proportion has been improved over the years, for example in work of Bui \cite{Bui} or recent work of Qin and Wu \cite{QW}.  When $q$ is prime, Khan, Mili\'cevi\'c, and Ngo \cite{KMN} have shown that the proportion of non-vanishing is at least $5/13$. 

Assuming the Generalized Riemann Hypothesis (GRH), Murty \cite{Murty} showed that $1/2-\varepsilon$ of the primitive characters $\chi \pmod q$ have $L(1/2,\chi) \neq 0$. Fixing $t\in\mathbb{R}^*$, Carneiro, Chirre, and Milinovich \cite{CCM} studied the non-vanishing of Dirichlet $L$-functions at low-lying heights using the one-level density of zeros and showed, under GRH, that the proportion of $\chi \pmod q$ with $L(\frac12+\frac{2\pi it}{\log q},\chi)\ne 0$ is at least
\[
1-\frac12\Big(1+\Big|\frac{\sin4\pi t}{4\pi t}\Big|\Big)^{-1}. 
\] This proportion is smaller than $\frac34$ and tends to $\frac34$ as $t\rightarrow 0^+$. In particular, this result shows that there are points on the critical line where three Dirichlet $L$-functions are simultaneously non-zero for a positive proportion of $\chi \pmod q$, but this method does not extend to the simultaneous non-vanishing of four $L$-functions. When $t=0$, their result recovers Murty’s theorem, while yielding improved proportions for small nonzero $t$. 

Motivated by the study of Landau-Siegel zeros, Iwaniec and Sarnak \cite{IS2} considered the simultaneous non-vanishing of $L(1/2,f)$ and $L(1/2,f \otimes \chi_D)$, where $\chi_D$ is the quadratic character modulo $D$, and where $f$ runs in a suitable family of even newforms (where one varies either the weight or the level). They showed that the average of $L(1/2,f)L(1/2,f \otimes \chi_D)$ over these families is proportional to $L(1,\chi_D)$, and thus (since these central values are non-negative) an effective lower bound for $L(1,\chi_D)$ can be established if it can be shown that $L(1/2,f)$ and $L(1/2,f \otimes \chi_D)$ are simultaneously non-zero and bounded below sufficiently often. In particular, if more than half of the even newforms satisfy $L(1/2,f) \neq 0$ and more than half satisfy $L(1/2,f \otimes \chi_D) \neq 0$,
then a positive proportion of forms are simultaneously non-zero, yielding a lower bound on $L(1,\chi_D)$.
In \cite{IS2}, Iwaniec and Sarnak show that $1/2-\varepsilon$ of the forms have the property that $L(1/2,f) \neq 0$ and $1/2-\varepsilon$ satisfy $L(1/2,f \otimes \chi_D) \neq 0$, falling just short of proving the non-existence of Landau-Siegel zeros. In the other direction, Bui, Pratt, and Zaharescu \cite{BPZ2} recently showed that sufficiently small values of $L(1,\chi_D)$ also lead to non-vanishing of the product $L(s,f)L(s,f \otimes \chi_D)$ and its derivatives at the central point. Similar results hold for Dirichlet $L$-functions \cite{BPZ, CM}.

Many simultaneous non-vanishing results for $L$-functions at the central point have been obtained in various other contexts beyond the work of Iwaniec and Sarnak. However, most of the results in the literature only establish that infinitely many $L$-functions in a family are simultaneously non-vanishing, with a proportion tending to zero as the size of the family grows. There are comparatively few results that prove simultaneous non-vanishing for a positive proportion of $L$-functions in a family; we are aware of two such results. Michel and Vanderkam \cite{MV} showed that, given three distinct Dirichlet characters $\chi_1,\chi_2,\chi_3$ of fixed moduli $D_1,D_2,D_3$ (satisfying certain technical conditions), a positive proportion of holomorphic
primitive Hecke cusp forms $f$ of weight $2$, prime level $q$ and trivial nebentypus satisfy $L(1/2,f \otimes \chi_1)L(1/2,f \otimes \chi_2)L(1/2,f \otimes \chi_3) \neq 0$, for $q$ sufficiently large in terms of $D_1,D_2,D_3$. Focusing on the family of Dirichlet characters, Zacharias \cite{Z2} proved that, given two Dirichlet characters $\chi_1,\chi_2 \pmod q$, where $q$ is a prime, a positive proportion of characters $\chi \pmod q$ satisfy $L(1/2,\chi)L(1/2,\chi \chi_1)L(1/2,\chi \chi_2) \neq 0$. 


In this paper, we study the simultaneous non-vanishing of four Dirichlet $L$-functions, with the goal of establishing a simultaneous non-vanishing result for a positive proportion of characters. We reiterate the fact that the one-level density method of \cite{CCM} cannot yield the simultaneous non-vanishing of more than three Dirichlet $L$-functions.



Throughout the paper, we let $q$ be a prime number. For $1\leq j\leq 4$, let $1\leq D_j\leq D$ and let $\chi_j$ be a primitive Dirichlet character modulo $D_j$. For simplicity, we assume that the $D_j$ are pairwise co-prime and square-free, and the $\chi_j$ are all even. 
We are now ready to state our main theorem.
\begin{thm}
\label{mainthm}
Assume GRH, and let $\varepsilon>0$. Suppose that 
$$D^{272} (1\!+\!|t|)^{10}\ll q^{11/16-\varepsilon}\qquad\text{and}\qquad D^{188}(1\!+\!|t|)\ll q^{25/32-\varepsilon}.$$
Then 
$$ \Big|  \Big\{  \chi \pmod q \text{ even primitive} : \prod_{j=1}^{4}L(\tfrac12+it,\chi\chi_j)\ne0\Big\}\Big| \gg q.$$
\end{thm}

We also prove the following unconditional result.

\begin{thm}
\label{unconditional}
Let $\varepsilon>0$ and suppose that 
$$D^{272}(1\!+\!|t|)^{10}\ll q^{11/16-\varepsilon}\qquad\text{and}\qquad D^{188}(1\!+\!|t|)\ll q^{25/32-\varepsilon}.$$
Then 
$$ \Big|  \Big\{  \chi \pmod q \text{ even primitive} : \prod_{j=1}^{4}L(\tfrac12+it,\chi\chi_j)\ne0\Big\}\Big| \gg q^{1/3-\varepsilon}\big(D(1\!+\!|t|)\big)^{-2/3}.$$
\end{thm}


We remark that the proportion in Theorem~\ref{mainthm} is explicit, although extremely small (essentially on the scale of a negative double exponential). This phenomenon arises from the fact that the implicit constant in Theorem~\ref{thm_ub}, which is used in the proof of Theorem~\ref{mainthm}, is itself very large (a double exponential). We have not tracked the implied constant here, but see \cite{harper,DFL} for discussions of its size in related contexts. By contrast, the implicit constant in Theorem~\ref{unconditional} is ineffective, due to the use of Siegel's lower bound for $L(1,\chi)$, for quadratic Dirichlet characters $\chi$ modulo $q$.

Proving Theorem \ref{mainthm} relies on evaluating mollified moments. In particular, we obtain a sharp lower bound for a mollified fourth moment and a sharp upper bound for a mollified sixth moment of Dirichlet $L$-functions. Let
\begin{equation}
\label{fourth_toevaluate}
I(\ell_1,\ell_2)=\frac{1}{\varphi^+(q)}\ \,  \sideset{}{^+} \sum_{\chi (\text{mod}\  q)} L(\tfrac12+it,\chi\chi_1)L(\tfrac12+it,\chi\chi_2)L(\tfrac12-it,\overline{\chi\chi_3})L(\tfrac12-it,\overline{\chi\chi_4})\chi(\ell_1)\overline{\chi}(\ell_2),
\end{equation}
 where $\sumplus$ denotes summation over even primitive characters and $\varphi^+(q)$ denotes the number of even primitive characters modulo $q$. The first main ingredient in the proof of Theorem \ref{mainthm} is the following result. 

\begin{thm}\label{twistedfirst}
Let $L=\max\{\ell_1,\ell_2\}$ and $\varepsilon>0$. Suppose that
\[
D^{272}L^{96}(1\!+\!|t|)^{10}\ll q^{11/16-\varepsilon}\qquad\text{and}\qquad D^{188}L^{50}(1\!+\!|t|)\ll q^{25/32-\varepsilon}.
\]
 Then 
\begin{align*}
I(\ell_1,\ell_2)&=M_{\chi_1,\chi_2,\chi_3,\chi_4}(\ell_1,\ell_2)+ \Big(\frac{D_1D_2}{D_3D_4}\Big)^{-it}\epsilon M_{\chi_3,\chi_4,\chi_1,\chi_2}(D_1D_2\ell_1,D_3D_4\ell_2)\\
&\qquad+\Big(\frac{D_1}{D_3}\Big)^{-it}\chi_1\overline{\chi_3}(q)\epsilon(\chi_1)\epsilon(\overline{\chi_3})M_{\chi_3,\chi_2,\chi_1,\chi_4}(D_1\ell_1,D_3\ell_2) \\
&\qquad+\Big(\frac{D_1}{D_4}\Big)^{-it}\chi_1\overline{\chi_4}(q)\epsilon(\chi_1)\epsilon(\overline{\chi_4})M_{\chi_4,\chi_2,\chi_3,\chi_1}(D_1\ell_1,D_4\ell_2)\\
&\qquad+\Big(\frac{D_2}{D_3}\Big)^{-it}\chi_2\overline{\chi_3}(q)\epsilon(\chi_2)\epsilon(\overline{\chi_3})M_{\chi_1,\chi_3,\chi_2,\chi_4}(D_2\ell_1,D_3\ell_2)\\
&\qquad+\Big(\frac{D_2}{D_4}\Big)^{-it}\chi_2\overline{\chi_4}(q)\epsilon(\chi_2)\epsilon(\overline{\chi_4})M_{\chi_1,\chi_4,\chi_3,\chi_2}(D_2\ell_1,D_4\ell_2)\\
&\qquad+O_\varepsilon\Big(q^{\varepsilon}\big(q^{-11/16}D^{272}L^{96}(1+|t|)^{10}\big)^{1/28}\Big)+O_\varepsilon\Big(q^{\varepsilon}\big(q^{-25/32}D^{188}L^{50}(1+|t|)\big)^{1/80}\Big),
\end{align*}
where
\begin{equation}\label{rootnumber}
\epsilon=\chi_1\chi_2\overline{\chi_3}\overline{\chi_4}(q)\epsilon(\chi_1)\epsilon(\chi_2)\epsilon(\overline{\chi_3})\epsilon(\overline{\chi_4}),
\end{equation}
\[
M_{\chi_1,\chi_2,\chi_3,\chi_4}(\ell_1,\ell_2)=\sum_{\ell_1m=\ell_2n}\frac{(\chi_1*\chi_2)(m)(\overline{\chi_3}*\overline{\chi_4})(n)}{m^{1/2+it}n^{1/2-it}}V\Big(\frac{mn}{\widehat{q}^{\,2}};t\Big),
\]
and $\epsilon(\chi)$ is given in \eqref{epsilonfactor} and $V$ is given in \eqref{formulaV+}.
Furthermore, we have
\begin{align}\label{MAiden}
M_{\chi_1,\chi_2,\chi_3,\chi_4}(\ell_1,\ell_2)&=\frac{A_{\chi_1,\chi_2,\chi_3,\chi_4}(\widetilde{\ell_1},\widetilde{\ell_2})}{\widetilde{\ell_1}^{1/2-it}\widetilde{\ell_2}^{1/2+it}} L(1,\chi_1\overline{\chi_3})L(1,\chi_1\overline{\chi_4})L(1,\chi_2\overline{\chi_3})L(1,\chi_2\overline{\chi_4}) \\
&\qquad\qquad+O_\varepsilon\big(q^{-1/2+\varepsilon}D^{5/6} \big). \nonumber 
\end{align}
Here $\widetilde{\ell_1}=\ell_1/(\ell_1,\ell_2)$, $\widetilde{\ell_2}=\ell_2/(\ell_1,\ell_2)$, and $A_{\chi_1,\chi_2,\chi_3,\chi_4}(\widetilde{\ell_1},\widetilde{\ell_2})$ is given by
\begin{align*}
&\prod_{\substack{p^{\lambda_1}||\widetilde{\ell_1}\\p^{\lambda_2}||\widetilde{\ell_2}}}\bigg((\chi_1*\chi_2)(p^{\lambda_2})(\overline{\chi_3}*\overline{\chi_4})(p^{\lambda_1})+\sum_{j\geq1}\frac{(\chi_1*\chi_2)(p^{j+\lambda_2})(\overline{\chi_3}*\overline{\chi_4})(p^{j+\lambda_1})}{p^{j}}\bigg)\\
&\qquad\qquad\times\bigg(1+\sum_{j\geq1}\frac{(\chi_1*\chi_2)(p^j)(\overline{\chi_3}*\overline{\chi_4})(p^j)}{p^{j}}\bigg)^{-1}\\
&\quad\times \prod_p\bigg(1+\sum_{j\geq1}\frac{(\chi_1*\chi_2)(p^j)(\overline{\chi_3}*\overline{\chi_4})(p^j)}{p^{j}}\bigg)\\
&\qquad\qquad\times\bigg(1-\frac{\chi_1\overline{\chi_3}(p)}{p}\bigg)\bigg(1-\frac{\chi_1\overline{\chi_4}(p)}{p}\bigg)\bigg(1-\frac{\chi_2\overline{\chi_3}(p)}{p}\bigg)\bigg(1-\frac{\chi_2\overline{\chi_4}(p)}{p}\bigg).
\end{align*}
\end{thm}

In the case of $D_j=1$, for $1 \leq j \leq 4$ and $\ell_1=\ell_2=1$, we recover an asymptotic for the usual (untwisted) fourth moment of Dirichlet $L$-functions, which was originally obtained in Young's breakthrough paper \cite{Y}. The error term in \cite{Y} was improved upon in \cite{BFKMM}. 
We note that our theorem also generalizes work of Hough \cite{H}, who considered the twisted fourth moment for $D_j=1$ for $1 \leq j \leq 4$, $t=0$, and $\ell_1,\ell_2$ square-free. Zacharias \cite{Z1} also established an asymptotic formula for the twisted fourth moment for $\ell_1,\ell_2$ cube-free (and again $D_j=1$ and $t=0$). Both of these results hold for $q$ prime. Gao and Zhao \cite{GZ} considered a shifted twisted fourth moment in the case of $q$ being a certain prime power and general twists $\ell_1,\ell_2$, with the restriction on $q$ being later removed in \cite{GWZ} (these results all assume $D_j=1$).

The second main ingredient in the proof of Theorem \ref{mainthm} is obtaining sharp bounds on high mollified moments of Dirichlet $L$-functions. This builds on ideas introduced by Soundararajan \cite{sound}  and Harper \cite{harper} who proved almost-sharp and sharp bounds for moments of the Riemann zeta-function, respectively. In the study of mollified moments, a key step is the choice of mollifier. We construct an Euler product mollifier rather than the classical Dirichlet series mollifier.  Lester and Radziwi\l\l  \ introduced the Euler product mollifier in \cite{LR} to study sign changes of Fourier coefficients of half-integral weight modular forms; this construction was subsequently used to obtain non-vanishing results for cubic $L$-functions \cite{DFL} and to obtain weighted central limit theorems for central values of $L$-functions \cite{BELP}, among other results.

Let $M(s,\chi)$ be the mollifier defined in equations \eqref{mj}, \eqref{mollifier}, \eqref{beta_j}, \eqref{condition_c}, and \eqref{sj}.
Then we prove the following upper bound for mollified moments.

\begin{thm}
\label{thm_ub}
Assume GRH. Let $\psi$ be a primitive character modulo $D$, with $D<q$ and $|t| \leq q^{O(1)}$. For $k\in\mathbb{N}$, we have
$$ \frac{1}{\varphi^+(q)}  \, \, \, \sumplus_{\chi \pmod q} \big| L( \tfrac{1}{2}+it,\chi \psi)M (\tfrac12+it,\chi \psi ) \big|^k \ll_k 1.$$
\end{thm}

Using Theorem \ref{twistedfirst}, we will prove the following lower bound.

\begin{thm}
\label{thm_lb}
Assume GRH and let $\varepsilon>0$. Suppose that condition \eqref{condition_c} on $\beta_K$ holds with $k=6$, and further assume that
\begin{equation}
\label{further_condition_c}
q^{124 \beta_K} D^{272}(1+|t|)^{10}\ll q^{11/16-\varepsilon}\qquad\text{and}\qquad q^{130 \beta_K} D^{188}(1+|t|)\ll q^{25/32-\varepsilon}.
\end{equation}
Then
\begin{equation*}
 \frac{1}{\varphi^+(q)} \ \, \sumplus_{\substack{\chi\pmod q}}\prod_{j=1}^2LM(\tfrac12+it,\chi\chi_j)\prod_{k=3}^4LM(\tfrac12-it,\overline{\chi\chi_k}) \gg 1.
\end{equation*}
\end{thm} 


We note that Theorem~\ref{twistedfirst} is unconditional, while Theorem~\ref{thm_lb} above is conditional on GRH. This is because we must show that the product of four $L$-functions in \eqref{MAiden}, when multiplied by the (short) Euler product arising from the mollifier, contributes a term of constant size. To achieve this, we approximate these four $L$-functions by short Euler products, which requires GRH. 

We also note that the precise form of the Euler product mollifier is crucial, and even minor modifications (for example, replacing the M\"obius function in its definition with the Liouville function) would make the problem significantly more difficult. Indeed, proving Theorem~\ref{thm_lb} amounts to showing that a certain Euler product is bounded below by a positive constant. Once the contribution of sufficiently large primes is controlled, one must verify that the Euler factors at small primes do not vanish. This step is highly sensitive to the choice of mollifier, and with our specific construction the local factors at small primes turn out to be (perhaps surprisingly) simple.

\subsection{Some ideas of proof} The proof consists of two (somewhat) independent parts. In order to obtain a positive proportion of simultaneous non-vanishing, we use H\"{o}lder’s inequality (see \eqref{holder}), which reduces the problem to obtaining a lower bound for a mollified fourth moment and an upper bound for a mollified sixth moment. We note that obtaining an asymptotic formula for the mollified sixth moment is well beyond reach of current methods, unless one introduces additional averaging (see \cite{CIS,CLMR}). Hence, we restrict ourselves to establishing a sharp upper bound for this moment. With this goal in mind, we introduce an Euler product mollifier (first introduced in \cite{LR} and used in various other contexts, see \cite{DFL,BELP,bfn}). By adapting the methods of \cite{sound,harper}, we obtain sharp bounds for all positive integral mollified moments, conditional on GRH.


After introducing the mollifier and establishing the required sharp upper bounds, we turn to the fourth mollified moment, and we obtain an asymptotic formula for the twisted fourth moment. As noted previously, Theorem~\ref{twistedfirst} generalizes previous results in the literature, which all treat the case $D_j=1$ for $1 \le j \le 4$. Introducing the four characters $\chi_1,\ldots,\chi_4$ introduces several new technical difficulties, including the proof and use of a new Voronoi summation formula for convolutions of characters as well as the use of a version of the Kuznetsov trace formula for twists of twisted Kloosterman sums. Moreover, the computation of the off-diagonal main terms is considerably more involved and requires establishing a delicate functional equation for complicated Euler products in order to show cancellation among the integrals that arise (see Lemmas~\ref{identity} and \ref{secondidentity}).

Having obtained an asymptotic formula for the twisted fourth moment (which is unconditional), we then need to obtain a lower bound for the mollified fourth moment. Due to the complexity of the Euler product mollifier, this part of the argument is technically involved. In order to show that the $L$-functions and the mollifiers ``interact'' with each other, we approximate the $L$-functions which appear in \eqref{MAiden} by short Euler products, which requires GRH. Using the structure of the Euler product mollifier, we show that the contribution from the large primes is bounded below by a constant, and it remains to show that the small primes contribute non-zero Euler factors. Doing so relies on what seems like a ``miracle'': the Euler product, though initially seeming very complicated, simplifies significantly in the case of small primes, and we are left with a simple expression that we can show does not vanish.

Finally, obtaining the unconditional result in Theorem \ref{unconditional} relies on showing that the main term in Theorem \ref{twistedfirst}, in the case of $\ell_1=\ell_2=1$, is non-zero. This is straightforward to verify when the $D_j$ are sufficiently large, but we are left with a small finite number of cases where we need to directly show that the main term is non-zero. We do this by considering permutations of the characters $\chi_1,\ldots,\chi_4$ and starting out with six different moment expressions. Applying Theorem \ref{twistedfirst}, we reduce the problem to showing that in each case a certain $6 \times 6$ determinant whose entries are products of various character values and Gauss sums is non-zero. These determinants can be calculated using Mathematica. 


The paper is organized as follows. In Section~\ref{section_ub}, we define the mollifier and prove sharp upper bounds for the mollified moments. In Section~\ref{section_ub_unm}, we establish upper bounds for the unmollified moments, which are used in the proof of Theorem~\ref{thm_ub}.  We begin the proof of Theorem~\ref{twistedfirst} in Section~\ref{diagonal_section}, where we evaluate the diagonal term in the asymptotic formula for the twisted moment. The treatment of the off-diagonal terms begins in Section~\ref{offdiagonal1}, where we bound the contribution from both the balanced and unbalanced ranges and obtain estimates for the error terms. In Section~\ref{odmsection}, we evaluate the off-diagonal main terms arising from the balanced range and complete the proof of Theorem~\ref{twistedfirst}. In Section~\ref{section_lb}, we obtain a lower bound for the mollified fourth moment and complete the proof of Theorem~\ref{thm_lb}. The proofs of Theorems~\ref{mainthm} and \ref{unconditional} are completed in Section~\ref{final_proofs}. Finally, in the appendix we prove a Voronoi summation formula for convolutions of characters.

\subsection{Notation}

Throughout the paper we assume that $q$ is prime, $D_1,\ldots,D_4$ are square-free and pairwise co-prime, and $1\leq D_j\leq D$ for $j=1,\ldots,4$. Let $\varphi^+(q)$ be the number of even primitive characters modulo $q$ and let $\sumplus$ denote summation over even primitive characters. For $1 \leq j \leq 4$, let $\chi_j$ be even primitive characters modulo $D_j$, respectively.  Define
\[
\widehat{q}=\frac{q(D_1D_2D_3D_4)^{1/4}}{\pi}.
\]

For $x \in \mathbb{R}$, let $e(x) = e^{2\pi i x}$. If $\chi$ is a primitive character modulo $q$, then the Gauss sum $\tau(\chi)$ is defined by
\[
\tau(\chi)=\sum_{a (\text{mod}\  q)}  \chi(a) \, e\Big(\frac{a}{q}\Big).
\]
We also define the normalized Gauss sum by \begin{equation}
 \epsilon(\chi)=\tau(\chi)/\sqrt{q},  \label{epsilonfactor}
\end{equation} so that $|\epsilon(\chi)|=1$.

Let $S(m,n;c)$ be the Kloosterman sum
\begin{equation*}
S(m,n;c)=\sideset{}{^*}\sum_{a \pmod
c}e\Big(\frac{ma+n\overline{a}}{c}\Big),
\end{equation*}
where $\sumstar$ denotes summation over $(a,c)=1$, and $\overline{a}$ denotes the inverse of $a \pmod c$.
 We denote by $S_\chi(m,n;c)$ the hybrid Gauss--Kloosterman sum
\[
S_\chi(m,n;c)=\sideset{}{^*}\sum_{a \pmod 
c}\chi(a) \, e\Big(\frac{ma+n\overline{a}}{c}\Big).
\]
 The Ramanujan sum $c_q(n)$ is given by $c_q(n)=S(0,n;q)$.

 For a dyadic partition of unity, we let $\omega$ be a smooth non-negative function supported in $[1, 2]$ such that
\begin{equation*}\label{eq:funcF}
	\sum_{M}\omega\Big(\frac xM\Big)=1,
\end{equation*}
where $M$ runs over a sequence of real numbers with $\#\{M: M\leq X\}\ll \log X$. We define
\[
\omega^+(x)=\frac{\omega(x)}{x^{1/2+it}},\qquad \omega^-(x)=\frac{\omega(x)}{x^{1/2-it}}\qquad\text{and}\qquad \omega_1(x)=\frac{\omega(x)}{x}.
\]
We also denote by $\omega_0$ a nonnegative, Schwartz-class function satisfying $\omega_0(x)\geq 1$ for $0\leq x\leq 1$.

We let $\theta = 7/64$, which is the best known exponent toward the Ramanujan--Petersson conjecture; see \cite{KS}. We also let $\varepsilon$ denote a sufficiently small positive constant, whose value may change from line to line.

\textbf{Acknowledgments:} The authors are very grateful to Xiannan Li for a suggestion that helped remove a previous assumption of the Ramanujan-Petersson conjecture in Theorem \ref{twistedfirst}, and to Matt Young for helpful discussions regarding the Kuznetsov trace formula. AF was supported by the National Science Foundation (NSF CAREER grant DMS-2339274). MBM was supported in part by NSF grant DMS-2401461.

\section{Upper bounds for mollified moments}
\label{section_ub}
In this section we will prove Theorem \ref{thm_ub}. Fix $\psi$ a Dirichlet character modulo $D$, for some $D<q$ (note that in this section $D$ is the modulus of $\psi$ and is not to be confused with the maximum of $D_j$, for $1 \leq j \leq 4$.)
\subsection{Defining the mollifier and some setup}

First we define the Euler product mollifier as follows.
We split the primes into intervals
$$I_0= (1, q^{\beta_0}], \, I_1=(q^{\beta_0}, q^{\beta_1}], \ldots, I_K= (q^{\beta_{K-1}}, q^{\beta_K}],$$ 
for some parameters $\beta_j$ to be defined later (see equation \eqref{beta_j}).
For $u \leq K$, let
$$P_{I_j}(\chi; u) = \sum_{p \in I_j} \frac{\chi(p) a(p;u)}{p^{1/2+it}} ,$$ where
$$a(p;u) = \Big(1 - \frac{\log p}{\beta_u \log q} \Big) \frac{1}{p^{\frac{\lambda}{\beta_u \log q}}},  $$ with $\lambda$ the unique solution to $e^{-\lambda} = \lambda+\lambda^2/2$.
We extend $a(p;u)$ to a completely multiplicative function in the first variable. 
We also define
\begin{equation}
\label{mj}
M_j(\tfrac12+it,\chi ) := \sum_{\substack{p|n \Rightarrow p \in I_j \\ \Omega(n) \leq \ell_j}} \frac{ \chi(n) \mu(n) a(n;K) \nu(n)}{n^{1/2+it}},
\end{equation} for parameters $\ell_j$ which we define later (see equation \eqref{sj}), where $\nu(n)$ is the multiplicative function which, on prime powers, is equal to $\nu(p^a) = 1/a!$. 
Also denote by
$$\nu_k(n) = \sum_{ n = n_1 \cdot \ldots \cdot n_k } \nu(n_1) \cdot \ldots \nu(n_k).$$
Let
\begin{equation} 
\label{mollifier}
M(\tfrac12+it,\chi ) = \prod_{j=0}^K M_j(\tfrac12+it,\chi).
\end{equation}
We have that 
$$M_j(\tfrac12+it,\chi )^k= \sum_{\substack{p |n \Rightarrow p \in I_j \\ \Omega(n) \leq k \ell_j}}\frac{ \chi (n)  a(n;K) \alpha_k(n;\ell_j) }{n^{1/2+it}} ,$$
where 
$$ \alpha_k(n; \ell_j) = \sum_{\substack{ n = n_1 \cdot \ldots \cdot n_k \\ \Omega(n_i) \leq \ell_j}} \nu(n_1) \mu(n_1) \cdot \ldots \cdot \nu(n_k) \mu(n_k).$$ 
We denote by 
$$\alpha_k (n) = \sum_{\substack{ n = n_1 \cdot \ldots \cdot n_k \\ }} \nu(n_1) \mu(n_1) \cdot \ldots \cdot \nu(n_k) \mu(n_k).$$

Let 
$$D_{j,k} (\chi ) = \prod_{r=0}^j  (1+e^{-\ell_r/2}) E_{\ell_r}(k \Re P_{I_r}(\chi ;j)),$$ where
$$E_{\ell}(t) = \sum_{s \leq \ell} \frac{t^s}{s!}.$$ 
Also let  
$$S_{j,k}(\chi ) = \exp \Big( k \Re \sum_{p \leq q^{\beta_j/2}} \frac{ \chi (p^2)b(p;j)}{p^{1+2it}} \Big),$$ where
$$b(p;j) = \Big(1 - \frac{2 \log p}{\beta_j \log q} \Big) \frac{1}{p^{\frac{2\lambda}{\beta_j \log q}}}.$$
Note that we have 
$$E_{\ell_r} (k \Re P_{I_r} (\chi \psi;j)) = \sum_{\substack{p|mn \Rightarrow p \in I_r \\\Omega(mn) \leq \ell_r}}\frac{ (k/2)^{\Omega(mn)} (\chi \psi)(n) \overline{ (\chi \psi)(m)} a(n;j) a(m;j) \nu(n) \nu(m)}{n^{1/2+it} m^{1/2-it}} .$$
For $0 \leq r \leq K$, let 
$$\mathcal{T}_r = \Big\{  \chi \pmod q \text{ even primitive}, \max_{r \leq u \leq K} |\Re P_{I_r}(\chi\psi ;u) | \leq \frac{\ell_r}{ke^2} \Big\}.$$

 \subsection{Some key lemmas and proposition}
 Using Theorem $2.1$ in \cite{chandee_ub}, we have the following lemma, which is used in the proof of Proposition \ref{main_prop}.
\begin{lem}
\label{lem_log}
Let $\chi$ be a primitive character modulo $q$. Let $\lambda$ be the solution to $e^{-\lambda} = \lambda+\lambda^2/2$. Then
\begin{align*}
\log |L(\tfrac12+it,\chi)| \leq \Re \sum_{n \leq x} \frac{ \Lambda(n) \chi(n) }{n^{1/2+it+\frac{\lambda}{\log x}} \log n } \frac{ \log(x/n)}{\log x} + \frac{1+\lambda}{2}\cdot  \frac{ \log (q(1+|t|))}{\log x} + O \big((\log x)^{-2} \big).
\end{align*} 
\end{lem}
Using Lemma \ref{lem_log} we obtain the following proposition, whose proof  is similar to that of Proposition $4.2$ in \cite{bfn}, so we will skip it.
\begin{prop}
\label{main_prop}
Let $\psi$ be a primitive character modulo $D$ for some $D<q$. 
We either have
$$ \max_{0 \leq u \leq K} \Big| \Re P_{I_0}(\chi \psi ;u)\Big|> \frac{\ell_0}{ke^2},$$ or
\begin{align*}
\big|L(\tfrac12+it,\, & \chi \psi)\big|^k   \ll \exp \Big(  \frac{k (1+\lambda) \log (qD(1+|t|))}{2 \beta_K \log q} \Big) D_{K,k}(\chi \psi ) S_{K,k}(\chi \psi ) \\
&+ \sum_{\substack{0 \leq j \leq K-1 \\ j < u \leq K}} \exp \Big( \frac{ k(1+\lambda) \log(qD(1+|t|))}{2\beta_j \log q} \Big) D_{j,k} (\chi \psi)S_{j,k}(\chi \psi ) \Big(\frac{ke^2 \Re P_{I_{j+1}}(\chi \psi ;u)}{\ell_{j+1}}  \Big)^{s_{j+1}},
\end{align*}
for any even integers $s_{j+1}$.
\end{prop}

We also need the following two lemmas.
\begin{lem}
\label{technical_lemma}
Suppose that $(k+2) \sum_{r=0}^K \ell_r \beta_r <1$ and for any $j \leq K-1$,  we have $(k+2) \sum_{r=0}^j \ell_r \beta_r + k \sum_{r=j+1}^{K} \ell_r \beta_r + 2s_{j+1} \beta_{j+1} <1$. 
Then 
 \begin{align*} & (1)\, \,  \sumplus_{\chi \pmod q} \big|  M(\tfrac12+it,\chi \psi ) \big|^{2k} \big( \Re P_{I_0}(\chi \psi ;u)\big)^{2s_0} \ll q e^{k^2K(1+\varepsilon)} \frac{ (2s_0)!k^{2s_0}\sqrt{s_0}}{(s_0)!2^{s_0}} (\log \log q^{\beta_0})^{2s_0}.\\ 
 & (2)\, \,  \sumplus_{\chi \pmod q} D_{K,k}(\chi \psi ) ^2 \big|M(\tfrac12+it,\chi \psi )\big|^{2k} \ll q \Big(1+ \frac{1}{2^{\ell_0}} (\log q^{\beta_0})^{16k^2} \Big).  \\
& (3)\, \,   \sumplus_{\chi \pmod q} D_{j,k}(\chi \psi) ^2 \big|M(\tfrac12+it,\chi \psi )\big|^{2k} \big(\Re P_{I_{j+1}}(\chi \psi ;u)\big)^{2s_{j+1}} \ll q e^{k^2(K-j)(1+\varepsilon)} \frac{(2s_{j+1})! k^{2s_{j+1}} \sqrt{s_{j+1}}}{(s_{j+1})! 2^{s_{j+1}}} \\  
&  \qquad \qquad \times \Big(\sum_{p \in I_{j+1}} \frac{1}{p} \Big)^{2s_{j+1}} \Big(1+ \frac{1}{2^{\ell_0}} (\log q^{\beta_0})^{16k^2} \Big).
\end{align*}
\end{lem}
We also have the following.
\begin{lem}
\label{lemma_squares}
For $j \leq K$, we have
$$\sumplus_{\chi \pmod q} S_{j,k}(\chi \psi)^2 \ll q.$$
\end{lem} 
\begin{proof}
The proof is similar to the proof of Lemma $4.4$ in \cite{bfn}, so we will skip it. 
\end{proof}
\begin{proof}[Proof of Lemma \ref{technical_lemma}]
The proof is similar to the proof of Lemma $4.3$ in \cite{bfn}, so we will only show the details (and differences from \cite{bfn}) for the second part of the lemma; the first and third follow in exactly the same way as in \cite{bfn}. The term we need to bound is
\begin{align}
\label{to_bd}
\sumplus_{\chi \pmod q} &  \prod_{r=0}^K (1+e^{-\ell_r/2})^2 \mathcal{E}_r (\chi \psi),
\end{align}
where
\begin{align}
\mathcal{E}_r(\chi  \psi) &= \sum_{\substack{ p | m_{r_1} n_{r_1} m_{r_2} n_{r_2} \Rightarrow p \in I_r \\ \Omega( m_{r_1} n_{r_1}) \leq \ell_r \\ \Omega(m_{r_2} n_{r_2}) \leq \ell_r}} \frac{ (k/2)^{\Omega(m_{r_1}n_{r_1}m_{r_2} n_{r_2})} \chi \psi (n_{r_1} n_{r_2} ) \overline{ \chi \psi(m_{r_1} m_{r_2})} a(n_{r_1} n_{r_2} m_{r_1} m_{r_2};K)  }{ (n_{r_1} n_{r_2})^{1/2+it} (m_{r_1} m_{r_2})^{1/2-it} } \nonumber  \\
& \times \nu(n_{r_1}) \nu(n_{r_2}) \nu(m_{r_1}) \nu(m_{r_2})  \sum_{\substack{p|f_rh_r \Rightarrow p \in I_r \\ \Omega(f_r) \leq k \ell_r \\ \Omega(h_r) \leq k \ell_r}} \frac{\chi \psi(f_r) \overline{\chi \psi(h_r)} a(f_r h_r;K)  \alpha_k(f_r; \ell_r) \alpha_k (h_r; \ell_r)}{f_r^{1/2+it} h_r^{1/2-it}}. \label{epsilonr}
\end{align}
Note that the summands in the sum we want to bound are positive,  so we bound \eqref{to_bd} by the sum over all characters $\chi$ modulo $q$. Using orthogonality of characters, we only keep the term with $ n_{r_1} n_{r_2} f_r \equiv  m_{r_1} m_{r_2} h_r  \pmod q$. Since $ (k+2) \sum_{r=0}^K \beta_r \ell_r <1$, it follows that this happens only when $n_{r_1} n_{r_2} f_r = m_{r_1} m_{r_2} h_r$ for each $r \leq K$. 

Then the term we need to bound after introducing the sum over characters is 
\begin{align}
q \prod_{r=0}^K E_r& :=  q \prod_{r=0}^K  \sum_{\substack{ p | m_{r_1} n_{r_1} m_{r_2} n_{r_2}f_rh_r  \Rightarrow p \in I_r \\ \Omega( m_{r_1} n_{r_1}) \leq \ell_r \\ \Omega(m_{r_2} n_{r_2}) \leq \ell_r \\\Omega(f_r) \leq k \ell_r \\ \Omega(h_r) \leq k \ell_r \\ n_{r_1} n_{r_2} f_r = m_{r_1} m_{r_2} h_r \\ (n_{r_1} n_{r_2} m_{r_1} m_{r_2} f_r h_r,D)=1}} \frac{ (k/2)^{\Omega(m_{r_1}n_{r_1}m_{r_2} n_{r_2})}a(n_{r_1} n_{r_2} m_{r_1} m_{r_2};K)  }{  (n_{r_1} n_{r_2} f_r)^{1/2+it} (m_{r_1} m_{r_2} h_r)^{1/2-it}} \nonumber \\
&\qquad \times  a(f_r h_r;K)  \nu(n_{r_1}) \nu(n_{r_2}) \nu(m_{r_1}) \nu(m_{r_2}) \alpha_k(f_r;\ell_r) \alpha_k(h_r;\ell_r) . \label{erj}
\end{align}

If $\max\{ \Omega(m_{r_1} m_{r_2}), \Omega(n_{r_1} n_{r_2})\}>\ell_r$ or if $ \max \{\Omega(f_r),\Omega(h_r)\} >k \ell_r$, then we have $$2^{\Omega(m_{r_1} m_{r_2} n_{r_1} n_{r_2} f_r h_r)} >2^{\ell_r},$$ since $k\geq 1$. It follows that
\begin{align*}
E_r & =  \sum_{\substack{ p | m_{r_1} n_{r_1} m_{r_2} n_{r_2}f_rh_r  \Rightarrow p \in I_r \\ n_{r_1} n_{r_2} f_r= m_{r_1} m_{r_2} h_r \\ (m_{r_1} m_{r_2} n_{r_1} n_{r_2} f_r h_r,D)=1 }} \frac{ (k/2)^{\Omega(m_{r_1}n_{r_1}m_{r_2} n_{r_2})}a(n_{r_1} n_{r_2} m_{r_1} m_{r_2};K)   }{ (n_{r_1} n_{r_2} f_r)^{1/2+it} (m_{r_1} m_{r_2} h_r)^{1/2-it} } \\
&\qquad \times  \nu(n_{r_1}) \nu(n_{r_2}) \nu(m_{r_1}) \nu(m_{r_2}) a(f_r h_r;K)  \alpha_k(f_r) \alpha_k(h_r)  \nonumber \\
&\qquad + O \Bigg(  \frac{1}{2^{\ell_r}} \sum_{\substack{ p | m_{r_1} n_{r_1} m_{r_2} n_{r_2}f_rh_r  \Rightarrow p \in I_r \\ n_{r_1} n_{r_2} f_r = m_{r_1} m_{r_2} h_r }} \frac{2^{\Omega(m_{r_1} m_{r_2} n_{r_1} n_{r_2} f_rh_r)} (k/2)^{\Omega(m_{r_1}n_{r_1}m_{r_2} n_{r_2})} }{\sqrt{m_{r_1} m_{r_2} n_{r_1} n_{r_2} f_rh_r}}   |\alpha_k(f_r) \alpha_k(h_r)| \Bigg),
\end{align*}
where in the last line we used the bound $\nu(a) \leq 1$ for any integer $a$. 
 Now let $M_r=m_{r_1} m_{r_2}$ and $N_r= n_{r_1} n_{r_2}$.  We then rewrite
 \begin{align}
 E_r &= \sum_{\substack{p|M_rN_r f_rh_r \Rightarrow p \in I_r \\ N_rf_r=M_rh_r \\ (M_r N_r f_r h_r,D)=1}}  \frac{ (k/2)^{\Omega(M_rN_r)}a(N_r M_r;K) \nu_2(N_r) \nu_2(M_r) }{ (N_r f_r)^{1/2+it} (M_r h_r)^{1/2-it}}  a(f_r h_r;K) \alpha_k(f_r) \alpha_k(h_r) \nonumber \\
 & \qquad + O \Bigg(  \frac{1}{2^{\ell_r}}  \sum_{\substack{ p | M_rN_rf_rh_r  \Rightarrow p \in I_r \\ N_r f_r = M_r h_r }} \frac{ (2k)^{\Omega( M_rN_rf_rh_r)} }{\sqrt{M_rN_r f_rh_r}}   \Bigg), \label{er}
 \end{align}
where we used the fact that  $|\alpha_k(f)| \leq k^{\Omega(f)}$.

Now let $A_r=(N_r, M_r)$ and $B_r=(f_r,h_r)$, and we write $N_r= A_r C_r$  and $M_r=A_r D_r$ with $(C_r,D_r)=1$.  Similarly write $f_r=B_rf_{r1}$ and $h_r=B_rh_{r1}$ with $(f_{r1}, h_{r1})=1$.  Since $N_rf_r=M_rh_r$,  it follows that $f_{r1}=D_r$ and $h_{r1}= C_r$.

The error term in the evaluation of $E_r$ (in equation \eqref{er}) then becomes (see equation $(45)$ in \cite{bfn})
\begin{equation}
 \frac{1}{2^{\ell_r}}  \sum_{\substack{ p | M_rN_rf_rh_r  \Rightarrow p \in I_r \\ N_r f_r = M_r h_r }} \frac{ (2k)^{\Omega( M_rN_rf_rh_r)} }{\sqrt{M_rN_r f_rh_r}}\ll 
\begin{cases}
\frac{1}{2^{\ell_r}}  & \mbox{ if } r \neq 0, \\
\frac{1}{2^{\ell_0}} (\log q^{\beta_0})^{16k^2} & \mbox{ if } r=0.
\end{cases}
\label{error_er}
\end{equation}

We now focus on the main term in \eqref{er}.  Similarly to \cite{bfn}, we get that
\begin{align*}
 \sum_{\substack{p| A_rB_rC_rD_r\Rightarrow p \in I_r \\ (C_r,D_r)=1\\ (A_rB_rC_rD_r,D)=1}}  &  \frac{ (k/2)^{\Omega(A_r^2 C_r D_r)} a(A_r^2 B_r^2 C_r^2 D_r^2;K)}{A_r B_r C_r D_r} \nu_2(A_r C_r) \nu_2(A_rD_r) \alpha_k(B_r C_r) \alpha_k(B_rD_r)  \\
&= \prod_{\substack{p \in I_r \\ p \nmid D}} \bigg( 1 + O \Big(\frac{1}{p^2}\Big) \bigg).
\end{align*}

Combining the above and equation \eqref{error_er}, it follows that
\begin{equation*}
q \prod_{r=0}^K E_r \ll q \Big(  1 + \frac{1}{2^{\ell_0}}(\log q^{\beta_0})^{16 k^2} \Big),
\end{equation*} which finishes the proof of the second bound in the lemma.
\end{proof}

\subsection{Choice of parameters and proof of Theorem \ref{thm_ub}}%
We choose 
 \begin{align}
 \label{beta_j}
 \beta_j = \frac{e^j}{(\log  \log q)^5},
 \end{align}
 for $j \leq K$.  Let $\beta_K=c$, where $c$ is a small constant such that
 \begin{equation}
 \label{condition_c}
  c< \min \Big\{ \frac{4 (e^{1/4}-1)^4}{e(k+2)^4},   \frac{1}{4} \Big( \frac{\log 2}{60} \Big)^{4/3}\Big\}.
  \end{equation}
  Note that with this choice of $\beta_K$, we have that $K \asymp \log \log \log q$. a
 We further choose for $j \leq K$, 
 \begin{equation}
 s_j = 2 \Big[ \frac{1}{8 \beta_j}\Big]\qquad\text{and} \qquad \ell_j=2[s_j^{3/4}/2].
 \label{sj}
 \end{equation} 
Note that the first condition in \eqref{condition_c} ensures that the conditions in Lemma \ref{technical_lemma} are satisfied. 

Using Proposition \ref{main_prop},  we write
\begin{align}
& \sumplus_{\chi \pmod q}  \big| L (\tfrac12+it,\chi \psi ) M(\tfrac12+it,\chi \psi ) \big|^k \leq \sumplus_{\substack{\chi \pmod q \\ \chi \notin \mathcal{T}_0}}  \big| L (\tfrac12+it,\chi \psi) M(\tfrac12+it,\chi \psi ) \big|^k  \label{bound1} \\
&\qquad + \sumplus_{\chi \pmod q }  \exp \Big(  \frac{k (1+\lambda) \log (qD(1+|t|))}{2 \beta_K \log q} \Big) D_{K,k}(\chi ) S_{K,k}(\chi \psi) \big|M(\tfrac12+it,\chi \psi ) \big|^k \label{bound2} \\
&\qquad+ \sumplus_{\chi \pmod q}  \sum_{\substack{0 \leq j \leq K-1 \\ j < u \leq K}} \exp \Big( \frac{ k(1+\lambda) \log(qD(1+|t|))}{2\beta_j \log q} \Big) D_{j,k} (\chi \psi) S_{j,k}(\chi \psi) \Big(\frac{ke^2 \Re P_{I_{j+1}}(\chi \psi ;u)}{\ell_{j+1}}  \Big)^{s_{j+1}} \nonumber \\
&  \qquad \qquad  \times\big|M(\tfrac12+it,\chi \psi ) \big|^k. \label{bound3}
\end{align}
For the first term, we have that for some $0 \leq u \leq K$ ,
\begin{align}
&\sumplus_{\substack{\chi \pmod q \\ \chi \notin \mathcal{T}_0}}  \big| L (\tfrac12+it,\chi \psi) M(\tfrac12+it,\chi \psi ) \big|^k \leq \sumplus_{\chi \pmod q}  \big| L (\tfrac12+it,\chi \psi) M(\tfrac12+it,\chi \psi ) \big|^k \Big( \frac{ke^2 \Re P_{I_0}(\chi \psi ;u)}{\ell_0} \Big)^{s_0} \nonumber  \\
&\quad \leq \bigg(\ \ \sumplus_{\chi \pmod q} \big|L(\tfrac12+it,\chi \psi) \big|^{2k}\bigg)^{1/2}  \bigg(\ \ \sumplus_{\chi \pmod q} \big|  M(\tfrac12+it,\chi \psi ) \big|^{2k} \Big( \frac{ke^2 \Re P_{I_0}(\chi \psi ;u)}{\ell_0} \Big)^{2s_0} \bigg)^{1/2}. \label{not_to}
\end{align}

Using equation \eqref{not_to}, Lemma \ref{technical_lemma}, Proposition \ref{ub_moments} and Stirling's formula  we have
\begin{align*}
\sumplus_{\substack{\chi \pmod q \\ \chi \notin \mathcal{T}_0}} & \big| L (\tfrac12+it,\chi \psi) M(\tfrac12+it,\chi \psi) \big|^k \leq q \exp \Big( - \frac{1}{4} s_0  \log s_0 + s_0 \log (2^{1/2}k^2 e^{3/2})+ \frac{1}{4} \log s_0 \\
&+s_0 \log \log \log q^{\beta_0}+\frac{k^2 K(1+\varepsilon)}{2} \Big) (\log q)^{k^2} = o(q),
\end{align*} 
where we used the choices \eqref{beta_j} and \eqref{sj}, as well as the fact that $K \asymp  \log \log \log q$.

We will focus on bounding \eqref{bound3}, and bounding \eqref{bound2} follows similarly. We have 
\begin{align}
& \sumplus_{\chi \pmod q}  D_{j,k} (\chi \psi) S_{j,k}(\chi \psi ) \Big(\frac{ke^2 \Re P_{I_{j+1}}(\chi \psi ;u)}{\ell_{j+1}}  \Big)^{s_{j+1}} \big|M(\tfrac12+it,\chi \psi ) \big|^k  \label{sumj} \\
& \leq \Big(  \frac{ke^2}{\ell_{j+1}} \Big)^{s_{j+1}}  \bigg( \ \ \sumplus_{\chi \pmod q} D_{j,k}(\chi \psi) ^2 \big|M(\tfrac12+it,\chi \psi )\big|^{2k}  \big(\Re P_{I_{j+1}}(\chi ;u)\big)^{2s_{j+1}}\bigg)^{1/2}  \bigg(\ \ \sumplus_{\chi \pmod q} S_{j,k}(\chi \psi )^2 \bigg)^{1/2}. \nonumber 
\end{align}  

 
Now using \eqref{sumj}, Lemma \ref{technical_lemma} and the fact that
$$ \sum_{ p \in I_{j+1}} \frac{1}{p} = 1+ O \Big( \frac{1}{\beta_j \log q} \Big) \leq 2,$$ 
we have that
\begin{align*}
&\sumplus_{\chi \pmod q}  \sum_{\substack{0 \leq j \leq K-1 \\ j < u \leq K}} \exp \Big( \frac{ k(1+\lambda) \log(qD(1+|t|))}{2\beta_j \log q} \Big) D_{j,k} (\chi \psi) S_{j,k}(\chi  \psi)   \\
&\quad \times \Big(\frac{ke^2 \Re P_{I_{j+1}}(\chi \psi ;u)}{\ell_{j+1}}  \Big)^{s_{j+1}} \big|M(\tfrac12+it,\chi \psi ) \big|^k   \\
&\qquad \ll q  \sum_{\substack{0 \leq j \leq K-1 }} (K-j) \exp \Big( \frac{ k(1+\lambda) \log(qD(1+|t|))}{2\beta_j \log q} \Big) \exp \Big( -\frac{1}{4}s_{j+1} \log s_{j+1} \\
&\qquad\quad+s_{j+1} \log (2^{3/2} k^2 e^{3/2}) +\frac{1}{4}\log s_{j+1}+ \frac{k^2(K-j)(1+\varepsilon)}{2} \Big) \nonumber \\
&\qquad  \ll q  \sum_{0 \leq j <K} (K-j) \exp \Big(\frac{C_1}{\beta_j}-\frac{C_2}{\beta_j} \log \Big( \frac{1}{\beta_j} \Big)+ \frac{k^2(K-j)(1+\varepsilon)}{2} \Big),
\end{align*}
for some explicit constants $C_1, C_2>0$.
Rearranging, we have that the above is
\begin{align}
 & \ll q \sum_{0 \leq r <K}(r+1) \exp \Big( \frac{C_1 e^{r+1}}{\beta_K}-\frac{C_2(r+1)e^{r+1}}{\beta_K} + \frac{k^2(r+1)(1+\varepsilon)}{2}+ \frac{C_2 e^{r+1} \log \beta_K}{\beta_K} \Big) \ll q. \label{firstj+1}
\end{align}

A similar bound holds for \eqref{bound2}, which finishes the proof of Theorem \ref{thm_ub}.

\section{Upper bounds for moments}
\label{section_ub_unm}
Here, we will prove upper bounds for 
(unmollified) moments. This will be a modification of the argument in the previous section, so we will skip some of the details. We will prove the following bound, which is not sharp, but which suffices for our purposes.

\begin{prop}
\label{ub_moments}
Let $\psi$ be a primitive character modulo $D<q$ and $|t| \leq q^{O(1)}$. For $k\in\mathbb{N}$, we have
$$ \sumplus_{\chi \pmod q} \big| L ( \tfrac{1}{2}+it,\chi \psi ) \big|^k \ll q (\log q)^{k^2/2}.$$
\end{prop}

We will need the following variant of Lemma \ref{technical_lemma}.

\begin{lem}
\label{technical_lemma2}
Suppose that $2 \sum_{r=0}^K \ell_r \beta_r <1$ and for any $j \leq K-1$,  we have $2 \sum_{r=0}^j \ell_r \beta_r + 2s_{j+1} \beta_{j+1} <1$. Then 
 \begin{align*} & (1)\, \, \sumplus_{\chi \pmod q}   \big( \Re P_{I_0}(\chi \psi ;u)\big)^{2s_0} \ll q \frac{(2s_0)!}{4^{s_0} s_0!} (\log \log q^{\beta_0})^{s_0}.
 \\
 &(2)\, \, \sumplus_{\chi \pmod q} D_{K,k}(\chi \psi ) ^2  \ll q (\log q^{\beta_K})^{k^2} \Big(1+ \frac{1}{2^{\ell_0}}(\log q^{\beta_0})^{4k^2} \Big) .
  \\
 &
 (3)\, \,  \sumplus_{\chi \pmod q} D_{j,k}(\chi \psi) ^2  \big(\Re P_{I_{j+1}}(\chi \psi ;u)\big)^{2s_{j+1}} \ll q (\log q^{\beta_j})^{k^2}  \frac{(2s_{j+1})!}{4^{s_{j+1}} s_{j+1}!} \Big(\sum_{p \in I_{j+1}} \frac{1}{p} \Big)^{s_{j+1}}  \Big(1+\frac{1}{2^{\ell_0}}(\log q^{\beta_0})^{4k^2}  \Big).
\end{align*}
\end{lem}
\begin{proof}
We will only sketch the proof, since it is very similar to the proof of Lemma \ref{technical_lemma}.

For the first part, proceeding as in Lemma \ref{technical_lemma}, we have that 
\begin{align}
\label{sum_i0}
\sumplus_{\chi \pmod q}   \big( \Re P_{I_0}(\chi \psi ;u)\big)^{2s_0}& \leq \frac{ (2s_0)!}{4^{s_0}} \sum_{\substack{ p |c \Rightarrow p \in I_0 \\ \Omega(c)=s_0 \\ (c,D)=1}} \frac{ a(c;u)^2 \nu(c)^2}{c} \leq \frac{(2s_0)!}{4^{s_0} s_0!} \Big(\sum_{p \in I_0} \frac{1}{p} \Big)^{s_0},
\end{align}
where we used orthogonality of characters, the fact that $2 s_0 \beta_0<1$, and the inequalities $a(c;u) \leq 1$ and $\nu(c) \leq 1$. 
To prove the second part of the Lemma, we proceed as in equation \eqref{erj} and we need to bound
\begin{align}
q \prod_{r=0}^K E_r& := q \prod_{r=0}^K  \sum_{\substack{ p | m_{r_1} n_{r_1} m_{r_2} n_{r_2}f_rh_r  \Rightarrow p \in I_r \\ \Omega( m_{r_1} n_{r_1}) \leq \ell_r \\ \Omega(m_{r_2} n_{r_2}) \leq \ell_r \\ n_{r_1} n_{r_2}  = m_{r_1} m_{r_2}  \\ (n_{r_1} n_{r_2} m_{r_1} m_{r_2} ,D)=1}} \frac{ (k/2)^{\Omega(m_{r_1}n_{r_1}m_{r_2} n_{r_2})}a(n_{r_1} n_{r_2} m_{r_1} m_{r_2};K) }{  (n_{r_1} n_{r_2} )^{1/2+it} (m_{r_1} m_{r_2} )^{1/2-it}} \nonumber \\
&\qquad \times  \nu(n_{r_1}) \nu(n_{r_2}) \nu(m_{r_1}) \nu(m_{r_2})  . \label{er_again}
\end{align}
Similarly as before (see equation \eqref{er}), with $M_r=m_{r_1} m_{r_2}$ and $N_r = n_{r_1} n_{r_2}$, we get that
\begin{align*}
 E_r &= \sum_{\substack{p|M_rN_r f_rh_r \Rightarrow p \in I_r \\ N_r=M_r \\ (M_r N_r,D)=1}}  \frac{ (k/2)^{\Omega(M_rN_r)}a(N_rM_r;K)\nu_2(N_r) \nu_2(M_r) }{ N_r^{1/2+it} M_r^{1/2-it} }   + O \bigg(  \frac{1}{2^{\ell_r}}  \sum_{\substack{ p | M_rN_r \Rightarrow p \in I_r \\ N_r  = M_r  }} \frac{ (2k)^{\Omega( M_rN_r)} }{\sqrt{M_rN_r }}   \bigg), 
 \end{align*}
and hence
\begin{align}
\label{er_term}
E_r & \leq \sum_{\substack{p|M_r \Rightarrow p \in I_r }} \frac{ (k/2)^{\Omega(M_r^2)} a(M_r;K)^2 \nu_2(M_r)^2}{M_r} + O \Big( \frac{1}{2^{\ell_r}} \sum_{p | M_r \Rightarrow p \in I_r} \frac{ (2k)^{\Omega(M_r)^2} }{M_r} \Big). 
\end{align}
We note that the error term above is bounded by 
$$ \ll_k 
\begin{cases}
\frac{1}{2^{\ell_r}} & \mbox{ if } r \neq 0, \\
\frac{1}{2^{\ell_0}} (\log q^{\beta_0})^{4k^2} & \mbox{ if } r = 0.
\end{cases}
$$
Moreover, the main term on the right hand side of equation \eqref{er_term} is bounded (up to a constant) by 
$$ \prod_{p \in I_r} \Big( 1+ \frac{k^2}{p} \Big).$$
Hence, we get that
\begin{align*}
q \prod_{r=0}^K E_r \ll q \Big(1+ \frac{1}{2^{\ell_0}}(\log q^{\beta_0})^{4k^2}  \Big) \prod_{p \leq q^{\beta_K}}  \Big( 1+ \frac{k^2}{p} \Big)\ll q (\log q^{\beta_K})^{k^2} \Big(1+  \frac{1}{2^{\ell_0}}(\log q^{\beta_0})^{4k^2} \Big).
\end{align*}
Finally, to prove the third bound in the proposition, we note that we need to bound
$$q \bigg(\prod_{r=0}^j E_r\bigg)  E_{j+1},$$ where $E_{j+1}$ is as in \eqref{sum_i0} (with $0$ replaced by $j+1$), and for $r\leq j$, $E_r$ is as in \eqref{er_again}. Proceeding as before, we get that
\begin{align*}
q \bigg(\prod_{r=0}^j E_r\bigg)  E_{j+1} &\ll q \prod_{p \leq q^{\beta_j}} \Big(1+ \frac{k^2}{p} \Big) \frac{(2s_{j+1})!}{4^{s_{j+1}} s_{j+1}!} \Big(\sum_{p \in I_{j+1}} \frac{1}{p} \Big)^{s_{j+1}} \Big(1+ \frac{1}{2^{\ell_0}}(\log q^{\beta_0})^{4k^2}  \Big) \\
& \ll (\log q^{\beta_j})^{k^2}  \frac{(2s_{j+1})!}{4^{s_{j+1}} s_{j+1}!} \Big( \log \frac{\beta_{j+1}}{\beta_j} \Big)^{s_{j+1}} \Big(1+ \frac{1}{2^{\ell_0}}(\log q^{\beta_0})^{4k^2} \Big). \end{align*}
\end{proof}

\begin{proof}
Now we will prove Proposition \ref{ub_moments}. 
We choose the parameters slightly differently than in the proof of Theorem \ref{thm_ub}. Namely, we choose
\begin{equation}
\label{betaj_again}
\beta_j = \frac{e^j (\log \log q)^2}{8k(B+2) \log q},
\end{equation}
for $j \leq K$, where $B \geq 0$ is such that $|t| \leq q^{B}$. We choose $\beta_K=c$ where $c$ is a small constant such that
\begin{equation}
c< \frac{(e^{1/4}-1)^4}{4e}.
\label{c_bound_again}
\end{equation}
We choose $s_j$ and $\ell_j$ as in equation \eqref{sj}. Note that with this choice of parameters, the conditions in Lemma \ref{technical_lemma} are satisfied.

We have 
\begin{align}
& \sumplus_{\chi \pmod q}  \big| L (\tfrac12+it,\chi \psi ) \big|^k \leq \sumplus_{\substack{\chi \pmod q \\ \chi \notin \mathcal{T}_0}}  \big| L (\tfrac12+it,\chi \psi)  \big|^k  \label{bound1_2} \\
&\quad + \sumplus_{\chi \pmod q }  \exp \Big(  \frac{k (1+\lambda) \log (qD(1+|t|))}{2 \beta_K \log q} \Big) D_{K,k}(\chi ) S_{K,k}(\chi \psi) \label{bound2_2} \\
&\quad+ \sumplus_{\chi \pmod q}  \sum_{\substack{0 \leq j \leq K-1 \\ j < u \leq K}} \exp \Big( \frac{ k(1+\lambda) \log(qD(1+|t|))}{2\beta_j \log q} \Big) D_{j,k} (\chi \psi) S_{j,k}(\chi \psi) \Big(\frac{ke^2 \Re P_{I_{j+1}}(\chi \psi ;u)}{\ell_{j+1}}  \Big)^{s_{j+1}}. \label{bound3_2}
\end{align}
For the first term, we have that for some $0 \leq u \leq K$ ,
\begin{align*}
&\sumplus_{\substack{\chi \pmod q \\ \chi \notin \mathcal{T}_0}}  \big| L (\tfrac12+it,\chi \psi) \big|^k \leq \sumplus_{\chi \pmod q}  \big| L (\tfrac12+it,\chi \psi)  \big|^k \Big( \frac{ke^2 \Re P_{I_0}(\chi \psi ;u)}{\ell_0} \Big)^{s_0} .
\end{align*}
We use the pointwise bound for $|L(1/2+it,\chi \psi)|$ as in Corollary $1.2$ in \cite{chandee_ub}, and we get that 
\begin{align*}
\sumplus_{\substack{\chi \pmod q \\ \chi \notin \mathcal{T}_0}}  \big| L (\tfrac12+it,\chi \psi) \big|^k \leq \exp \Big(  \frac{3k}{8} \frac{ \log \tilde{C}}{\log \log \tilde{C}} + \frac{23k}{25} \frac{\log \tilde{C}}{\log^2 \log \tilde{C}}\Big)\ \sumplus_{\chi \pmod q}   \Big( \frac{ke^2 \Re P_{I_0}(\chi \psi ;u)}{\ell_0} \Big)^{s_0},
\end{align*}
where $\tilde{C} =  qD (1+|t|)$. Now recall that $D<q$, and we have that $|t| \leq q^{B}$, for $B \geq 0$. Using Lemma \ref{technical_lemma2} and Stirling's formula, we get that
\begin{align}
\sumplus_{\substack{\chi \pmod q \\ \chi \notin \mathcal{T}_0}}  \big| L (\tfrac12+it,\chi \psi) \big|^k \ll q \exp& \bigg( \Big( \frac{3k(B+2)}{8} + \varepsilon\Big) \frac{ \log q}{\log \log q} - \frac{s_0}{4} \log s_0 + s_0 \log  \frac{ke^{3/2}}{2^{1/2}}  \nonumber \\
&+ \frac{s_0}{2} \log \log \log q^{\beta_0} \bigg)  = o(q),\label{o1term}
\end{align}
where we used the expression for $s_0$ from equations \eqref{sj} and \eqref{betaj_again}.

To bound the term \eqref{bound2_2}, we use Cauchy-Schwarz, the expression for $\beta_K$ in \eqref{c_bound_again} and Lemmas \ref{technical_lemma2} and \ref{lemma_squares}, and it follows that 
\begin{equation}
\label{int4}
 \sumplus_{\chi \pmod q }  \exp \Big(  \frac{k (1+\lambda) \log (qD(1+|t|))}{2 \beta_K \log q} \Big) D_{K,k}(\chi ) S_{K,k}(\chi \psi) \ll q (\log q)^{k^2/2}. 
\end{equation}
Now again using Lemmas \ref{technical_lemma2} and \ref{lemma_squares}, together with Stirling's formula, we have that 
\begin{align*}
&\sumplus_{\chi \pmod q} \sum_{\substack{0 \leq j \leq K-1 \\ j < u \leq K}} \exp \Big( \frac{ k(1+\lambda) \log(qD(1+|t|))}{2\beta_j \log q} \Big) D_{j,k} (\chi \psi) S_{j,k}(\chi  \psi) \Big(\frac{ke^2 \Re P_{I_{j+1}}(\chi \psi ;u)}{\ell_{j+1}}  \Big)^{s_{j+1}}    \\
&\qquad \ll q  \sum_{\substack{0 \leq j \leq K-1 }} (K-j) \exp \Big( \frac{ k(1+\lambda) \log(qD(1+|t|))}{2\beta_j \log q} \Big) \exp \Big( -\frac{1}{4}s_{j+1} \log s_{j+1} +s_{j+1} \log ( k e^{3/2} 2^{1/2})  \\
&\qquad\qquad+ \frac{k^2}{2} \log \log q^{\beta_j} \Big).
\end{align*}
We bound
$$  \frac{k^2}{2} \log \log q^{\beta_j}  \ll \frac{k^2}{2} \log \log q,$$ and similarly to the bounds in \eqref{firstj+1}, we get that
\begin{align*}
\sumplus_{\chi \pmod q} & \sum_{\substack{0 \leq j \leq K-1 \\ j < u \leq K}} \exp \Big( \frac{ k(1+\lambda) \log(qD(1+|t|))}{2\beta_j \log q} \Big) D_{j,k} (\chi \psi) S_{j,k}(\chi  \psi) \Big(\frac{ke^2 \Re P_{I_{j+1}}(\chi \psi ;u)}{\ell_{j+1}}  \Big)^{s_{j+1}}    \\
& \ll q  (\log q)^{k^2/2}.
\end{align*}
Combining this with equations \eqref{bound1_2}, \eqref{bound2_2}, \eqref{bound3_2}, \eqref{o1term} and \eqref{int4} finishes the proof. 
\end{proof}

\section{Proof of Theorem \ref{twistedfirst} - Initial manipulations}
\label{diagonal_section}

\subsection{Auxiliary lemmas}

Recall that $\chi$ and $\chi_j$ are even primitive characters modulo $q$ and $D_j$, respectively. Then $\chi\chi_j$ is an even primitive character modulo $qD_j$. The Dirichlet $L$-function $L(s,\chi\chi_j)$ satisfies a functional equation
\begin{align}\label{fe}
\Lambda(\tfrac12+s,\chi\chi_j)&:=\Big(\frac{qD_j}{\pi}\Big)^{s/2}\Gamma\Big(\frac{1/2+s+\kappa_j}{2}\Big)L(\tfrac12+s,\chi\chi_j)\nonumber\\
&= \epsilon(\chi\chi_j)\Lambda(\tfrac12-s,\overline{\chi\chi_j}).
\end{align}
Notice that $$\epsilon(\chi\chi_j)=\chi(D_j)\chi_j(q)\epsilon(\chi)\epsilon(\chi_j),$$ by the Chinese Remainder Theorem.

For $\left( mn, q \right) = 1$, we have the orthogonality property
    \begin{equation}\label{ortho}
        \sideset{}{^+} \sum_{\substack{\chi (\text{mod}\  q)}} \chi( m) \overline{\chi}( n)
        = \frac{\varphi(q)}{2}\mathds{1}_{m\equiv\pm n(\text{mod}\ q)}-1.
    \end{equation}
    In particular, setting $m=n=1$ we obtain that
    \begin{equation}\label{phiq+}
        \varphi^+(q)=\frac{\varphi(q)}{2}-1
    \end{equation}
    for $q>2$.

Using \eqref{fe} we can derive the following approximate functional equation.

\begin{lem}\label{afe}
Let $G(s)$ be an even entire function of rapid decay in any fixed strip $|\emph{Re}(s)| \leq C$ satisfying $G(0)= 1$.
Let
\[
g(s,t)=\frac{\Gamma\big(\frac{1/2+  it+s}{2}\big)^2\Gamma\big(\frac{1/2- it+s}{2}\big)^2}{\Gamma\big(\frac{1/2+it}{2}\big)^2\Gamma\big(\frac{1/2-it}{2}\big)^2}
\]
and
\begin{equation}\label{formulaV+}
V(x;t)=\frac{1}{2\pi i}\int_{(1)}G(s)g(s,t)x^{-s}\frac{ds}{s}.
\end{equation}
Then we have
\begin{align*}
&L(\tfrac 12+it,\chi\chi_1)L(\tfrac 12+it,\chi\chi_2)L(\tfrac 12-it,\overline{\chi\chi_3})L(\tfrac 12-it,\overline{\chi\chi_4})\\
&\qquad =\sum_{m,n\geq1}\frac{(\chi_1*\chi_2)(m)(\overline{\chi_3}*\overline{\chi_4})(n)\chi(m)\overline{\chi}(n)}{m^{1/2+it}n^{1/2-it}}V\Big(\frac{mn}{\widehat{q}^{\,2}};t\Big)\\
&\qquad \qquad + \Big(\frac{D_1D_2}{D_3D_4}\Big)^{-it} \epsilon\chi(D_1D_2\overline{D_3}\overline{D_4})\sum_{m,n\geq1}\frac{(\overline{\chi_1}*\overline{\chi_2})(m)(\chi_3*\chi_4)(n)\overline{\chi}(m)\chi(n)}{m^{1/2-it}n^{1/2+it}}V\Big(\frac{mn}{\widehat{q}^{\,2}};t\Big),\nonumber
\end{align*}
where $\epsilon$ is given in \eqref{rootnumber}.
\end{lem}
\begin{proof}
This is standard and follows from \cite[Theorem 5.3]{IK}. We provide a proof for completeness.

From Cauchy's residue theorem we have
\begin{align}\label{residue}
&\frac{1}{2\pi i}\int_{(1)}G(s)\frac{\Lambda(\frac12+it+s,\chi\chi_1)\Lambda(\frac12+it+s,\chi\chi_2)\Lambda(\frac12-it+s,\overline{\chi\chi_3})\Lambda(\frac12-it+s,\overline{\chi\chi_4})}{\Gamma(\frac{1/2+it}{2})^2\Gamma(\frac{1/2-it}{2})^2}\frac{ds}{s}\nonumber\\
&\ =\frac{1}{2\pi i}\int_{(-1)}G(s)\frac{\Lambda(\frac12+it+s,\chi\chi_1)\Lambda(\frac12+it+s,\chi\chi_2)\Lambda(\frac12-it+s,\overline{\chi\chi_3})\Lambda(\frac12-it+s,\overline{\chi\chi_4})}{\Gamma(\frac{1/2+it}{2})^2\Gamma(\frac{1/2-it}{2})^2}\frac{ds}{s}\nonumber\\
&\qquad+\text{Res}_{s=0}.
\end{align}
The residue at $s=0$ is
\begin{displaymath}
\Big(\frac{D_1D_2}{D_3D_4}\Big)^{it/2}L(\tfrac 12+it,\chi\chi_1)L(\tfrac 12+it,\chi\chi_2)L(\tfrac 12-it,\overline{\chi\chi_3})L(\tfrac 12-it,\overline{\chi\chi_4}).
\end{displaymath}
By \eqref{fe} and the change of variables $s \mapsto -s$ in \eqref{residue}, we then obtain that
\begin{align*}
&\Big(\frac{D_1D_2}{D_3D_4}\Big)^{it/2}L(\tfrac 12+it,\chi\chi_1)L(\tfrac 12+it,\chi\chi_2)L(\tfrac 12-it,\overline{\chi\chi_3})L(\tfrac 12-it,\overline{\chi\chi_4})\\
&\ =\frac{1}{2\pi i}\int_{(1)}G(s)\frac{\Lambda(\frac12+it+s,\chi\chi_1)\Lambda(\frac12+it+s,\chi\chi_2)\Lambda(\frac12-it+s,\overline{\chi\chi_3})\Lambda(\frac12-it+s,\overline{\chi\chi_4})}{\Gamma(\frac{1/2+it}{2})^2\Gamma(\frac{1/2-it}{2})^2}\frac{ds}{s}\\
&\qquad + \epsilon\chi(D_1D_2\overline{D_3}\overline{D_4})\frac{1}{2\pi i}\int_{(1)}G(s)\frac{\Lambda(\frac12-it+s,\overline{\chi\chi_1})\Lambda(\frac12-it+s,\overline{\chi\chi_2})}{\Gamma(\frac{1/2+it}{2})^2}\nonumber\\
&\qquad\qquad\times\frac{\Lambda(\frac12+it+s,\chi\chi_3)\Lambda(\frac12+it+s,\chi\chi_4)}{\Gamma(\frac{1/2-it}{2})^2}\frac{ds}{s}.
\end{align*}
The lemma now follows by writing the $L$-functions in terms of Dirichlet series and then integrating term-by-term.
\end{proof}

\begin{rem}\label{boundforV}
We have
\[
x^jt^k \frac{\partial^{j+k}V(x;t)}{\partial x^j\partial t^k}\ll_{j,k,C} \Big(1+\frac{x}{(1+|t|)^2}\Big)^{-C}
\]
for any fixed $j,k\geq0$ and $C > 0$.
\end{rem}

\subsection{Separating the diagonal terms and off-diagonal terms}

By Lemma \ref{afe} we  write
\[
I(\ell_1,\ell_2)=I^+(\ell_1,\ell_2)+  \Big(\frac{D_1D_2}{D_3D_4}\Big)^{-it}\epsilon I^-(D_1D_2\ell_1,D_3D_4\ell_2),
\]
where
\begin{align*}
I^+(\ell_1,\ell_2)&= \sum_{\substack{m,n\geq1}}\frac{(\chi_1*\chi_2)(m)(\overline{\chi_3}*\overline{\chi_4})(n)}{m^{1/2+it}n^{1/2-it}}V\Big(\frac{ mn}{\widehat{q}^{\,2}};t\Big)\frac{1}{\varphi^+(q)}\ \sideset{}{^+} \sum_{\chi (\text{mod}\  q)}\chi(\ell_1m)\overline{\chi}(\ell_2n)
\end{align*}
and
\begin{align*}
I^-(\ell_1,\ell_2)&=\sum_{\substack{m,n\geq1}}\frac{(\overline{\chi_1}*\overline{\chi_2})(m)(\chi_3*\chi_4)(n)}{m^{1/2-it}n^{1/2+it}}V\Big(\frac{mn}{\widehat{q}^{\,2}};t\Big)\frac{1}{\varphi^+(q)}\ \sideset{}{^+} \sum_{\chi (\text{mod}\  q)}\chi(\ell_1n)\overline{\chi}(\ell_2m).
\end{align*}

We apply \eqref{ortho} to obtain that
\begin{align*}
I^+(\ell_1,\ell_2)&=\frac{\varphi(q)}{2\varphi^+(q)}\sum_{\substack{\ell_1m\equiv \pm \ell_2n(\text{mod}\ q) \\ (mn,q)=1}}\frac{(\chi_1*\chi_2)(m)(\overline{\chi_3}*\overline{\chi_4})(n)}{m^{1/2+it}n^{1/2-it}}V\Big(\frac{ mn}{\widehat{q}^{\,2}};t\Big)\\
&\qquad-\frac{1}{\varphi^+(q)}\sum_{\substack{(mn,q)=1}}\frac{(\chi_1*\chi_2)(m)(\overline{\chi_3}*\overline{\chi_4})(n)}{m^{1/2+it}n^{1/2-it}}V\Big(\frac{ mn}{\widehat{q}^{\,2}};t\Big).
\end{align*}
Note that by Remark \ref{boundforV}, the condition $(mn,q)=1$ may be omitted with the cost of an error of size $O_\varepsilon(q^{-1+\varepsilon}D(1+|t|))$. Replacing $\varphi^+(q)$ by $\varphi(q)/2$, see \eqref{phiq+}, incurs the same error. 
Also, by \eqref{formulaV+} we have
\begin{align*}
&\sum_{\substack{m,n}}\frac{(\chi_1*\chi_2)(m)(\overline{\chi_3}*\overline{\chi_4})(n)}{m^{1/2+it}n^{1/2-it}}V\Big(\frac{ mn}{\widehat{q}^{\,2}};t\Big)\\
&\ =\frac{1}{2\pi i}\int_{(1)}G(s)g(s,t)\widehat{q}^{\,2s}L(\tfrac12+it+s,\chi_1)L(\tfrac12+it+s,\chi_2)L(\tfrac12-it+s,\overline{\chi_3})L(\tfrac12-it+s,\overline{\chi_4})\frac{ds}{s}.
\end{align*}
Moving the line of integration to Re$(s) =\eps$ and using the subconvexity bound $L(s,\chi_j)\ll_\eps (D_j(1+|s|))^{1/6+\eps}$ \cite[Theorem 1.1]{PY}, we see that 
\begin{equation}\label{subconvexbound}
\sum_{\substack{m,n}}\frac{(\chi_1*\chi_2)(m)(\overline{\chi_3}*\overline{\chi_4})(n)}{m^{1/2+it}n^{1/2-it}}V\Big(\frac{ mn}{\widehat{q}^{\,2}};t\Big)\ll_\eps q^{\eps}D^{2/3}(1+|t|)^{2/3}.
\end{equation}
Thus, 
\begin{align*}
I^+(\ell_1,\ell_2)&=\sum_{\substack{\ell_1m\equiv \pm \ell_2n(\text{mod}\ q) }}\frac{(\chi_1*\chi_2)(m)(\overline{\chi_3}*\overline{\chi_4})(n)}{m^{1/2+it}n^{1/2-it}}V\Big(\frac{ mn}{\widehat{q}^{\,2}};t\Big)+O_\varepsilon\big(q^{-1+\varepsilon}D(1+|t|)\big)\nonumber\\
&=I^{+}_{D}(\ell_1,\ell_2)+I^{+}_{OD}(\ell_1,\ell_2)+O_\varepsilon\big(q^{-1+\varepsilon}D(1+|t|)\big),
\end{align*}
where $I^{+}_{D}(\ell_1,\ell_2),I^{+}_{OD}(\ell_1,\ell_2)$ are the contributions from the diagonal terms $\ell_1m=\ell_2n$ and the off-diagonal terms $\ell_1m\ne\ell_2n$ in the above sum, respectively. 

Similarly, 
\begin{align*}
I^-(\ell_1,\ell_2)=I^{-}_{D}(\ell_1,\ell_2)+I^{-}_{OD}(\ell_1,\ell_2)+O_\varepsilon\big(q^{-1+\varepsilon}D(1+|t|)\big),
\end{align*}
and the sum
\[
I^{+}_{D}(\ell_1,\ell_2)+\Big(\frac{D_1D_2}{D_3D_4}\Big)^{-it}\epsilon I^-_{D}(D_1D_2\ell_1,D_3D_4\ell_2)
\]
gives rise to the first two main terms in Theorem \ref{twistedfirst}.

\subsection{The diagonal terms in terms of Euler products}\label{diagonal1}

Let $\widetilde{\ell_1}=\ell_1/(\ell_1,\ell_2)$ and $\widetilde{\ell_2}=\ell_2/(\ell_1,\ell_2)$. Then we have
\begin{align*}
I^{+}_{D}(\ell_1,\ell_2)&=\sum_{\substack{n\geq 1}}\frac{(\chi_1*\chi_2)(\widetilde{\ell_2} n)(\overline{\chi_3}*\overline{\chi_4})(\widetilde{\ell_1} n)}{(\widetilde{\ell_2} n)^{1/2+it}(\widetilde{\ell_1}n)^{1/2-it}}V\Big(\frac{ \widetilde{\ell_1}\widetilde{\ell_2}n^2}{\widehat{q}^{\,2}};t\Big)\\
&=\frac{1}{\widetilde{\ell_1}^{1/2-it}\widetilde{\ell_2}^{1/2+it}}\frac{1}{2\pi i}\int_{(1)}G(s)g(s,t)\Big(\frac{\widehat{q}^{\,2}}{\widetilde{\ell_1}\widetilde{\ell_2}}\Big)^{s}\sum_{\substack{n\geq1}}\frac{(\chi_1*\chi_2)(\widetilde{\ell_2}n)(\overline{\chi_3}*\overline{\chi_4})(\widetilde{\ell_1}n)}{n^{1+2s}}\frac{ds}{s},
\end{align*}
by \eqref{formulaV+}.
We have
\[
\sum_{\substack{n}}\frac{(\chi_1*\chi_2)(\widetilde{\ell_2}n)(\overline{\chi_3}*\overline{\chi_4})(\widetilde{\ell_1}n)}{n^{s}}=F(\widetilde{\ell_1},\widetilde{\ell_2};s)\prod_p\bigg(1+\sum_{j\geq1}\frac{(\chi_1*\chi_2)(p^j)(\overline{\chi_3}*\overline{\chi_4})(p^j)}{p^{js}}\bigg),
\]
where
\begin{align*}
F(\widetilde{\ell_1},\widetilde{\ell_2};s)&=\prod_{\substack{p^{\lambda_1}||\widetilde{\ell_1}\\p^{\lambda_2}||\widetilde{\ell_2}}}\bigg((\chi_1*\chi_2)(p^{\lambda_2})(\overline{\chi_3}*\overline{\chi_4})(p^{\lambda_1})+\sum_{j\geq1}\frac{(\chi_1*\chi_2)(p^{j+\lambda_2})(\overline{\chi_3}*\overline{\chi_4})(p^{j+\lambda_1})}{p^{js}}\bigg)\\
&\qquad\qquad\times\bigg(1+\sum_{j\geq1}\frac{(\chi_1*\chi_2)(p^j)(\overline{\chi_3}*\overline{\chi_4})(p^j)}{p^{js}}\bigg)^{-1}.
\end{align*}
Also,
\begin{align*}
\prod_p\bigg(1+\sum_{j\geq1}\frac{(\chi_1*\chi_2)(p^j)(\overline{\chi_3}*\overline{\chi_4})(p^j)}{p^{js}}\bigg)&=H(s)L(s,\chi_1\overline{\chi_3})L(s,\chi_1\overline{\chi_4})L(s,\chi_2\overline{\chi_3})L(s,\chi_2\overline{\chi_4}),
\end{align*}
where
\begin{align*}
H(s)&=\prod_p\bigg(1+\sum_{j\geq1}\frac{(\chi_1*\chi_2)(p^j)(\overline{\chi_3}*\overline{\chi_4})(p^j)}{p^{js}}\bigg)\\
&\qquad\qquad\times\bigg(1-\frac{\chi_1\overline{\chi_3}(p)}{p^s}\bigg)\bigg(1-\frac{\chi_1\overline{\chi_4}(p)}{p^s}\bigg)\bigg(1-\frac{\chi_2\overline{\chi_3}(p)}{p^s}\bigg)\bigg(1-\frac{\chi_2\overline{\chi_4}(p)}{p^s}\bigg).
\end{align*}
Hence
\begin{align*}
I^{+}_{D}(\ell_1,\ell_2)&=\frac{1}{\widetilde{\ell_1}^{1/2-it}\widetilde{\ell_2}^{1/2+it}}\frac{1}{2\pi i}\int_{(1)}G(s)g(s,t)\Big(\frac{\widehat{q}^{\,2}}{\widetilde{\ell_1}\widetilde{\ell_2}}\Big)^{s}F(\widetilde{\ell_1},\widetilde{\ell_2};1+2s)H(1+2s) \nonumber \\
&\qquad\times L(1+2s,\chi_1\overline{\chi_3})L(1+2s,\chi_1\overline{\chi_4})L(1+2s,\chi_2\overline{\chi_3})L(1+2s,\chi_2\overline{\chi_4})\frac{ds}{s}. 
\end{align*}
We move the line of integration to Re$(s) =-1/4 + \varepsilon$, crossing a simple pole at $s = 0$. 
Estimating the
new integral trivially gives
\begin{align}
\label{diagonal_mt}
I^{+}_{D}(\ell_1,\ell_2)&=\frac{1}{\widetilde{\ell_1}^{1/2-it}\widetilde{\ell_2}^{1/2+it}}F(\widetilde{\ell_1},\widetilde{\ell_2};1)H(1) L(1,\chi_1\overline{\chi_3})L(1,\chi_1\overline{\chi_4})L(1,\chi_2\overline{\chi_3})L(1,\chi_2\overline{\chi_4})  \\
&\qquad+O_\varepsilon\big(q^{-1/2+\varepsilon}D^{5/6}\big),\nonumber
\end{align}
and we obtain \eqref{MAiden}.

\section{Proof of Theorem \ref{twistedfirst} - The off-diagonal terms}\label{offdiagonal1}

Recall that
\begin{align*}
I^{+}_{OD}(\ell_1,\ell_2)&=\sum_{\substack{\ell_1m\equiv \pm \ell_2n(\text{mod}\ q) \\\ell_1m\ne \ell_2n}}\frac{(\chi_1*\chi_2)(m)(\overline{\chi_3}*\overline{\chi_4})(n)}{m^{1/2+it}n^{1/2-it}}V\Big(\frac{ mn}{\widehat{q}^{\,2}};t\Big).
\end{align*}
We apply a dyadic partition of unity to the sums over $m$ and $n$ so that
\[
I^{+}_{OD}(\ell_1,\ell_2)=\sum_{M,N}I^{+}_{OD}(\ell_1,\ell_2,M,N),
\]
where
\begin{align}\label{firstIMN}
&I^{+}_{OD}(\ell_1,\ell_2,M,N)\\
&\qquad=\frac{1}{M^{1/2+it}N^{1/2-it}}\sum_{\substack{\ell_1m\equiv\pm \ell_2n(\text{mod}\ q)\\\ell_1m\ne \ell_2n}}(\chi_1*\chi_2)(m)(\overline{\chi_3}*\overline{\chi_4})(n)\omega^+\Big ( \frac{m}{M} \Big )\omega^- \Big ( \frac{n}{N} \Big )V\Big(\frac{ mn}{\widehat{q}^{\,2}};t\Big).\nonumber
\end{align}

We shall think of $D\asymp q^\delta$, $1+|t|\asymp q^\tau$, $L=q^\kappa$, $M\asymp q^\mu$, $N\asymp q^\nu$. 
Let
\begin{equation}\label{choiceeta1}
\eta=\min\left\{\frac{1}{28}\Big(\frac{11}{16}-80\delta-96\kappa-10\tau\Big),\frac{1}{80}\Big(\frac{25}{32}-88\delta-50\kappa-\tau\Big)\right\},
\end{equation}
so that
\begin{align}
     80\delta+96\kappa+10\tau+28\eta&\leq \frac{11}{16},\label{choiceeta}\\
     88\delta+50\kappa+\tau+80\eta&\leq \frac{25}{32}.\label{choiceeta2}
\end{align}

Due to Remark \ref{boundforV}, we may assume that $MN\ll \widehat{q}^{\,2+\varepsilon}(1+|t|)^2$ at the cost of a negligible error term. Also, because of the trivial bound
\[
I^{+}_{OD}(\ell_1,\ell_2,M,N)\ll_\eps q^{-1+\eps}L\sqrt{MN},
\]
we have $I^{+}_{OD}(\ell_1,\ell_2,M,N)\ll_\eps q^{-\eta+\eps}$ provided that
\[
MN\ll q^{2-2\kappa-2\eta}.
\]
So we are left to consider the pairs $(M,N)$ in
\begin{equation}\label{trivialcond}
\mathcal{S}:=\big\{(M,N):2-2\kappa-2\eta<\mu+\nu< 2+2\delta+2\tau+\eps\big\}.
\end{equation}
Within this we shall break the terms into the balanced case and the unbalanced case. Let 
\[
\mathcal{S}_1:=\big\{(M,N)\in\mathcal{S}:|\mu-\nu|\leq 1-2\theta-10\delta-16\kappa-\tau-2\eta \big\},
\]
\[
\mathcal{S}_2^>:=\big\{(M,N)\in\mathcal{S}:\mu-\nu> 1-2\theta-10\delta-16\kappa-\tau-2\eta \big\}
\]
and
\[
\mathcal{S}_2^<:=\big\{(M,N)\in\mathcal{S}:\nu-\mu> 1-2\theta-10\delta-16\kappa-\tau-2\eta \big\}.
\]
We shall consider the balanced case $(M,N)\in\mathcal{S}_1$ in Subsection \ref{closesection} and the unbalanced case $(M,N)\in\mathcal{S}_2^{>}\cup \mathcal{S}_2^{<}$ in Subsection \ref{notclosesection}.

\subsection{$M$, $N$ are relatively close: $(M,N)\in\mathcal{S}_1$}\label{closesection}


We have
\begin{align}\label{555}
I^{+}_{OD}(\ell_1,\ell_2,M,N)&=\frac{1}{M^{1/2+it}N^{1/2-it}}\sum_{0<|r|\leq R}\\
&\qquad\times\sum_{\substack{\widetilde{\ell_1}m\pm \widetilde{\ell_2}n=qr}}(\chi_1*\chi_2)(m)(\overline{\chi_3}*\overline{\chi_4})(n)\omega^+\Big ( \frac{m}{M} \Big )\omega^- \Big ( \frac{n}{N} \Big )V\Big(\frac{ mn}{\widehat{q}^{\,2}};t\Big)\nonumber
\end{align}
with
\[
R=\frac{2L(M+N)}{q}.
\]
We shall use the $\delta$-method of Duke, Friedlander and Iwaniec in \cite{DFI}. Let 
$$Q=L(M+N)^{1/2+\eps}\qquad\text{and}\qquad U=L(M+N)$$
as in \cite[Lemma  4.3]{Z1}. We have
\begin{equation}\label{deltafnc}
\mathds{1}_{n=0} = \sum_{\ell\geq 1}\, \sideset{}{^*}\sum_{k(\text{mod}\ \ell)} e\Big(\frac{kn}{\ell}\Big) \Delta_\ell(n),
\end{equation}
where $\Delta_\ell(u)$ is some smooth function that vanishes unless $|u| \leq U$ and $\ell\leq 2Q$ (see \cite[Section 4]{DFI}), and in this range (see \cite[Lemma 2]{DFI})
\begin{equation}\label{bdDelta}
\Delta_\ell(u)\ll \min\Big\{\frac{1}{\ell Q},\frac{1}{u}\Big\}.
\end{equation}
We also attach to both sides of \eqref{deltafnc} a redundant factor $\varphi(n)$, where $\varphi(u)$ is a smooth function supported on $|u|\leq U$ satisfying $\varphi(0)=1$ and $\varphi^{(j)}(u)\ll_j U^{-j}$. Applying this to \eqref{555} we see that
\begin{align}\label{summn}
I^{+}_{OD}(\ell_1,\ell_2,M,N)& =\frac{1}{M^{1/2+it}N^{1/2-it}}\sum_{0<|r|\leq R}\sum_{\ell\leq 2Q}\, \sideset{}{^*}\sum_{k(\text{mod}\ \ell)}e\Big(-\frac{qrk}{\ell}\Big)\nonumber\\
&\quad\times\sum_{m,n}(\chi_1*\chi_2)(m)(\overline{\chi_3}*\overline{\chi_4})(n)e\Big(\frac{k(\widetilde{\ell_1} m \pm \widetilde{\ell_2}n)}{\ell}\Big)g_{\ell_1,\ell_2}^{\pm}(m,n,\ell),
\end{align}
where
\[
g_{\ell_1,\ell_2}^{\pm}(m,n,\ell)=\Delta_\ell(\widetilde{\ell_1} m \pm \widetilde{\ell_2}n-qr)\varphi(\widetilde{\ell_1} m \pm \widetilde{\ell_2}n-qr)\omega^+\Big ( \frac{m}{M} \Big )\omega^- \Big ( \frac{n}{N} \Big )V\Big(\frac{mn}{\widehat{q}^{\,2}};t\Big),
\]
and where in the sum over $m,n$ we sum $g_{\ell_1,\ell_2}^{+}$ and $g_{\ell_1,\ell_2}^{-}$ with the corresponding $\pm$ in the exponential. 


Let $d_1=(\ell,\widetilde{\ell_1})$ and $d_2=(\ell,\widetilde{\ell_2})$. Also, let $\ell_1' =\widetilde{\ell_1}/d_1$, $\ell_1'' = \ell/d_1$, $\ell_2 ' = \widetilde{\ell_2}/d_2$ and $\ell_2'' = \ell/d_2$.  We apply the Voronoi summation
formula in Theorem \ref{Voronoithm} to the sums over $m,n$ in \eqref{summn}. In doing so, we can write $I^{+}_{OD}(\ell_1,\ell_2,M,N)$ as the sum of 8 off-diagonal main terms and 24 error terms, 
\begin{align}\label{expandIOD}
I^{+}_{OD}(\ell_1,\ell_2,M,N)&=\mathcal{M}_{1,3}^+(M,N)+\mathcal{M}_{1,4}^+(M,N)+\mathcal{M}_{2,3}^+(M,N)+\mathcal{M}_{2,4}^+(M,N)+\sum_{j=1}^{24}\mathcal{E}_{j,\ell_1,\ell_2}^+,
\end{align}
where
\[
\mathcal{M}_{1,3}^+(M,N)=\mathcal{M}_{+;1,3}^+(M,N)+\mathcal{M}_{-;1,3}^+(M,N)
\]
with
\begin{align}\label{2002}
&\mathcal{M}_{\pm;1,3}^+(M,N)= \frac{\sqrt{D_1D_3}\epsilon(\chi_1)\epsilon(\overline{\chi_3})\overline{\chi_2}(D_1)\chi_4(D_3)}{M^{1/2+it}N^{1/2-it}}L(1,\overline{\chi_1}\chi_2)L(1,\chi_3\overline{\chi_4})\\
&\ \times \sum_{0<|r|\leq R}\sum_{\substack{\ell\leq 2Q\\D_1|\ell_1''\\D_3|\ell_2''}}\frac{\overline{\chi_1}(\ell_1')\chi_3(\ell_2')\chi_2(\ell_1'')\overline{\chi_4}(\ell_2'')}{\ell_1''\ell_2''}\, \sideset{}{^*}\sum_{k(\text{mod}\ \ell)}\overline{\chi_1}\chi_3(k)e\Big(\frac{-qrk}{\ell}\Big)\int_{0}^{\infty}\int_{0}^{\infty} g_{\ell_1,\ell_2}^{\pm}(x,y,\ell)dxdy,\nonumber
\end{align}
and similar expressions hold for the other off-diagonal main terms. We shall postpone the evaluation of these off-diagonal main terms until Section \ref{odmsection}. In the remaining of this subsection, we shall bound the error terms by this proposition.


\begin{prop}\label{boundforbalancedcase}
	We have
	\begin{align*}
		\sum_{j=1}^{24}\mathcal{E}_{j,\ell_1,\ell_2}^+&\ll_\eps q^{-\eta+\eps}.
	\end{align*}
\end{prop}
\begin{proof}
In fact we shall assume without loss of generality that $M\geq N$ and show that
\begin{equation}\label{errorwhattoprove}
\sum_{j=1}^{24}\mathcal{E}_{j,\ell_1,\ell_2}^+\ll_\eps q^{-1+\varepsilon}DLM^{1/2}N^{1/4}+q^{-1/2+\theta+\eps}D^5L^8 \sqrt{\frac{M}{N}}+q^{-1+\theta+\eps}D^4L^{11/2}\sqrt{M}.
\end{equation}

We only focus on the first error term as the other error terms can be treated similarly. Let
\[
\begin{cases}
D_j=D_j'(\ell_1'',D_j) & \text{for }j=1,2,\\
D_j=D_j'(\ell_2'',D_j) & \text{for }j=3,4.
\end{cases}
\]
By the Chinese Remainder Theorem we may factor $\chi_j=\chi_j'\chi_j''$, where $\chi_j'$ is a primitive character modulo $D_j'$ and $\chi_j''$ is a primitive character modulo $(\ell_1'',D_j)$ for $j=1,2$, or modulo $(\ell_2'',D_j)$ for $j=3,4$. Then by Theorem \ref{Voronoithm}, 
\begin{align*}
\mathcal{E}_{1;\ell_1,\ell_2}^+&=\frac{4\pi^2}{M^{1/2+it}N^{1/2-it}}\sum_{0<|r|\leq R}\sum_{\ell\leq 2Q}\frac{\epsilon(\chi_1')\epsilon(\chi_2')\epsilon(\overline{\chi_3'})\epsilon(\overline{\chi_4'})}{\ell_1''\ell_2''\sqrt{D'}}\\
&\quad\times \chi_1'\chi_2'(\ell_1'')\overline{\chi_3'\chi_4'}(\ell_2'')\chi_1''(-\overline{\ell_1'D_2'})\chi_2''(-\overline{\ell_1'D_1'})\chi_3''(\ell_2'D_4')\chi_4''(\ell_2'D_3')\, \sideset{}{^*}\sum_{k(\text{mod}\ \ell)}\overline{\chi_1''\chi_2''}\chi_3''\chi_4''(k)e\Big(-\frac{qrk}{\ell}\Big)\\
&\quad\times\sum_{m',n'}\big(\overline{\chi_1'}\chi_2''*\chi_1''\overline{\chi_2'}\big)(m')\big(\chi_3'\overline{\chi_4''}*\overline{\chi_3''}\chi_4'\big)(n')e\Big(-\frac{\overline{k\ell_1'D_1'D_2'}m'}{\ell_1''}+\frac{\overline{k\ell_2'D_3'D_4'}n'}{\ell_2''}\Big)G_{\ell_1,\ell_2}(m',n',\ell),
\end{align*}
where we have defined
\[
D'=\prod_{j=1}^4D_j'
\]
and
\[
G_{\ell_1,\ell_2}(m',n',\ell)=\int_{0}^{\infty}\int_{0}^{\infty} g_{\ell_1,\ell_2}(x,y,\ell)Y_0\Big(\frac{4\pi\sqrt{m'x}}{\ell_1''\sqrt{D_1'D_2'}}\Big)Y_0\Big(\frac{4\pi\sqrt{n'y}}{\ell_2''\sqrt{D_3'D_4'}}\Big)dxdy.
\]
We apply a partition of unity to the interval $[1,2Q]$ to write $\mathcal{E}_{1;\ell_1,\ell_2}^+$ as $O(\log q)$ sums of the shape
\begin{align}\label{firsterror}
&\frac{1}{\mathcal{Q}M^{1/2+it}N^{1/2-it}}\sum_{0<|r|\leq R}\sum_{\substack{\widetilde{\ell_1}=\ell_1'd_1\\\widetilde{\ell_2}=\ell_2'd_2}}\sum_{(c,\ell_1'\ell_2')=1}\frac{\epsilon(\chi_1')\epsilon(\chi_2')\epsilon(\overline{\chi_3'})\epsilon(\overline{\chi_4'})}{c\sqrt{D'}}\nonumber\\
&\qquad\times\chi_1'\chi_2'(cd_2)\overline{\chi_3'\chi_4'}(cd_1)\chi_1''(-\overline{\ell_1'D_2'})\chi_2''(-\overline{\ell_1'D_1'})\chi_3''(\ell_2'D_4')\chi_4''(\ell_2'D_3')\vartheta\Big(\frac{cd_1d_2}{\mathcal{Q}}\Big) \nonumber\\
&\qquad \times\sum_{m',n'}\big(\overline{\chi_1'}\chi_2''*\chi_1''\overline{\chi_2'}\big)(m')\big(\chi_3'\overline{\chi_4''}*\overline{\chi_3''}\chi_4'\big)(n')G_{\ell_1,\ell_2}(m',n',cd_1d_2)\\
&\qquad\times \, \sideset{}{^*}\sum_{k \pmod {cd_1d_2}}\overline{\chi_1''\chi_2''}\chi_3''\chi_4''(k)e\Big(-\frac{qrk}{cd_1d_2}-\frac{d_1\overline{k\ell_1'D_1'D_2'}m'}{cd_1d_2}+\frac{d_2\overline{k\ell_2'D_3'D_4'}n'}{cd_1d_2}\Big),\nonumber
\end{align}
where $\mathcal{Q}\leq Q$ and $\vartheta$ is a smooth and compactly supported function on $\mathbb{R}_{>0}$ such that $\vartheta^{(j)}\ll_j 1$. Here we have written $\ell=cd_1 d_2$. 

We also wish to remove the dependence of $\chi_j'$ and $\chi_j''$ on $c$. Recall that
\[
\begin{cases}
D_j=D_j'(cd_2,D_j) & \text{for }j=1,2,\\
D_j=D_j'(cd_1,D_j) & \text{for }j=3,4,
\end{cases}
\]
and that $\chi_j=\chi_j'\chi_j''$, where $\chi_j'$ is a primitive character modulo $D_j'$ and $\chi_j''$ is a primitive character modulo $(cd_2,D_j)$ for $j=1,2$, or modulo $(cd_1,D_j)$ for $j=3,4$. Since $(d_2,D_j)|(cd_2,D_j)$ for $j=1,2$, and $(d_1,D_j)|(cd_1,D_j)$ for $j=3,4$, we let 
\begin{equation*}
\begin{cases}
(cd_2,D_j)=D_j''(d_2,D_j) & \text{for }j=1,2,\\
(cd_1,D_j)=D_j''(d_1,D_j) & \text{for }j=3,4.
\end{cases}
\end{equation*}
So
\[
\begin{cases}
D_j=D_j'D_j''(d_2,D_j) & \text{for }j=1,2,\\
D_j=D_j'D_j''(d_1,D_j) & \text{for }j=3,4,
\end{cases}
\]
and hence $D_j''=(c,D_j'D_j'')$. Thus,
\[
D''|c\qquad\text{and}\qquad (c,D')=1,
\]
where
$D''=\prod_{j=1}^{4}D_j''$. Also, $\chi_j'$ is a primitive character modulo $D_j'$ and $\chi_j''$ is a primitive character modulo $D_j''(d_2,D_j)$ for $j=1,2$, or modulo $D_j''(d_1,D_j)$ for $j=3,4$.

We change the variable $c\longrightarrow cD''$, where $(c,D')=1$. The sum over $k$ in \eqref{firsterror} is
\[
S_{\overline{\chi_1''\chi_2''}\chi_3''\chi_4''}(qr, d_1\overline{\ell_1'D_1'D_2'}m'- d_2\overline{\ell_2'D_3'D_4'}n';cd_1d_2D'').
\]
We remark that the inverses in the Kloosterman sum are with respect to different moduli. For instance, $\overline{\ell_1'D_1'D_2'}$ (resp. $\overline{\ell_2'D_3'D_4'}$) is the inverse of $\ell_1'D_1'D_2'$ (resp. $\ell_2'D_3'D_4'$) modulo $cd_2D''$ (resp. $cd_1D''$). To solve this problem, we write in a unique way $d_j=d_j^*d_j'$, $j=1,2$, with 
$$(d_1^*,\ell_1'D_1'D_2')=1,\quad d_1'|(\ell_1'D_1'D_2')^\infty,\quad (d_2^*,\ell_2'D_3'D_4')=1,\quad d_2'|(\ell_2'D_3'D_4')^\infty.$$
In doing so we get $(cd_1^*d_2^*D'',v)=1$, where $$v=d_1'd_2'.$$ 

Notice that $\overline{\chi_1''\chi_2''}\chi_3''\chi_4''$ is a character modulo $D''(d_1,D_3D_4)(d_2,D_1D_2)$, which is equal to
\[
D''(d_1^*,D_3D_4)(d_2^*,D_1D_2)(d_1',D_3D_4)(d_2',D_1D_2).
\]
By the Chinese Remainder Theorem, we can factor $\overline{\chi_1''\chi_2''}\chi_3''\chi_4''=\phi_1\phi_2$, where $\phi_1$ is a character modulo $D''(d_1^*,D_3D_4)(d_2^*,D_1D_2)$ and $\phi_2$ is a character modulo $(d_1',D_3D_4)(d_2',D_1D_2)$. We can now use the following result.

\begin{lem}\label{multK}
Let $(c,d)=1$ and let $\phi_1,\phi_2$ be characters modulo $c,d$, respectively. Then
\[
S_{\phi_1\phi_2}(a,b;cd)=S_{\phi_1}(a,b\overline{d^2};c)S_{\phi_2}(a,b\overline{c^2};d).
\]
\end{lem}
\begin{proof}
By the Chinese Remainder Theorem, every $(x,cd)=1$ can be written uniquely as $x=d\overline{d}u+c\overline{c}v$ with $(u,c)=1$ and $(v,d)=1$. Hence
\begin{align*}
    S_{\phi_1\phi_2}(a,b;cd)&=\sideset{}{^*}\sum_{x(\text{mod}\
cd)}\phi_1\phi_2(x)e\Big(\frac{ax+b\overline{x}}{cd}\Big)=\sideset{}{^*}\sum_{u(\text{mod}\
c)}\ \sideset{}{^*}\sum_{v(\text{mod}\
d)}\phi_1(u)\phi_2(v)e\Big(\frac{a\overline{d}u+b\overline{du}}{c}+\frac{a\overline{c}v+b\overline{cv}}{d}\Big)\\
&=S_{\phi_1}(a,b\overline{d^2};c)S_{\phi_2}(a,b\overline{c^2};d),
\end{align*}
as claimed.
\end{proof}

By Lemma \ref{multK} we obtain that
\begin{align*}
&S_{\overline{\chi_1''\chi_2''}\chi_3''\chi_4''}(qr, d_1\overline{\ell_1'D_1'D_2'}m'- d_2\overline{\ell_2'D_3'D_4'}n';cd_1d_2D'')\\
& =S_{\phi_1}\big(qr, \overline{v^2}(d_1\overline{\ell_1'D_1'D_2'}m'- d_2\overline{\ell_2'D_3'D_4'}n');cd_1^*d_2^*D''\big)S_{\phi_2}\big(qr, \overline{(cd_1^*d_2^*D'')^2}(d_1\overline{\ell_1'D_1'D_2'}m'- d_2\overline{\ell_2'D_3'D_4'}n');v\big).
\end{align*}
Notice that $\ell_1',\ell_2'$ and $D'$ are all co-prime with $cd_1^*d_2^*D''$. So we can take the inverses to be with respect to $cd_1^*d_2^*D''$, and hence the first Kloosterman sum on the right hand side is
\begin{align*}
S_{\phi_1}\big(qr, \overline{\ell_1'\ell_2'v^2D'}(d_1\ell_2'D_3'D_4'm'- d_2\ell_1'D_1'D_2'n');cd_1^*d_2^*D''\big).
\end{align*}
 For the second Kloosterman sum, we separate the dependence in $c$ by writing 
\[
S_{\phi_2}(qr, \overline{(cd_1^*d_2^*D'')^2}(d_1\overline{\ell_1'D_1'D_2'}m'- d_2\overline{\ell_2'D_3'D_4'}n');v)=\frac{1}{\varphi(v)}\sum_{\psi \pmod v}\psi(cd_1^*d_2^*D'')\widehat{S}_{\psi}(qr,m',n',\ell_1,\ell_2),
\]
where
\[
\widehat{S}_{\psi}(qr,m',n',\ell_1,\ell_2)=\, \sideset{}{^*}\sum_{a \pmod v}\overline{\psi}(a)S_{\phi_2}(qr\overline{a},(d_1\overline{\ell_1'D_1'D_2'}m'- d_2\overline{\ell_2'D_3'D_4'}n')\overline{a};v).
\]

It hence follows that \eqref{firsterror} is equal to
\begin{align}\label{mainerrorformula}
&\frac{1}{\mathcal{Q}M^{1/2+it}N^{1/2-it}}\sum_{\substack{\widetilde{\ell_1}=\ell_1'd_1\\\widetilde{\ell_2}=\ell_2'd_2}}\sum_{\substack{j=1\\D_j=D_j'D_j''(d_2,D_j)}}^{2}\sum_{\substack{j=3\\D_j=D_j'D_j''(d_1,D_j)}}^{4} \frac{1}{\varphi(v)}\sum_{\psi \pmod v}\psi(d_1^*d_2^*D'')\\
&\ \times\frac{\epsilon(\chi_1')\epsilon(\chi_2')\epsilon(\overline{\chi_3'})\epsilon(\overline{\chi_4'})}{\sqrt{D'}D''}\chi_1'\chi_2'(d_2D'')\overline{\chi_3'\chi_4'}(d_1D'')\chi_1''(-\overline{\ell_1'D_2'})\chi_2''(-\overline{\ell_1'D_1'})\chi_3''(\ell_2'D_4')\chi_4''(\ell_2'D_3')\mathcal{D}(r,m',n';\psi),\nonumber
\end{align}
where
\begin{align*}
&\mathcal{D}(r,m',n';\psi)=\sum_{0<|r|\leq R}\sum_{m',n'}\big(\overline{\chi_1'}\chi_2''*\chi_1''\overline{\chi_2'}\big)(m')\big(\chi_3'\overline{\chi_4''}*\overline{\chi_3''}\chi_4'\big)(n')\widehat{S}_{\psi}(qr,m',n',\ell_1,\ell_2)\\
	&\quad \times \sum_{(c,\ell_1'\ell_2'D')=1}\frac{\psi\chi_1'\chi_2'\overline{\chi_3'\chi_4'}(c)}{c}S_{\phi_1}\big(qr, \overline{\ell_1'\ell_2'v^2D'}(d_1\ell_2'D_3'D_4'm'- d_2\ell_1'D_1'D_2'n');cd_1^*d_2^*D''\big)\\
&\quad\times\vartheta\Big(\frac{cd_1d_2D''}{\mathcal{Q}}\Big)G_{\ell_1,\ell_2}(m',n',cd_1d_2D'').
\end{align*}
We shall carefully bound the sum over $c$ and use the fact that (see \cite[Eq. (3.6)]{T})
\begin{equation}\label{Topa}
\widehat{S}_{\psi}(qr,m',n',\ell_1,\ell_2)\ll_\varepsilon q^\varepsilon \Big(r,\frac{v}{\text{cond}(\psi)}\Big)v.
\end{equation}

We now apply a dyadic partition of unity to the sums over $r,m',n'$ and $b=d_1\ell_2'D_3'D_4'm'-d_2\ell_1'D_1'D_2'n'$. This reduces the expression to $O((\log q)^4)$ sums of the form
\begin{align*}
&\mathcal{D}(R,B,M',N';q,\psi)\\
&\qquad=\sum_{|r|\asymp R}\sum_{|b|\asymp B}\sum_{\substack{d_1\ell_2'D_3'D_4'm'- d_2\ell_1'D_1'D_2'n'=b\\m'\asymp M',n'\asymp N'}}\big(\overline{\chi_1'}\chi_2''*\chi_1''\overline{\chi_2'}\big)(m')\big(\chi_3'\overline{\chi_4''}*\overline{\chi_3''}\chi_4'\big)(n')\widehat{S}_{\psi}(qr,m',n',\ell_1,\ell_2)\nonumber\\
&\quad\quad\times \sum_{(c,\ell_1'\ell_2'D')=1}\frac{\psi\chi_1'\chi_2'\overline{\chi_3'\chi_4'}(c)}{c}S_{\phi_1}(qr, \overline{\ell_1'\ell_2'v^2D'}b;cd_1^*d_2^*D'')\vartheta\Big(\frac{cd_1d_2D''}{\mathcal{Q}}\Big)\mathcal{J}(cd_1d_2D'',r,b,m',n')\\
&\qquad:=\mathcal{D}^0+\mathcal{D}^++\mathcal{D}^-,
\end{align*}
where $\mathcal{D}^0$ (resp. $\mathcal{D}^+,\mathcal{D}^-$) denotes the contribution from $b=0$ (resp. $b>0$, $b<0$), and $\mathcal{J}(z,r,b,m',n')=G_{\ell_1,\ell_2}(m',n',z)H(|r|,|b|,m',n')$ with $H$ being a smooth  and compactly supported function on $[R/2,R]\times[B/2,B]\times[M'/2,M']\times[N'/2,N']$ satisfying
\[
H^{(i,j,k,l)}\ll_{i,j,k,l} R^{-i}B^{-j}M'^{-k}N'^{-l}.
\]
Notice that using the recurrence formula $(x^\nu Y_\nu(x))'=x^\nu Y_{\nu-1}(x)$ 
and integration by parts in \eqref{firsterror}, we see  that $\mathcal{D}$ is negligible unless 
\begin{equation}\label{rangemn}
M'\leq q^{\varepsilon}\frac{D_1'D_2'\mathcal{Q}^{2}}{M}\qquad\text{and}\qquad N'\leq q^{\varepsilon}\frac{D_3'D_4'\mathcal{Q}^{2}}{N}.
\end{equation}
Also, the size of $B$ depends on the sign of $b$. If $b>0$, then
\[
B=B^+\leq D_3'D_4'L^2M',
\]
and if $b<0$, then
\[
B=B^-\leq D_1'D_2'L^2N'.
\]

The term $\mathcal{D}^0$ is easier. Using the bound $Y_0(x)\ll x^{-1/2}$ and \eqref{bdDelta} we see that
\begin{align}\label{boundG1}
G_{\ell_1,\ell_2}(m',n',cd_1d_2D'')&\ll \frac{D'^{1/4}cD''\sqrt{d_1d_2}}{(MNM'N')^{1/4}}\int_{0}^{\infty}\int_{0}^{\infty} |g_{\ell_1,\ell_2}(x,y,cd_1d_2D'')|dxdy\nonumber\\
&\ll\frac{D'^{1/4}(MN)^{3/4}}{\sqrt{d_1d_2}(M'N')^{1/4}Q}.
\end{align}
The Weil bound for the hybrid Gauss Kloosterman sum (see, for example, \cite[Lemma 3]{BC}), \eqref{Topa} and \eqref{rangemn} then imply that
\begin{align*}
\mathcal{D}^0&\ll_\varepsilon q^{\varepsilon}\frac{D'^{1/4}(MNM')^{3/4}}{N'^{1/4}Q}\sum_{|r|\asymp R}\sum_{c\asymp \mathcal{Q}/d_1d_2D''}\frac{(r,v)(r,cd_1^*d_2^*D'')^{1/2}v\sqrt{d_1^*d_2^*D''}}{\sqrt{cd_1d_2}}\nonumber\\
&\ll_\varepsilon q^{-1+\varepsilon}D'^{1/4}L\cdot\frac{M(MNM')^{3/4}\mathcal{Q}^{1/2}}{N'^{1/4}Q}\\
&\ll_\varepsilon q^{-1+\varepsilon}D_1'D_2'(D_3'D_4')^{1/4}L MN^{3/4}\mathcal{Q}.
\end{align*}
Thus the contribution of $\mathcal{D}^0$ to $\mathcal{E}_{1;\ell_1,\ell_2}^+$ is
\begin{align}
\label{boundD0}
&\ll_\varepsilon q^{-1+\varepsilon}DL M^{1/2}N^{1/4},
\end{align}
giving the first term in \eqref{errorwhattoprove}.

We now consider $\mathcal{D}^-$. We also assume $r>0$. The treatment of the case $r<0$ and $\mathcal{D}^+$ is similar. We have
\begin{align*}
&\mathcal{D}^-=\sum_{r\asymp R}\sum_{|b|\asymp B}\sum_{\substack{d_1\ell_2'D_3'D_4'm'- d_2\ell_1'D_1'D_2'n'=b\\m'\asymp M',n'\asymp N'}}\big(\overline{\chi_1'}\chi_2''*\chi_1''\overline{\chi_2'}\big)(m')\big(\chi_3'\overline{\chi_4''}*\overline{\chi_3''}\chi_4'\big)(n')\widehat{S}_{\psi}(qr,m',n',\ell_1,\ell_2)\nonumber\\
&\quad\qquad\times \sum_{(c,\ell_1'\ell_2'D')=1}\frac{\psi\chi_1'\chi_2'\overline{\chi_3'\chi_4'}(c)}{c}S_{\phi_1}(qr, \overline{\ell_1'\ell_2'v^2D'}b;cd_1^*d_2^*D'')\Phi\Big(\frac{4\pi \sqrt{qr|b|}}{cd_1d_2D''\sqrt{\ell_1'\ell_2'D'}},r,b,m',n'\Big),
\end{align*}
where
\begin{align*}
&\Phi(z,r,b,m',n')=\vartheta\Big(\frac{4\pi \sqrt{qr|b|}}{z\mathcal{Q}\sqrt{\ell_1'\ell_2'D'}}\Big)H(|r|,|b|,m',n')\\
&\qquad\times \int_{0}^{\infty}\int_{0}^{\infty} g_{\ell_1,\ell_2}\Big(x,y,\frac{4\pi \sqrt{qr|b|}}{z\sqrt{\ell_1'\ell_2'D'}}\Big)Y_0\bigg(zd_1\sqrt{\frac{\ell_1'\ell_2'D_3'D_4'm'x}{qr|b|}}\bigg)Y_0\bigg(zd_2\sqrt{\frac{\ell_1'\ell_2'D_1'D_2'n'y}{qr|b|}}\bigg)dxdy.
\end{align*}
The function $\Phi$ satisfies the bound (see \eqref{boundG1} and \cite[Proposition 4.6]{Z1})
\[
\Phi^{(j_1,\ldots,j_5)}\ll_{j_1,\ldots,j_5}\frac{D'^{1/4}(MN)^{3/4}}{\sqrt{d_1d_2}(M'N')^{1/4}Q}Z^{-j_1}R^{-j_2}B^{-j_3}M'^{-j_4}N'^{-j_5}
\]
for any fixed $j_1,\ldots,j_5\geq0$, where
\[
Z=\frac{\sqrt{qRB}}{\mathcal{Q}\sqrt{\ell_1'\ell_2'D'}}.
\]

Following \cite[(2.3)]{T} and \cite[Theorem 2.7]{KY} we write 
\begin{align*}
&\sum_{(c,\ell_1'\ell_2'D')=1}\frac{\psi\chi_1'\chi_2'\overline{\chi_3'\chi_4'}(c)}{cd_1^*d_2^*D''\sqrt{\ell_1'\ell_2'v^2D'}}S_{\phi_1}(qr, \overline{\ell_1'\ell_2'v^2D'}b;cd_1^*d_2^*D'')\Phi\Big(\frac{4\pi \sqrt{qr|b|}}{cd_1d_2D''\sqrt{\ell_1'\ell_2'D'}},r,b,m',n'\Big)\\
&\qquad=e\Big(-\frac{b\overline{d_1^*d_2^*D''}}{\ell_1'\ell_2'v^2D'}\Big)\sum_{\gamma}^{\Gamma_0(\widetilde{\ell_1}\widetilde{\ell_2}vD'D'')}\frac{S_{\infty\mathfrak{a}}^{\psi\chi_1'\chi_2'\overline{\chi_3'\chi_4'}\phi_1}(qr,b;\gamma)}{\gamma}\Phi\Big(\frac{4\pi \sqrt{qr|b|}}{\gamma},r,b,m',n'\Big),
\end{align*}
where $S_{\infty\mathfrak{a}}^{\chi}(m,n;c)$ is defined in \cite[Proposition 2.3]{Z1} and $\mathfrak{a}=1/d_1^*d_2^*D''$ is a singular cusp for $\Gamma_0(\widetilde{\ell_1}\widetilde{\ell_2}vD'D'')$. 
Applying the Kuznetsov formula \cite[Proposition 2.3]{Z1} (see also \cite[Theorem 6.9]{PY}) to the sum over $\gamma$ 
we obtain that 
\begin{align*}
&\mathcal{D}^-=d_1d_2D''\sqrt{\ell_1'\ell_2'D'}\sum_{r\asymp R}\sum_{|b|\asymp B}e\Big(-\frac{b\overline{d_1^*d_2^*D''}}{\ell_1'\ell_2'v^2D'}\Big)\\
&\qquad\times\sum_{\substack{d_1\ell_2'D_3'D_4'm'- d_2\ell_1'D_1'D_2'n'=b\\m'\asymp M',n'\asymp N'}}\big(\overline{\chi_1'}\chi_2''*\chi_1''\overline{\chi_2'}\big)(m')\big(\chi_3'\overline{\chi_4''}*\overline{\chi_3''}\chi_4'\big)(n')\widehat{S}_{\psi}(qr,m',n',\ell_1,\ell_2)\nonumber\\
	&\qquad\times \Big(\mathcal{M}^-(r,b,m',n')+\mathcal{E}^-(r,b,m',n')\Big),
\end{align*}
where $\mathcal{M}^-$ and $\mathcal{E}^-$ denote the contribution of the Maa\ss\ cusp forms and the Eisenstein spectrum, respectively. Since the treatment of these terms is similar, we shall focus only on $\mathcal{M}^-$, which is given by
\[
\mathcal{M}^-(r,b,m',n')=\sum_{f\in\mathcal{B}(\widetilde{\ell_1}\widetilde{\ell_2}vD'D'',\psi\chi_1'\chi_2'\overline{\chi_3'\chi_4'}\phi_1)}\check{\Phi}_{r,b,m',n'}(t_f)\frac{\sqrt{qr|b|}}{\cosh(\pi t_f)}\overline{\rho_{f,\infty}}(qr)\rho_{f,\mathfrak{a}}(b).
\] 
Here $\mathcal{B}(\widetilde{\ell_1}\widetilde{\ell_2}vD'D'',\psi\chi_1'\chi_2'\overline{\chi_3'\chi_4'}\phi_1)$ is a Hecke basis of the space of Maa\ss\ cusp forms of weight $\kappa=(1-\psi\chi_1'\chi_2'\overline{\chi_3'\chi_4'}\phi_1(-1))/2$ with respect to $\Gamma_0(\widetilde{\ell_1}\widetilde{\ell_2}vD'D'')$ and with nebentypus $\psi\chi_1'\chi_2'\overline{\chi_3'\chi_4'}\phi_1$, and the integral transform $\check{\Phi}$ is defined as
\[
\check{\Phi}_{r,b,m',n'}(t)=8i^{-\kappa}\cosh(\pi t)\int_{0}^\infty K_{2it}(x)\Phi(x,r,b,m',n')\frac{dx}{x}.
\] 

We use the Mellin inversion formula to separate the variables in $\check{\Phi}_{r,b,m',n'}$. We also restrict the spectral parameter to the dyadic interval $K\leq t_f<2K$ with $K\leq q^\varepsilon Z$ at a negligible cost. Similar to \cite[(4.25)]{Z1}, the contribution of $\mathcal{M}^-$ in this interval is
\begin{align}\label{2000}
\mathcal{D}_{K}^{-,\mathcal{M}}&=\frac{d_1d_2D''\sqrt{\ell_1'\ell_2'D'}}{(2\pi i)^4}\iiiint_{|\text{Im}(s_j)|\leq (Kq)^\eps}\mathcal{B}_{K}^{-,\mathcal{M}}(s_1,s_2,s_3,s_4)ds_1ds_2ds_3ds_4,
\end{align}
up to a negligible error term. Here
\begin{align}\label{2001}
\mathcal{B}_{K}^{-,\mathcal{M}}(s_1,s_2,s_3,s_4)&=\ \sideset{}{^*}\sum_{x,y \pmod v}\psi\chi_1'\chi_2'\overline{\chi_3'\chi_4'}\phi_1(y)\mathcal{A}_K(x,y,s_1,s_2,s_3,s_4),
\end{align}
where
\begin{align*}
\mathcal{A}_K(x,y,s_1,s_2,s_3,s_4)&=\sum_{\substack{f\in\mathcal{B}(\widetilde{\ell_1}\widetilde{\ell_2}vD'D'',\psi\chi_1'\chi_2'\overline{\chi_3'\chi_4'}\phi_1)\\K\leq t_f<2K}}\frac{\widetilde{\check{\Phi}(t_f)}(s_1,s_2,s_3,s_4)}{\cosh(\pi t_f)}\sum_{r\asymp R}\delta(r,s_1)\\
&\qquad \times \sum_{|b|\asymp B}|b|^{-s_4}\omega(s_2,s_3,b)\sqrt{qr|b|}\overline{\rho_{f,\infty}}(qr)\rho_{f,\mathfrak{a}}(b),
\end{align*}
with 
\[
\delta(r,s_1)=r^{-s_1}e\Big(\frac{qrx\overline{y}}{v}\Big)
\]
and
\begin{align*}
&\omega(s_2,s_3,b)=e\Big(-\frac{b\overline{d_1^*d_2^*D''}}{\ell_1'\ell_2'v^2D'}\Big)\\
&\quad\times\sum_{\substack{d_1\ell_2'D_3'D_4'm'- d_2\ell_1'D_1'D_2'n'=b\\m'\asymp M',n'\asymp N'}}\frac{(\overline{\chi_1'}\chi_2''*\chi_1''\overline{\chi_2'})(m')(\chi_3'\overline{\chi_4''}*\overline{\chi_3''}\chi_4')(n')}{m'^{s_2}n'^{s_3}}e\Big(\frac{(d_1\overline{\ell_1'D_1'D_2'}m'- d_2\overline{\ell_2'D_3'D_4'}n')\overline{xy}}{v}\Big).
\end{align*}
The same arguments in \cite{Z1} apply and we have
\[
\widetilde{\check{\Phi}(t)}(s_1,s_2,s_3,s_4)\ll_\varepsilon q^\varepsilon \frac{D'^{1/4}(MN)^{3/4}}{Z\sqrt{d_1d_2}(M'N')^{1/4}Q}
\]
and
\[
||\delta||_R\leq R^{1/2},\qquad ||\omega||_B\ll q^\varepsilon  M'B^{1/2}.
\]
The Cauchy-Schwarz inequality and the spectral large sieve \cite[Proposition 2.4]{Z1}  then imply that
\begin{align*}
\mathcal{A}_K(x,y,s_1,s_2,s_3,s_4)&\ll_\eps q^{\theta+\eps}\frac{D'^{1/4}(MN)^{3/4}}{Z\sqrt{d_1d_2}(M'N')^{1/4}Q}\cdot M'(RB)^{1/2}\\
&\qquad\times\Big(K+\frac{(vD')^{1/4}R^{1/2}}{\sqrt{\widetilde{\ell_1}\widetilde{\ell_2}vD'D''}}\Big)\Big(K+\frac{(vD')^{1/4}B^{1/2}}{\sqrt{\widetilde{\ell_1}\widetilde{\ell_2}vD'D''}}\Big)\\
&\ll_\eps q^{\theta+\eps}\frac{D'^{1/4}(MNM')^{3/4}(RB)^{1/2}}{\sqrt{d_1d_2}N'^{1/4}Q}\Big(Z+\frac{R^{1/2}+B^{1/2}}{(vD')^{1/4}\sqrt{\widetilde{\ell_1}\widetilde{\ell_2}}}+\frac{(RB)^{1/2}}{(vD')^{1/2}\widetilde{\ell_1}\widetilde{\ell_2}Z}\Big)\\
&=\mathcal{A}_{K,1}+\mathcal{A}_{K,2}+\mathcal{A}_{K,3}+\mathcal{A}_{K,4},
\end{align*}
say, where we have used the fact that $K\leq q^\varepsilon Z$.

Recall that
\[
R= \frac{2L(M+N)}{q},\quad B\leq D^2L^2N', \quad Z=\frac{\sqrt{qRB}}{\mathcal{Q}\sqrt{\ell_1'\ell_2'D'}},\quad M'\leq q^\eps\frac{D_1'D_2'\mathcal{Q}^2}{M}\quad\text{and}\quad N'\leq q^\eps\frac{D_3'D_4'\mathcal{Q}^2}{N}.
\]
So
\begin{align*}
\mathcal{A}_{K,1}&\ll_\eps \frac{q^{1/2+\theta+\eps}}{D'^{1/4}\sqrt{\widetilde{\ell_1}\widetilde{\ell_2}}}\cdot\frac{(MNM')^{3/4}RB}{N'^{1/4}Q\mathcal{Q}}\ll_\eps \frac{q^{-1/2+\theta+\eps}D^2L^3}{D'^{1/4}\sqrt{\widetilde{\ell_1}\widetilde{\ell_2}}}\cdot\frac{M(MNM'N')^{3/4}}{Q\mathcal{Q}}\\
&\ll_\eps \frac{q^{-1/2+\theta+\eps}D^2D'^{1/2}L^3}{\sqrt{\widetilde{\ell_1}\widetilde{\ell_2}}}\cdot M\mathcal{Q}.
\end{align*}
Hence, by \eqref{mainerrorformula}, \eqref{2000} and \eqref{2001}, the contribution of $\mathcal{A}_{K,1}$ to $\mathcal{E}_{1;\ell_1,\ell_2}^+$ is
\begin{align}\label{boundAK1}
&\ll_\eps q^{-1/2+\theta+\eps}D^4L^8 \sqrt{\frac{M}{N}}.
\end{align}
Similarly we obtain that
\begin{align*}
\mathcal{A}_{K,2}&\ll_\eps \frac{q^{\theta+\eps}}{\sqrt{\widetilde{\ell_1}\widetilde{\ell_2}}}\cdot\frac{(MNM')^{3/4}RB^{1/2}}{v^{1/4}\sqrt{d_1d_2}N'^{1/4}Q}\ll_\eps \frac{q^{-1+\theta+\eps}DL^2}{\sqrt{\widetilde{\ell_1}\widetilde{\ell_2}}}\cdot\frac{M(MNM')^{3/4}N'^{1/4}}{v^{1/4}\sqrt{d_1d_2}Q}\\
&\ll_\eps \frac{q^{-1+\theta+\eps}D^2D'^{1/4}L^2}{\sqrt{\widetilde{\ell_1}\widetilde{\ell_2}}}\cdot\frac{MN^{1/2}\mathcal{Q}}{v^{1/4}\sqrt{d_1d_2}},
\end{align*}
and its contribution to $\mathcal{E}_{1;\ell_1,\ell_2}^+$ is
\begin{align}\label{boundAK2}
&\ll_\eps q^{-1+\theta+\eps}D^{3}L^{11/2} \sqrt{M}.
\end{align}
The same argument shows that the contribution of $\mathcal{A}_{K,3}$ and $\mathcal{A}_{K,4}$ to 
$\mathcal{E}_{1;\ell_1,\ell_2}^+$ is
\begin{align}\label{boundAK34}
&\ll_\eps q^{-1/2+\theta+\eps}D^5L^{6} QN^{-1/2}+q^{-1+\theta+\eps}D^{4}L^{7/2} Q\nonumber\\
&\ll_\eps q^{-1/2+\theta+\eps}D^5L^{7} \sqrt{\frac{M}{N}}+q^{-1+\theta+\eps}D^{4}L^{9/2}\sqrt{M}.
\end{align}
Combining \eqref{boundD0}, \eqref{boundAK1}, \eqref{boundAK2} and \eqref{boundAK34} we obtain \eqref{errorwhattoprove}.

By \eqref{trivialcond}, $N\ll \widehat{q}^{\,1+\eps}(1+|t|)$ and $MN\ll \widehat{q}^{\,2+\eps}(1+|t|)^2$. So we deduce from \eqref{errorwhattoprove} that
	\begin{align*}
		\sum_{i=1}^{24}\mathcal{E}_{i,\ell_1,\ell_2}^+&\ll_\eps q^{-1+\varepsilon}DLM^{1/2}N^{1/4}+q^{-1/2+\theta+\eps}D^5L^8(1+|t|)^{1/2} \sqrt{\frac{M}{N}}\\
&\ll_\eps q^{-1+\varepsilon}DL(MN)^{3/8}\Big(\frac MN\Big)^{1/8}+q^{-1/2+\theta+\eps}D^5L^8(1+|t|)^{1/2} \sqrt{\frac{M}{N}}\nonumber\\
&\ll_\eps q^{-1/8-\theta/4+\delta/4+5\tau/8+\varepsilon}+q^{-\eta+\eps},
	\end{align*}
since $(M,N)\in\mathcal{S}_1$. We can verify that the first term on the right hand side is $O_\eps(q^{-\eta+\eps})$ with the constraint of $\eta$ in \eqref{choiceeta}, and the proof of Proposition \ref{boundforbalancedcase} is complete.
\end{proof}

\subsection{$M$, $N$ are not relatively close: $(M,N)\in\mathcal{S}_2^{>}\cup \mathcal{S}_2^{<}$}\label{notclosesection}

We note that the condition $\ell_1m\ne \ell_2n$ can be omitted in this case. We consider the case $(M,N)\in\mathcal{S}_2^{>}$, i.e.
\begin{equation}\label{condu-v}
\mu-\nu> 1-2\theta-10\delta-16\kappa-\tau-2\eta
\end{equation}
The other case is similar. We open the factor $(\chi_1*\chi_2)(m)=\sum_{m=m_1m_2}\chi_1(m_1)\chi_2(m_2)$ and apply a partition of unity to both $m_1$ and $m_2$ so that $m_1\asymp M_1$, $m_2\asymp M_2$ with $M_1M_2\asymp M$. Without loss of generality we assume that $M_1\leq M_2$. We shall follow \cite[Section 4]{Y} in this section to prove the following proposition.

\begin{prop}\label{unblancedcase}
	We have
	\begin{align}\label{unblancedcase1}
	I^{+}_{OD}(\ell_1,\ell_2,M,N)\ll_\eps	
		q^\eps DM^{1/4}\sqrt{N/M}+q^{1/4+\eps}\sqrt{D/M_1}.
				\end{align}
Furthermore,
\begin{align}\label{unblancedcase4}
	I^{+}_{OD}(\ell_1,\ell_2,M,N)&\ll_\eps q^\eps D^{1/2}LM_1\Big(\frac{ N}{M}\Big)^{3/2}+q^\eps\sqrt{\frac{N}{M}}\nonumber\\
&\qquad+q^\eps M_1^2\min\Big\{\frac{ D^{5/2}L}{\sqrt{MN}},q^{1/2+\eps}D^{9/2}L^{4/3}M^{-5/6}N^{-2/3}\Big\}.
\end{align}
\end{prop}
\begin{proof}
We shall consider the case $\ell_1m\equiv \ell_2n(\text{mod}\ q)$. The case $\ell_1m\equiv -\ell_2n(\text{mod}\ q)$ is similar. Applying the partition of unity to both $m_1$ and $m_2$ leads to $O((\log q)^2)$ sums of the shape
\begin{align*}
	&\frac{1}{\sqrt{MN}}\sum_{n}(\overline{\chi_3}*\overline{\chi_4})(n)\omega^- \Big ( \frac{n}{N} \Big )\sum_{m_1}\chi_1(m_1)\omega\Big ( \frac{m_1}{M_1} \Big )\\
	&\qquad\times \sum_{\substack{m_2\equiv \ell_2n\overline{\ell_1m_1}(\text{mod}\ q)}}\chi_2(m_2)\omega\Big ( \frac{m_2}{M_2} \Big )\omega^+\Big ( \frac{m_1m_2}{M} \Big )V\Big(\frac{ m_1m_2n}{\widehat{q}^{\,2}};t\Big).\nonumber
\end{align*}
Observe that the factors $\omega^+$ and $V$ can be removed at a negligible cost, so the problem reduces to bound
\begin{align*}
	I_{OD}(\ell_1,\ell_2,M_1,M_2,N)&=\frac{1}{\sqrt{MN}}\sum_{n}(\overline{\chi_3}*\overline{\chi_4})(n)\omega^- \Big ( \frac{n}{N} \Big )\sum_{m_1}\chi_1(m_1)\omega\Big ( \frac{m_1}{M_1} \Big )\\
	&\qquad\times \sum_{\substack{m_2\equiv \ell_2n\overline{\ell_1m_1}(\text{mod}\ q)}}\chi_2(m_2)\omega\Big ( \frac{m_2}{M_2} \Big ).\nonumber
\end{align*}
We have
\begin{align*}
	&\sum_{\substack{m_2\equiv \ell_2n\overline{\ell_1m_1}(\text{mod}\ q)}}\chi_2(m_2)\omega\Big ( \frac{m_2}{M_2} \Big )=\sum_{u \pmod {D_2}}\chi_2(u)\sum_{\substack{m_2\equiv \ell_2n\overline{\ell_1m_1} \pmod q\\m_2\equiv u \pmod {D_2}}}\omega\Big ( \frac{m_2}{M_2} \Big )\\
	&\qquad=\sum_{u \pmod {D_2}}\chi_2(u)\sum_{\substack{m_2\equiv \ell_2n\overline{\ell_1m_1}D_2\overline{D_2}+uq\overline{q}\pmod{ qD_2}}}\omega\Big ( \frac{m_2}{M_2} \Big ).
\end{align*}
Here $\overline{D_2}$ is the inverse of $D_2$ modulo $q$, and $\overline{q}$ is the inverse of $q$ modulo $D_2$. Applying Poisson summation to the sum over $m_2$ gives
\begin{align*}
	&\sum_{\substack{m_2\equiv \ell_2n\overline{\ell_1m_1} \pmod q}}\chi_2(m_2)\omega\Big ( \frac{m_2}{M_2} \Big )=\frac{M_2}{qD_2}\sum_{u \pmod {D_2}}\chi_2(u)\sum_{h\in\mathbb{Z}}\widehat{\omega}\Big(\frac{M_2h}{qD_2}\Big)e\Big(\frac{h\ell_2n\overline{\ell_1m_1}\overline{D_2}}{q}+\frac{hu\overline{q}}{D_2}\Big)\\
	&\qquad=\frac{\chi_2(q)\tau(\chi_2)M_2}{qD_2}\sum_{h\in\mathbb{Z}}\overline{\chi_2}(h)e\Big(\frac{h\ell_2n\overline{\ell_1D_2m_1}}{q}\Big)\widehat{\omega}\Big(\frac{M_2h}{qD_2}\Big).
\end{align*}
Hence
\begin{align}\label{IM1M2}
	I_{OD}(\ell_1,\ell_2,M_1,M_2,N)&=\frac{\chi_2(q)\tau(\chi_2)M_2}{qD_2\sqrt{MN}}\sum_{n}(\overline{\chi_3}*\overline{\chi_4})(n)\omega^- \Big ( \frac{n}{N} \Big )\sum_{m_1}\chi_1(m_1)\omega\Big ( \frac{m_1}{M_1} \Big )\nonumber\\
	&\qquad\times \sum_{h\in\mathbb{Z}}\overline{\chi_2}(h)e\Big(\frac{h\ell_2n\overline{\ell_1D_2m_1}}{q}\Big)\widehat{\omega}\Big(\frac{M_2h}{qD_2}\Big).
\end{align}
Due to the rapid decay of the Fourier transform we can restrict the sum over $h$ to
\[
|h|\leq H=\frac{q^{1+\eps}D_2}{M_2}.
\]

\begin{lem}\label{SabKM1}
	Let
	\[
		S(a,b,K,M_1;q)=\sideset{}{^*}\sum_{0<k\leq K}\bigg|\sideset{}{^*}\sum_{m_1}\chi_1(m_1)e\Big(\frac{ak\overline{bm_1}}{q}\Big)\omega\Big ( \frac{m_1}{M_1} \Big )\bigg|.
	\]
	If $K,M_1\ll q^{1+\eps}$, then 
	\[
S(a,b,K,M_1;q)\ll_\eps 
			q^\eps \sqrt{D_1}K\sqrt{M_{1}}+q^{3/4+\eps}\sqrt{D_1K}.
		\]
	\end{lem}
\begin{proof}
	%
%
We can bound $S$ by 
\[
S_0(a,b,K,M_1;q)=\sideset{}{^*}\sum_{0<k\leq K}\bigg|\sideset{}{^*}\sum_{m_1}\chi_1(m_1)e\Big(\frac{ak\overline{bm_1}}{q}\Big)\omega\Big ( \frac{m_1}{M_1} \Big )\bigg|\omega_0\Big(\frac{k}{K}\Big).
\]
Applying Cauchy-Schwartz's inequality we obtain that
\begin{equation}\label{S0sq}
S_0(a,b,K,M_1;q)^2\ll K\ \sideset{}{^*}\sum_{m_1',m_1'',k}\chi_1(m_1')\overline{\chi_1}(m_1'')e\Big(\frac{ak\overline{b}(\overline{m_1'}-\overline{m_1''})}{q}\Big)\omega\Big ( \frac{m_1'}{M_1} \Big )\omega\Big ( \frac{m_1''}{M_1} \Big )\omega_0\Big(\frac{k}{K}\Big).
\end{equation}
The right hand side is
\begin{align}\label{S02}
&\frac{K^2M_1^2}{q^3D_1}\sum_{x,y,z\in\mathbb{Z}}\overline{\chi_1}(x)\chi_1(y)T(x,y,z;q)\widehat{\omega}\Big(\frac{M_1x}{qD_1}\Big)\widehat{\omega}\Big(\frac{M_1y}{qD_1}\Big)\widehat{\omega_0}\Big(\frac{Kz}{q}\Big),
\end{align}
by Poisson summation, where
\[
T(x,y,z;q)=\sideset{}{^*}\sum_{\alpha,\beta,\gamma \pmod q}e\Big(\frac{\gamma(\overline{\alpha}-\overline{\beta})}{q}\Big)e\Big(\frac{\alpha x\overline{D_1}+\beta y\overline{D_1}+\gamma z}{q}\Big).
\]

The sum over $\gamma$ is $(q-1)$ if $\overline{\alpha}-\overline{\beta}+z\equiv 0 \pmod q$, otherwise it is $-1$. So
\begin{align*}
T(x,y,z;q)&=(q-1)\sideset{}{^*}\sum_{\substack{\alpha,\beta \pmod q\\ \overline{\alpha}-\overline{\beta}\equiv -z \pmod q}}\ e\Big(\frac{\alpha x\overline{D_1}+\beta y\overline{D_1}}{q}\Big)-\sideset{}{^*}\sum_{\substack{\alpha,\beta \pmod q\\ \overline{\alpha}-\overline{\beta}\not\equiv -z \pmod q}}\ e\Big(\frac{\alpha x\overline{D_1}+\beta y\overline{D_1}}{q}\Big)\\
&=q\sideset{}{^*}\sum_{\substack{\alpha,\beta \pmod q\\ \overline{\alpha}-\overline{\beta}\equiv -z \pmod q}}\ e\Big(\frac{\alpha x\overline{D_1}+\beta y\overline{D_1}}{q}\Big)-c_q(x)c_q(y).
\end{align*}
If $z=0$, then $\alpha=\beta$, and we get
\[
T(x,y,0;q)=qc_q(x+y)-c_q(x)c_q(y).
\]
If $(z,q)=1$, then solving $\overline{\alpha}-\overline{\beta}\equiv -z(\text{mod}\ q)$ gives $\beta=\overline{\overline{\alpha}+z}=\alpha\overline{(1+\alpha z)}$. Hence
\begin{align*}
T(x,y,z;q)&=q\sideset{}{^*}\sum_{\substack{\alpha \pmod q\\ \overline{\alpha}\not\equiv -z \pmod q}}\ e\Big(\frac{\alpha x\overline{D_1}+\alpha\overline{(1+\alpha z)}y\overline{D_1}}{q}\Big)-c_q(x)c_q(y)\\
&=qe\Big(\frac{- x\overline{zD_1}+y\overline{zD_1}}{q}\Big)\sideset{}{^*}\sum_{\substack{\alpha \pmod q\\ \alpha\not\equiv 1 \pmod q}}\ e\Big(\frac{\alpha x\overline{zD_1}-\overline{\alpha}y\overline{zD_1}}{q}\Big)-c_q(x)c_q(y)\\
&=qS(x\overline{z},-y\overline{z};q)e\Big(\frac{- x\overline{zD_1}+y\overline{zD_1}}{q}\Big)-q-c_q(x)c_q(y),
\end{align*}
by the change of variables $\alpha\rightarrow \alpha\overline{z}-\overline{z}$. Thus
\begin{align*}
T(x,y,z;q)\ll\begin{cases}
q & \text{if }z=0,x\ne -y,\\
q^2 & \text{if }z=0,x= -y,\\
1 & \text{if }z\ne 0,x=y=0,\\
q & \text{if }z\ne 0,x\ \text{or }y=0,\\
q^{3/2} & \text{if }x,y,z\ne 0.
\end{cases}
\end{align*}
Combining these bounds with \eqref{S0sq} and \eqref{S02} we obtain that
\[
S_0(a,b,K,M_1;q)^2\ll_\eps q^\eps D_1K^2M_1+q^{3/2+\eps}D_{1}K,
\]
which implies the last bound of the lemma.
\end{proof}

We return to \eqref{IM1M2} and let $k=hn$. Applying Lemma \ref{SabKM1} we obtain \eqref{unblancedcase1}.

For the second statement, we use the fact that
\begin{align*}
e\Big(\frac{h\ell_2n\overline{\ell_1D_2m_1}}{q}\Big)&=e\Big(-\frac{h\ell_2n\overline{q}}{\ell_1D_2m_1}\Big)e\Big(\frac{h\ell_2n}{q\ell_1D_2m_1}\Big)\\
&=e\Big(-\frac{h\ell_2n\overline{q}}{\ell_1D_2m_1}\Big)+O\Big(\frac{H\ell_2N}{q\ell_1D_2M_1}\Big),
\end{align*}
where $\overline{q}$ here is the inverse of $q$ modulo $\ell_1D_2m_1$. Applying this to \eqref{IM1M2} we get that
\begin{align}\label{IN1small}
	I_{OD}(\ell_1,\ell_2,M_1,M_2,N)&=\frac{\chi_2(q)\tau(\chi_2)M_2}{qD_2\sqrt{MN}}\sum_{h\in\mathbb{Z}}\overline{\chi_2}(h)\widehat{\omega}\Big(\frac{M_2h}{qD_2}\Big)\sum_{m_1}\chi_1(m_1)\omega\Big ( \frac{m_1}{M_1} \Big )\\
	&\qquad\times \sum_{n}(\overline{\chi_3}*\overline{\chi_4})(n) e\Big(-\frac{h\ell_2n\overline{q}}{\ell_1D_2m_1}\Big)\omega^- \Big ( \frac{n}{N} \Big )+O\Big(q^\eps\frac{\ell_2\sqrt{D_2} M_1N^{3/2}}{\ell_1M^{3/2}}\Big).\nonumber
\end{align}

Let $g=(h\ell_2,\ell_1D_2m_1)$. For $j=3,4$, let 
\[
D_j=(\ell_1D_2m_1/g,D_j)D_j'.
\]
By the Chinese Remainder Theorem we may factor $\chi_j=\chi_j'\chi_j''$, where $\chi_j'$ is a primitive character modulo $(\ell_1D_2m_1/g,D_j)$ and $\chi_j''$ is a primitive character modulo $D_j'$. In view of Theorem \ref{Voronoithm}, we can write the first term in \eqref{IN1small} as the sum of two main terms and two error terms. The main terms are of the shape
\begin{align*}
	&\frac{\tau(\chi_2)\tau(\overline{\chi_3})M_2L(1,\chi_3\overline{\chi_4})}{qD_2\sqrt{MN}}\sum_g\sum_{\substack{h\in\mathbb{Z}\\g|h\ell_2}}\overline{\chi_2}(h)\chi_3(h\ell_2/g)\widehat{\omega}\Big(\frac{M_2h}{qD_2}\Big)\\
	&\qquad\times \sum_{\substack{m_1\\D_3|\ell_1D_2m_1/g}}\frac{\chi_1(m_1)\overline{\chi_4}(\ell_1D_2m_1\overline{D_3}/g)}{\ell_1D_2m_1/g}\omega\Big ( \frac{m_1}{M_1} \Big )\int_0^\infty \omega^- \Big ( \frac{x}{N} \Big )dx,\nonumber
\end{align*}
which is trivially bounded by $O_\eps( q^\eps\sqrt{N/M})$, and the error terms are of the shape
\begin{align*}
	R&=\frac{\tau(\chi_2)M_2}{qD_2\sqrt{MN}}\sum_g\sum_{\substack{h\in\mathbb{Z}\\g|h\ell_2}}\overline{\chi_2}(h)\widehat{\omega}\Big(\frac{M_2h}{qD_2}\Big)\sum_{\substack{m_1\\g|\ell_1D_2m_1}}\frac{\chi_1(m_1)}{\ell_1D_2m_1\sqrt{D_3'D_4'}/g}\omega\Big ( \frac{m_1}{M_1} \Big )\\
	&\qquad\times \sum_{n'}(\chi_3''\overline{\chi_4'}*\overline{\chi_3'}\chi_4'')(n')e\Big(\frac{\overline{(h\ell_2/g)D_3'D_4'}qn'}{\ell_1D_2m_1/g}\Big)\int_0^\infty Y_0 \Big(\frac{4\pi\sqrt{n'x}}{(\ell_1D_2m_1/g)\sqrt{D_3'D_4'}} \Big)\omega^- \Big ( \frac{x}{N} \Big )dx.\nonumber
\end{align*}
We note that from the recurrence formula $(x^\nu Y_\nu(x))'=x^\nu Y_{\nu-1}(x)$ and integration by parts, we can truncate the sum over $n'$ to
\[
n'\leq N'=q^\eps \frac{\ell_1^2D_2^2D_3'D_4'm_1^2}{g^2N}.
\]
at a negligible cost. So trivially we have 
$$R\ll_\eps q^\eps\frac{ D^{5/2}LM_1^2}{\sqrt{MN}},$$
which is partially the third term in the second statement of the proposition.

We can also bound $R$ using the theory of exponential sums over the sum over $m_1$. Denote $\widetilde{g}=g/(g,\ell_1D_2)$. The the sum over $m_1$ becomes
\[
S=\frac{\chi_1(\widetilde{g})}{M_1}\sum_{m_1}\chi_1(m_1)e\Big(\frac{\overline{(h\ell_2/g)D_3'D_4'}qn'}{(\ell_1D_2\widetilde{g}/g)m_1}\Big)\omega_1\Big ( \frac{\widetilde{g}m_1}{M_1} \Big )\int_0^\infty Y_0 \Big(\frac{4\pi\sqrt{n'x}}{(\ell_1D_2\widetilde{g}/g)m_1\sqrt{D_3'D_4'}} \Big)\omega^- \Big ( \frac{x}{N} \Big )dx.
\]
Using the reciprocity law we write
\[
e\Big(\frac{\overline{(h\ell_2/g)D_3'D_4'}qn'}{(\ell_1D_2\widetilde{g}/g)m_1}\Big)=e\Big(-\frac{qn'\overline{(\ell_1D_2\widetilde{g}/g)m_1}}{(h\ell_2/g)D_3'D_4'}\Big)e\Big(\frac{qn'}{(\ell_1D_2\widetilde{g}/g)(h\ell_2/g)m_1D_3'D_4'}\Big).
\]
We also write 
\[
\chi_1(m_1)=\frac{1}{\tau(\overline{\chi_1})}\ \sideset{}{^*}\sum_{\substack{u \pmod{D_1}}}\overline{\chi_1}(u)e\Big(\frac{um_1}{D_1}\Big).
\]
Splitting $m_1$ into residue classes modulo $(h\ell_2/g)D_3'D_4'$ then gives
\begin{align*}
S&=\frac{\chi_1(\widetilde{g})}{M_1\tau(\overline{\chi_1})}\ \sideset{}{^*}\sum_{\substack{u \pmod{D_1}}}\overline{\chi_1}(u)\ \sideset{}{^*}\sum_{\substack{a \pmod{ (h\ell_2/g)D_3'D_4'}}}e\Big(-\frac{qn'\overline{(\ell_1D_2\widetilde{g}/g)a}}{(h\ell_2/g)D_3'D_4'}\Big)\\
&\qquad\times \sum_{\substack{r\\ m_1=a+(h\ell_2/g)D_3'D_4'r}}e\Big(\frac{qn'}{(\ell_1D_2\widetilde{g}/g)(h\ell_2/g)m_1D_3'D_4'}+\frac{um_1}{D_1}\Big)\\
&\qquad\qquad\qquad\qquad\qquad\qquad\times\omega_1\Big ( \frac{\widetilde{g}m_1}{M_1} \Big )\int_0^\infty Y_0 \Big(\frac{4\pi\sqrt{n'x}}{(\ell_1D_2\widetilde{g}/g)m_1\sqrt{D_3'D_4'}} \Big)\omega^- \Big ( \frac{x}{N} \Big )dx.
\end{align*}
To estimate the sum over $r$, we apply Proposition 4.5 in \cite{Y} to 
\[
f(x)=\frac{qn'}{(\ell_1D_2\widetilde{g}/g)(h\ell_2/g)D_3'D_4'(a+(h\ell_2/g)D_3'D_4'x)}+\frac{u(a+(h\ell_2/g)D_3'D_4'x)}{D_1}
\]
with the choices of
\[
F\asymp \frac{qn'g^2}{\ell_1\ell_2hD_2D_3'D_4'M_1},\qquad Q\asymp \frac{M_1}{\widetilde{g}(h\ell_2/g)D_3'D_4'}
\]
and $A$ an absolute constant. 
In doing so we get
\[
\sum_r \ll_\eps q^\eps F^{1/6}Q^{1/2},
\]
and hence
\[
S \ll_\eps q^\eps F^{1/6}Q^{1/2}\cdot \frac{\sqrt{D_1}}{M_1}(h\ell_2/g)D_3'D_4'N.
\]
Using this bound we obtain that
\[
R \ll_\eps q^{1/2+\eps}\ell_1\ell_2^{1/3}\sqrt{D_1}D_2^2D_3D_4M_1^2M^{-5/6}N^{-2/3},
\]
as desired.
\end{proof}

Proposition \ref{unblancedcase} and the choice of $\eta$ in \eqref{choiceeta1} lead to the following proposition.

\begin{prop}\label{unblancedcase2}
	We have
	\begin{align*}
	I^{+}_{OD}(\ell_1,\ell_2,M,N)\ll_\eps	
		q^{-\eta+\eps}.
\end{align*}
\end{prop}
\begin{proof}
    By \eqref{unblancedcase1} in Proposition \ref{unblancedcase}, $I^{+}_{OD}(\ell_1,\ell_2,M,N)$ is $O_\eps( q^{-\eta+\eps})$ unless either
\[
\mu-\nu<\frac{\mu}{2}+2\delta+2\eta
\]
or
\begin{equation}\label{conditionno3}
\mu_1<\frac12+\delta+2\eta.
\end{equation}
If $\mu-\nu<\mu/2+2\delta+2\eta$, then
\begin{align*}
2\mu=(\mu+\nu)+(\mu-\nu)<2+4\delta+2\tau+2\eta+\frac{\mu}{2},
\end{align*}
by \eqref{trivialcond}, and
\begin{align*}
\frac{\mu}{2}> 1-2\theta-12\delta-16\kappa-\tau-4\eta,
\end{align*}
by \eqref{condu-v}.
Hence
\[
\frac{25}{16}-2(12\delta+16\kappa+\tau+4\eta)<\mu <\frac43+\frac23(4\delta+2\tau+2\eta).
\]
This is impossible given the constraint of $\eta$ in \eqref{choiceeta}.
So we now assume that \eqref{conditionno3} holds. 

We now use estimate \eqref{unblancedcase4} in Proposition \ref{unblancedcase}. By \eqref{condu-v}, \eqref{conditionno3} and \eqref{choiceeta}, the first two terms in \eqref{unblancedcase4} are $O_\eps( q^{-\eta+\eps})$. Also, by \eqref{trivialcond}, \eqref{condu-v} and \eqref{choiceeta} we have
\begin{align*}
M^2N=(MN)^{3/2}\sqrt{\frac{M}{N}}\gg q^{\frac32(2-2\kappa-2\eta)+\frac12(1-2\theta-10\delta-16\kappa-\tau-2\eta)}\gg q^3.
\end{align*}
It hence follows that the last term in \eqref{unblancedcase4} is bounded by
\begin{align*}
	&\ll_\eps q^\eps D^{9/2}L^{4/3}M_1^2\min\Big\{\frac{ 1}{\sqrt{MN}},q^{1/2+\eps}M^{-5/6}N^{-2/3}\Big\}\\
&\ll_\eps q^{1/2+\eps} D^{9/2}L^{4/3}M_1^2(MN)^{-3/4}\Big(\frac{M}{N}\Big)^{-1/12}.
\end{align*}
In view of \eqref{trivialcond}, \eqref{condu-v} and \eqref{conditionno3}, this is 
\begin{align*}
&\ll_\eps q^{1/2+9\delta/2+4\kappa/3+2(1/2+\delta+2\eta)-3(2-2\kappa-2\eta)/4-(1-2\theta-10\delta-16\kappa-\tau-2\eta)/12+\eps}\\
&\ll_\eps q^{-\eta+\eps},
\end{align*}
by \eqref{choiceeta2}. This completes the proof of Proposition \ref{unblancedcase2}.
\end{proof}

\subsection{Combining the error terms}

Combining Proposition \ref{boundforbalancedcase}, Proposition \ref{unblancedcase2} and the choice of $\eta$ in \eqref{choiceeta1} we see that the total error terms for $I^+(\ell_1,\ell_2)$ are
\[
\ll_\eps q^\eps (q^{-11/16}D^{80}L^{96}(1+|t|)^{10})^{1/28}+q^\eps (q^{-25/32}D^{88}L^{50}(1+|t|))^{1/80}.
\]

Similar arguments apply to $I^-$ and we have that the error terms coming from $I^-(D_1D_2\ell_1,D_3D_4\ell_2)$ give the error terms in the theorem.

\section{Proof of Theorem \ref{twistedfirst} - The off-diagonal main terms}\label{odmsection}

We start with $\mathcal{M}_{\pm;1,3}^+(M,N)$ as given in \eqref{2002}.
The sum over $k$ is
\[
\chi_1\overline{\chi_3}(q)\sideset{}{^*}\sum_{k \pmod \ell}\overline{\chi_1}\chi_3(k)e\Big(\frac{-rk}{\ell}\Big).
\]
Moreover, by a change of variables we have
\begin{align*}
&\frac{1}{M^{1/2+it}N^{1/2-it}}\int\int g_{\ell_1,\ell_2}^{\pm}(x,y,\ell)dxdy\\
&\qquad=\frac{1}{\widetilde{\ell_1}^{1/2-it}\widetilde{\ell_2}^{1/2+it}}\int\int x^{-(1/2+it)}y^{-(1/2-it)}\Delta_\ell(x\pm y-qr)\varphi(x\pm y-qr)\\
&\qquad\qquad\qquad\qquad\qquad\qquad\times\omega\Big ( \frac{x}{\widetilde{\ell_1}M} \Big )\omega \Big ( \frac{y}{\widetilde{\ell_2}N} \Big )V\Big(\frac{xy}{\widetilde{\ell_1}\widetilde{\ell_2}\widehat{q}^{\,2}};t\Big)dxdy\\
&\qquad=\frac{1}{\widetilde{\ell_1}^{1/2-it}\widetilde{\ell_2}^{1/2+it}}\int\int x^{-(1/2+it)}\big(\pm(u+qr-x)\big)^{-(1/2-it)}\Delta_\ell(u)\varphi(u)\\
&\qquad\qquad\qquad\qquad\qquad\qquad\times\omega\Big ( \frac{x}{\widetilde{\ell_1}M} \Big )\omega \Big ( \frac{\pm(u+qr-x)}{\widetilde{\ell_2}N} \Big )V\Big(\frac{\pm x(u+qr-x)}{\widetilde{\ell_1}\widetilde{\ell_2}\widehat{q}^{\,2}};t\Big)dxdu.
\end{align*}
This is, for all $\ell$, bounded by
\begin{align*}
&\ll \frac{1}{\sqrt{\widetilde{\ell_1}\widetilde{\ell_2}}}\cdot \frac{1}{\sqrt{\widetilde{\ell_1}\widetilde{\ell_2}MN}}\cdot \min\{\widetilde{\ell_1} M,\widetilde{\ell_2}N\}\cdot\int\big|\Delta_\ell(u)\big| du\ll_\varepsilon \frac{L^{\eps}\sqrt{MN}}{\widetilde{\ell_1}M+\widetilde{\ell_2}N},
\end{align*}
by \eqref{bdDelta},
and if $\ell\ll Q^{1-\varepsilon}$, then it is equal to
\begin{align*}
&\frac{1}{\widetilde{\ell_1}^{1/2-it}\widetilde{\ell_2}^{1/2+it}}\int x^{-(1/2+it)}\big(\pm(qr-x)\big)^{-(1/2-it)}\omega\Big ( \frac{x}{\widetilde{\ell_1}M} \Big )\omega \Big ( \frac{\pm(qr-x)}{\widetilde{\ell_2}N} \Big )V\Big(\frac{\pm x(qr-x)}{\widetilde{\ell_1}\widetilde{\ell_2}\widehat{q}^{\,2}};t\Big)dx\\
&\qquad\qquad+O_C(Q^{-C})
\end{align*}
for any fixed $C>0$, by \cite[(18)]{DFI}. So we can first restrict the sum over $\ell$ in \eqref{2002} to $\ell\ll Q^{1-\varepsilon}$, and then extend it to all $\ell$ at the cost of an error of size
\begin{align*}
	\ll_\varepsilon \frac{q^{\varepsilon}\sqrt{MN}}{\sqrt{Q}}.
\end{align*}
The condition $|r|\leq R$ can also be removed due to the support of $V^+$.

\subsection{Off-diagonal main terms - The $x$-integral}

We have
\begin{align*}
&\mathcal{M}_{-;1,3}^+(M,N)=\sqrt{D_1D_3}\epsilon(\chi_1)\epsilon(\overline{\chi_3})\chi_1\overline{\chi_3}(q)\overline{\chi_2}(D_1)\chi_4(D_3)L(1,\overline{\chi_1}\chi_2)L(1,\chi_3\overline{\chi_4})\\
&\qquad\times \sum_{r\ne 0}\sum_{\substack{\ell\geq 1\\D_1|\ell_1''\\D_3|\ell_2''}}\frac{\overline{\chi_1}(\ell_1')\chi_3(\ell_2')\chi_2(\ell_1'')\overline{\chi_4}(\ell_2'')}{\ell_1''\ell_2''\widetilde{\ell_1}^{1/2-it}\widetilde{\ell_2}^{1/2+it}}\, \sideset{}{^*}\sum_{k \pmod{\ell}}\overline{\chi_1}\chi_3(k)e\Big(\frac{-rk}{\ell}\Big)\\
&\qquad\times\int_{\max\{0,qr\}}^\infty x^{-(1/2+it)}(x-qr)^{-(1/2-it)}\omega\Big ( \frac{x}{\widetilde{\ell_1}M} \Big )\omega \Big ( \frac{x-qr}{\widetilde{\ell_2}N} \Big )V\Big(\frac{x(x-qr)}{\widetilde{\ell_1}\widetilde{\ell_2}\widehat{q}^{\,2}};t\Big)dx.
\end{align*}
Using \eqref{formulaV+} and writing $\omega$ in terms of its Mellin transform, the $x$ integral is
\begin{align*}
&\frac{1}{(2\pi i)^3}\int_{(c)}\int_{(c_1)}\int_{(c_2)}G(s)g(s,t)\widetilde{\omega}(u)\widetilde{\omega}(v)\big(\widetilde{\ell_1}\widetilde{\ell_2}\widehat{q}^{\,2}\big)^{s}(\widetilde{\ell_1}M)^u(\widetilde{\ell_2}N)^v\\
&\qquad\times \int_{\max\{0,qr\}}^\infty x^{-(1/2+s+u+it)}(x-qr)^{-(1/2+s+v-it)}dxdvdu\frac{ds}{s},
\end{align*}
where $c,c_1,c_2>0$ are small. For absolute convergence, we shall require the conditions
\begin{equation}\label{contour1}\begin{cases}
\text{Re}(2s+u+v)>0,\ \text{Re}(s+v)<1/2 & \quad\text{if }r>0,\\
\text{Re}(2s+u+v)>0,\ \text{Re}(s+u)<1/2 & \quad\text{if }r<0.
\end{cases}\end{equation} Under these assumptions, the innermost integral is equal to (see, for instance, \cite[17.43.21 and 17.43.22]{GR})
\[
(q|r|)^{-(2s+u+v)}\times\begin{cases}
\frac{\Gamma(2s+u+v)\Gamma(1/2-s-v+it)}{\Gamma(1/2+s+u+it)} & \quad\text{if }r>0,\\
\frac{\Gamma(2s+u+v)\Gamma(1/2-s-u-it)}{\Gamma(1/2+s+v-it)} & \quad\text{if } r<0.
\end{cases}
\]
Hence
\begin{align*}
\mathcal{M}_{-;1,3}^+(M,N)&=\sqrt{D_1D_3}\epsilon(\chi_1)\epsilon(\overline{\chi_3})\chi_1\overline{\chi_3}(q)\overline{\chi_2}(D_1)\chi_4(D_3)L(1,\overline{\chi_1}\chi_2)L(1,\chi_3\overline{\chi_4})\\
&\qquad\times \sum_{r\geq 1}\sum_{\substack{\ell\geq 1\\D_1|\ell_1''\\D_3|\ell_2''}}\frac{\overline{\chi_1}(\ell_1')\chi_3(\ell_2')\chi_2(\ell_1'')\overline{\chi_4}(\ell_2'')}{\ell_1''\ell_2'}\, \sideset{}{^*}\sum_{k \pmod \ell}\overline{\chi_1}\chi_3(k)e\Big(\frac{-rk}{\ell}\Big)\\
&\qquad\times\frac{1}{(2\pi i)^3}\int_{(c)}\int_{(c_1)}\int_{(c_2)}G(s)g(s,t)\widetilde{\omega}(u)\widetilde{\omega}(v)\widehat{q}^{\,2s}q^{-(2s+u+v)}M^uN^v\\
&\qquad\times \frac{1}{\widetilde{\ell_1}^{1/2-s-u-it}\widetilde{\ell_2}^{1/2-s-v+it}r^{2s+u+v}}H^-(s,u,v)dvdu\frac{ds}{s},
\end{align*}
where
\[
H^-(s,u,v)=\Gamma(2s+u+v)\Big(\frac{\Gamma(1/2-s-v+it)}{\Gamma(1/2+s+u+it)}+\frac{\Gamma(1/2-s-u-it)}{\Gamma(1/2+s+v-it)}\Big).
\]

A similar formula holds for $\mathcal{M}_{+;1,3}^+(M,N)$ with $H^-$ being replaced by $H^+$, where
\[
H^+(s,u,v)=\frac{\Gamma(1/2-s-u-it)\Gamma(1/2-s-v+it)}{\Gamma(1-2s-u-v)}.
\]
It follows that
\begin{align}\label{ODMbefore}
\mathcal{M}_{1,3}^+(M,N)&=\mathcal{M}_{+;1,3}^+(M,N)+\mathcal{M}_{-;1,3}^+(M,N)\nonumber\\
&=\sqrt{D_1D_3}\epsilon(\chi_1)\epsilon(\overline{\chi_3})\chi_1\overline{\chi_3}(q)\overline{\chi_2}(D_1)\chi_4(D_3)L(1,\overline{\chi_1}\chi_2)L(1,\chi_3\overline{\chi_4})\nonumber\\
&\qquad\times \sum_{r\geq 1}\sum_{\substack{\ell\geq 1\\D_1|\ell_1''\\D_3|\ell_2''}}\frac{\overline{\chi_1}(\ell_1')\chi_3(\ell_2')\chi_2(\ell_1'')\overline{\chi_4}(\ell_2'')}{\ell_1''\ell_2''}\, \sideset{}{^*}\sum_{k(\text{mod}\ \ell)}\overline{\chi_1}\chi_3(k)e\Big(\frac{-rk}{\ell}\Big)\nonumber\\
&\qquad\times\frac{1}{(2\pi i)^3}\int_{(c)}\int_{(c_1)}\int_{(c_2)}G(s)g(s,t)\widetilde{\omega}(u)\widetilde{\omega}(v)\widehat{q}^{\,2s}q^{-(2s+u+v)}M^uN^v\\
&\qquad\times \frac{1}{\widetilde{\ell_1}^{1/2-s-u-it}\widetilde{\ell_2}^{1/2-s-v+it}r^{2s+u+v}}H(s,u,v)dvdu\frac{ds}{s},\nonumber
\end{align}
where
\begin{align*}
H(s,u,v)&=\frac{\Gamma(1/2-s-u-it)\Gamma(1/2-s-v+it)}{\Gamma(1-2s-u-v)}\\
&\qquad\qquad+\Gamma(2s+u+v)\Big(\frac{\Gamma(1/2-s-v+it)}{\Gamma(1/2+s+u+it)}+\frac{\Gamma(1/2-s-u-it)}{\Gamma(1/2+s+v-it)}\Big)\\
&=\pi^{1/2}\frac{\Gamma(\frac{2s+u+v}{2})\Gamma(\frac{1/2-s-u-it}{2})\Gamma(\frac{1/2-s-v+it}{2})}{\Gamma(\frac{1-2s-u-v}{2})\Gamma(\frac{1/2+s+u+it}{2})\Gamma(\frac{1/2+s+v-it}{2})},
\end{align*}
from \cite[Lemma 8.2]{Y}.

\subsection{Off-diagonal main terms - The arithmetic sums}

In this section we shall evaluate the arithmetic sum
\begin{align}
\label{arithmetic}
\sum_{\substack{\ell\geq 1\\D_1|\ell_1''\\D_3|\ell_2''}}\frac{\overline{\chi_1}(\ell_1')\chi_3(\ell_2')\chi_2(\ell_1'')\overline{\chi_4}(\ell_2'')}{\ell_1''\ell_2''}\sum_{r\geq 1}\frac{1}{r^s}\, \sideset{}{^*}\sum_{k \pmod \ell}\overline{\chi_1}\chi_3(k)e\Big(\frac{-rk}{\ell}\Big).
\end{align}

We have 

\begin{align*}
\sum_{r\geq 1}\frac{1}{r^s}\sideset{}{^*}\sum_{k \pmod \ell}\overline{\chi_1}\chi_3(k)e\Big(\frac{-rk}{\ell}\Big)& = \sum_{d|\ell} \frac{1}{d^s} \sum_{(r,\ell/d)=1} \frac{1}{r^s} \,\sideset{}{^*}\sum_{k \pmod \ell}\overline{\chi_1}\chi_3(k)\ e\Big(  - \frac{rk}{\ell/d}\Big)
\\
&= \sum_{d|\ell} \frac{1}{d^s} \sumstar_{k  \pmod \ell}\overline{\chi_1}\chi_3(k) \sumstar_{a \pmod \ell} e\Big(  - \frac{ak}{\ell/d}\Big) \frac{1}{\varphi(\ell/d)} \sum_{\chi \pmod{\ell/d}} \overline{\chi}(a) L(s,\chi) \\
&= \sum_{d|\ell} \frac{1}{\varphi(\ell/d)d^s}  \sum_{\chi \pmod{\ell/d}} \chi(-1) \tau(\overline{\chi}) L(s,\chi)  \sumstar_{k \pmod \ell}\overline{\chi_1}\chi_3 \chi(k)  .
\end{align*}
 Hence, letting $\chi_{0,\ell}$ denote the principal character modulo $\ell$, we have
\begin{align*}
\sideset{}{^*}\sum_{k \pmod \ell} \chi\overline{\chi_1}\chi_3 (k) &= \sum_{k \pmod \ell}\chi\overline{\chi_1}\chi_3 \chi_{0,\ell} (k)= \begin{cases}
\varphi(\ell), & \text{if }  \, \chi \chi_{0,\ell} =\chi_1\overline{ \chi_3} \chi_{0,\ell},\\
0, & \text{otherwise}.
\end{cases}
\end{align*}
In the case when $\chi \chi_{0,\ell} = \chi_1 \overline{\chi_3} \chi_{0,\ell}$, since $\chi$ is a character modulo $\ell/d$, we must have that $D_1 D_3|(\ell/d)$. We also have that $\chi = \chi_1 \overline{\chi_3} \chi_{0,\ell/d}$.

The primitive character $\chi_1\overline{\chi_3}$ (mod $D_1D_3$) induces the character $\chi_1\overline{\chi_3} \chi_{0,\ell/d}$ (mod $\ell/d$), so by Theorem 9.10 of \cite{MV}, we have
\begin{align*}
 \tau(\overline{\chi_1}\chi_3 \chi_{0,\ell/d}) &= \mu\Big( \frac{\ell/d}{D_1D_3} \Big) \, \overline{\chi_1}\Big( \frac{\ell/d}{D_1D_3} \Big) \, \chi_3 \Big( \frac{\ell/d}{D_1D_3} \Big) \, \tau(\overline{\chi_1}\chi_3)
 \\
 &=\mu\Big( \frac{\ell/d}{D_1D_3} \Big) \, \overline{\chi_1}\Big( \frac{\ell/d}{D_1D_3} \Big) \, \chi_3 \Big( \frac{\ell/d}{D_1D_3} \Big) \, \overline{\chi_1}(D_3) \,  \chi_3 (D_1) \,  \tau(\overline{\chi_1}) \, \tau(\chi_3)
 \\
 &= \mu\Big( \frac{\ell/d}{D_1D_3} \Big) \, \overline{\chi_1}\Big( \frac{\ell/d}{D_1} \Big) \, \chi_3 \Big( \frac{\ell/d}{D_3} \Big)  \,  \tau(\overline{\chi_1}) \, \tau(\chi_3).
\end{align*}
Thus
\begin{align*}
\sum_{r\geq1}  \frac{1}{r^s}\sideset{}{^*}\sum_{k \pmod \ell}\overline{\chi_1}\chi_3(k)e\Big(\frac{-rk}{\ell}\Big)&= \sum_{d| \frac{\ell}{D_1D_3}} \frac{\varphi(\ell)}{\varphi(\ell/d)d^s} \mu\Big( \frac{\ell/d}{D_1D_3} \Big) \, \overline{\chi_1}\Big( \frac{\ell/d}{D_1} \Big) \, \chi_3 \Big( \frac{\ell/d}{D_3} \Big)  \\
&\qquad \times  \tau(\overline{\chi_1}) \, \tau(\chi_3) \, L(s,\chi_1\overline{\chi_3} \chi_{0,\ell/d}),
\end{align*}
where we also used the fact that $\chi_1$ and $\chi_3$ are even characters.

We need to consider
\begin{align*}
\eqref{arithmetic} & = \tau(\overline{\chi_1}) \, \tau(\chi_3) \,\sum_{\substack{\ell\geq 1\\D_1|\ell_1''\\D_3|\ell_2''}}\frac{\overline{\chi_1}(\ell_1')\chi_3(\ell_2')\chi_2(\ell_1'')\overline{\chi_4}(\ell_2'')}{\ell_1''\ell_2''}\sum_{d| \frac{\ell}{D_1D_3}} \frac{\varphi(\ell)}{\varphi(\ell/d)d^{s}} \mu\Big( \frac{\ell/d}{D_1D_3} \Big) \, \overline{\chi_1}\Big( \frac{\ell/d}{D_1} \Big) \, \chi_3 \Big( \frac{\ell/d}{D_3} \Big) \\
& \qquad\times   L(s,\chi_1\overline{\chi_3} \chi_{0,\ell/d}).
\end{align*}
Since $D_1|\ell$ and $D_3|\ell$ and $(D_1,D_3)=1$, we have that $D_1 D_3|\ell$. Write $\ell=D_1D_3 W$. We rewrite \eqref{arithmetic} as 
\begin{align*}
\eqref{arithmetic} &= \tau(\overline{\chi_1}) \, \tau(\chi_3) \,\sum_{\substack{\ell\geq 1\\D_1|\ell_1''\\D_3|\ell_2''}}\frac{\overline{\chi_1}(\ell_1')\chi_3(\ell_2')\chi_2(\ell_1'')\overline{\chi_4}(\ell_2'')}{\ell_1''\ell_2''}\sum_{d| W} \frac{\varphi(D_1D_3W)}{\varphi( D_1D_3 W/d) d^{s}} \mu \Big( \frac{W}{d} \Big) \\
& \qquad\times \overline{\chi_1} \Big( \frac{D_3W}{d} \Big) \chi_3 \Big(  \frac{D_1W}{d}\Big)  L(s,\chi_1\overline{\chi_3} \chi_{0,\ell/d}).
\end{align*}
We write
$$L(s,\chi_1 \overline{\chi_3} \chi_{0,\ell/d}) = L(s,\chi_1 \overline{\chi_3}) \prod_{p|W/d} \Big(1 - \frac{\chi_1 \overline{\chi_3}(p)}{p^{s}} \Big).$$
With a change of variables $d \mapsto W/d$, we have that
\begin{align*}
\sum_{d|W} & \frac{d^{s}}{\varphi(D_1D_3d)W^{s}} \mu(d) \overline{\chi_1} \chi_3(d) \prod_{p|d} \Big(1 - \frac{\chi_1 \overline{\chi_3}(p)}{p^{s}} \Big) = \frac{1}{\varphi(D_1D_3)W^{s}}\prod_{\substack{p|W \\ p \nmid D_1 D_3}}\Big( 1 + \frac{1}{p-1}-  \frac{ \overline{\chi_1} \chi_3(p)}{(p-1)p^{-s}}\Big).
\end{align*}
Hence
\begin{align}
\eqref{arithmetic} &=  \frac{L(s,\chi_1 \overline{\chi_3})}{\varphi(D_1D_3)}\tau(\overline{\chi_1}) \, \tau(\chi_3) \sum_{\substack{\ell\geq 1\\D_1|\ell_1''\\D_3|\ell_2''}} \varphi(D_1D_3W)  \frac{\overline{\chi_1}(\ell_1')\chi_3(\ell_2')\chi_2(\ell_1'')\overline{\chi_4}(\ell_2'')}{\ell_1''\ell_2''W^{s}}  \overline{\chi_1}(D_3) \chi_3(D_1) \nonumber  \\
 &\qquad \times  \prod_{\substack{p|W \\ p \nmid D_1 D_3}}\Big( 1 + \frac{1}{p-1}-  \frac{ \overline{\chi_1} \chi_3(p)}{(p-1)p^{-s}}\Big).\label{off1}
\end{align}
 We write $\tilde{\ell_1}=A_1B_1$, where $A_1|D_1^{\infty}$, and $(B_1,D_1)=1$. From the condition $D_1 | D_1D_3W/ (D_1D_3W,\tilde{\ell_1})$ we get that $A_1|W$. Similarly if we write $\tilde{\ell_2} =A_3B_3$ where $A_3|D_3^{\infty}$ and $(B_3,D_3)=1$, it follows that $A_3|W.$ Since $(\tilde{\ell_1},\tilde{\ell_2})=1$, we get that $A_1 A_3|W$. Write $W=A_1A_3 W'$. Now note that
\begin{align*}
\overline{\chi_1}(\ell_1') = \overline{\chi_1} \Big( \frac{B_1}{(B_1,D_1D_3A_1A_3 W')} \Big) = \overline{\chi_1}(B_1) \chi_1 ((B_1,D_3A_3W')),
\end{align*}
and
\begin{align*}
\chi_3(\ell_2') = \chi_3(B_3) \overline{\chi_3}((B_3, D_1A_1 W')).
\end{align*} 
We get overall 
\begin{align*}
\eqref{arithmetic} &=  \frac{ L(s,\chi_1 \overline{\chi_3}) }{\varphi(D_1D_3)} \tau(\overline{\chi_1}) \, \tau(\chi_3) \overline{\chi_1}(D_3B_1) \chi_3(D_1B_3) \sum_{W' \geq 1} \varphi(D_1D_3 A_1 A_3 W') \chi_1((B_1, D_3A_3W'))  \\
& \qquad\times \overline{\chi_3}((B_3, D_1 A_1 W'))  \chi_2 \Big(  \frac{D_1D_3A_3W'}{(D_3A_3W',B_1)}\Big) \overline{\chi_4} \Big(  \frac{D_1D_3A_1  W'}{(D_1A_1W',B_3)}\Big) \frac{ (D_3 A_3W',B_1)(D_1 A_1W',B_3)}{(A_1A_3)^{s+1} (D_1D_3)^{2} W'^{2+s}} \\
&\qquad \times \prod_{\substack{p|W' \\ p \nmid D_1 D_3}} \Big( 1 + \frac{1}{p-1}-  \frac{ \overline{\chi_1} \chi_3(p)}{(p-1)p^{-s}}\Big).
\end{align*}
We evaluate the sum over $W'$, and we write
\begin{align*}
& \sum_{W' \geq 1}  \varphi(D_1D_3 A_1 A_3 W') \chi_1((B_1, D_3A_2W')) \overline{\chi_3}((B_3, D_1 A_1 W')) \chi_2 \Big(  \frac{D_1D_3A_3W'}{(D_3A_3W',B_1)}\Big) \\
&\qquad \times   \overline{\chi_4} \Big(  \frac{D_1D_3A_1  W'}{(D_1A_1W',B_3)}\Big) \frac{ (D_3 A_3W',B_1)(D_1 A_1W',B_3)}{ W'^{2+s}} \prod_{\substack{p|W' \\ p \nmid D_1 D_3}} \Big( 1 + \frac{1}{p-1}-  \frac{ \overline{\chi_1} \chi_3(p)}{(p-1)p^{-s}}\Big) \\
&\ = \prod_{p \nmid D_1 D_3 B_1 B_3} A_p(s) \prod_{p|A_1} B_p(s) \prod_{p|(D_1,B_3)} C_p(s) \prod_{\substack{p|D_1 \\ p \nmid A_1B_3}} D_p(s) \prod_{\substack{p|B_3 \\ p \nmid D_1D_2  D_4 }} E_p(s) \prod_{\substack{p|(B_3,D_4)}} F_p(s) \\
&\qquad \times \prod_{p|A_3} G_p(s) \prod_{p|(D_3,B_1)} H_p(s)\prod_{\substack{p|D_3 \\ p \nmid A_3 B_1}} I_p(s) \prod_{\substack{p|B_1 \\ p \nmid D_3D_2 D_4}} J_p(s)  \prod_{\substack{p|(B_1,D_2) }} K_p(s). 
\end{align*}
The Euler factor when $p \nmid D_1 D_3 B_1 B_3$ is equal to
$$ A_p(s) = \Big( 1 - \frac{\chi_2 \overline{\chi_4}(p)}{p^{1+s}} \Big)^{-1} \Big( 1 - \frac{\overline{\chi_1} \chi_2 \chi_3 \overline{\chi_4}(p)}{p^2}\Big).$$
The Euler factor when $p|D_1$ is equal to
\begin{align*}
\sum_{j=0}^{\infty} \frac{\varphi(p^{1+a_1+j})p^{\min\{b_3,1+a_1+j\}} \overline{\chi_3}(p^{\min\{b_3,1+a_1+j\}} \chi_2(p^{1+j}) \overline{\chi_4}(p^{1+a_1+j-\min\{b_3,1+a_1+j\}})}{p^{j(2+s)}},
\end{align*} where $b_3= v_p(B_3)$ and $a_1=v_p(A_1)$. 
Note that since $(A_1,B_3)=1$ we have $a_1b_3=0$. The factor above when $p|A_1$ simplifies to
$$ B_p(s) = p^{a_1}(p-1) \chi_2(p) \overline{\chi_4}(p^{1+a_1})  \Big( 1 - \frac{\chi_2 \overline{\chi_4}(p)}{p^{1+s}} \Big)^{-1},$$
and the factor when $p|D_1, p|B_3$ is
\begin{align*}
C_p(s) &= \chi_2(p)(p-1) \Big( -\frac{\chi_2\overline{\chi_3 \chi_4}(p)}{p^{s}} + \overline{\chi_3}(p)p + \frac{ \chi_2\overline{\chi_3}(p^{b_3})(\overline{\chi_4}(p)-p\chi_3(p))}{p^{b_3 s}}\Big)  \\
&\qquad \times \Big( 1- \frac{\chi_2 \overline{\chi_4}(p)}{p^{1+s}}\Big)^{-1} \Big( 1- \frac{\chi_2 \overline{\chi_3}(p)}{p^{s}}\Big)^{-1}.
\end{align*}
The factor when $p|D_1, p \nmid A_1B_3$ is equal to  
$$ D_p(s) = (p-1) \chi_2(p) \overline{\chi_4}(p)  \Big( 1 - \frac{\chi_2 \overline{\chi_4}(p)}{p^{1+s}} \Big)^{-1}.$$
The Euler factor when $p|B_3, p \nmid D_1 D_2 D_4$ is equal to 
\begin{align*}
E_p(s) = 1+ \sum_{j=1}^{\infty} \frac{ \varphi(p^j) \overline{\chi_3}(p^{\min\{j,b_3\}}) \chi_2(p^j) \overline{\chi_4}(p^{j-\min\{b_3,j\}}) p^{ \min\{b_3,j\}}}{p^{j(2+s)}} \Big(  1+ \frac{1}{p-1} - \frac{\overline{\chi_1} \chi_3(p)}{(p-1)p^{-s}}\Big) .
\end{align*}
This simplifies to 
\begin{align*}
E_p(s) = \Big( & 1- \frac{\chi_2 \overline{\chi_4}(p)}{p^{1+s}}\Big)^{-1} \Big( 1- \frac{\chi_2 \overline{\chi_3}(p)}{p^{s}}\Big)^{-1}  \Big( 1 - \frac{\overline{\chi_1} \chi_2(p)}{p}+ \frac{ \overline{\chi_1} \chi_2^2 \overline{\chi_4}(p)}{p^{2+s}}- \frac{\chi_2 \overline{\chi_4}(p)}{p^{1+s}} \\
& - \frac{\overline{\chi_1} \chi_2(p)  (\chi_2\overline{\chi_3})(p^{b_2})}{p^{1+b_3 s}}(1- p \chi_3 \overline{\chi_4}(p)  )  + \frac{\chi_2(p) (\chi_2 \overline{\chi_3}(p))^{b_3}}{p^{1+(1+b_3)s}} (\overline{\chi_4}(p) - p \overline{\chi_3}(p))\Big).
\end{align*}
The Euler factor when $p|(B_3,D_4)$ is equal to 
\begin{align*}
F_p(s)  = 1+ \sum_{j=1}^{b_3} \frac{ \varphi(p^j) \overline{\chi_3}(p^{\min\{j,b_3\}}) \chi_2(p^j) \overline{\chi_4}(p^{j-\min\{b_3,j\}}) p^{ \min\{b_3,j\}}}{p^{j(2+s)}} \Big(  1+ \frac{1}{p-1} - \frac{\overline{\chi_1} \chi_3(p)}{(p-1)p^{-s}}\Big).
\end{align*}
This simplifies to 
\begin{align*}
F_p(s) &=  \chi_1(p) \overline{\chi_3}(p)\Big( 1- \frac{\chi_2 \overline{\chi_3}(p)}{p^{s}}\Big)^{-1} \Big(  \chi_1 \chi_3(p) +\chi_3(p)- \frac{\chi_2\chi_3(p)}{p} - \frac{\chi_2(p^{1+b_3}) \overline{\chi_3}(p^{b_3})}{p^{1+(b_3+1)s}} - \frac{\chi_1(p)}{p^{s-1}} \Big).
\end{align*}
Similar expressions hold for $ G_p, H_p, I_p, J_p, K_p$ (with $\chi_2$ replaced by $\overline{\chi_4}$ and $\chi_1$ by $\overline{\chi_3}$). 
We denote by
\begin{align*}
    C_p^{+}(s) &= \Big( -\frac{\chi_2\overline{\chi_3 \chi_4}(p)}{p^{s}} + \overline{\chi_3}(p)p + \frac{ \chi_2\overline{\chi_3}(p^{b_3})(\overline{\chi_4}(p)-p\chi_3(p))}{p^{b_3 s}}\Big) \Big( 1- \frac{\chi_2 \overline{\chi_3}(p)}{p^{s}}\Big)^{-1},\\
    E_p^{+}(s) &= \Big( 1- \frac{\chi_2 \overline{\chi_3}(p)}{p^{s}}\Big)^{-1}  \Big( 1 - \frac{\overline{\chi_1} \chi_2(p)}{p}+ \frac{ \overline{\chi_1} \chi_2^2 \overline{\chi_4}(p)}{p^{2+s}}- \frac{\chi_2 \overline{\chi_4}(p)}{p^{1+s}} \\
& - \frac{\overline{\chi_1} \chi_2(p)  (\chi_2\overline{\chi_3})(p^{b_2})}{p^{1+b_3 s}}(1- p \chi_3 \overline{\chi_4}(p)  )  + \frac{\chi_2(p) (\chi_2 \overline{\chi_3}(p))^{b_3}}{p^{1+(1+b_3)s}} (\overline{\chi_4}(p) - p \overline{\chi_3}(p))\Big), \\
F_p^{+}(s) &= \Big( 1- \frac{\chi_2 \overline{\chi_3}(p)}{p^{s}}\Big)^{-1} \Big(  \chi_1 \chi_3(p) +\chi_3(p)- \frac{\chi_2\chi_3(p)}{p} - \frac{\chi_2(p^{1+b_3}) \overline{\chi_3}(p^{b_3})}{p^{1+(b_3+1)s}} - \frac{\chi_1(p)}{p^{s-1}} \Big),
\end{align*}
and similarly for $H_p^{+}(s), J_p^{+}(s), K_p^{+}(s).$

Overall we get 
\begin{align}\label{doublesum}
 \eqref{arithmetic} = c^+(s;\widetilde{\ell_1},\widetilde{\ell_2},D_1,D_3)\epsilon(\overline{\chi_1})\epsilon(\chi_3)\overline{\chi_1\chi_4}(D_3)\chi_2\chi_3(D_1)\frac{(\widetilde{\ell_1},D_3)(\widetilde{\ell_2},D_1)}{(D_1D_3)^{3/2}}L(s,\chi_1\overline{\chi_3})L(1+s,\chi_2\overline{\chi_4}),
\end{align}
where
\begin{align}
c^+(s;\tilde{\ell_1},\tilde{\ell_2},D_1,D_3) &= \frac{\overline{\chi_1}(B_1) \chi_3(B_3) \overline{\chi_4}(D_1) \chi_2(D_3) \overline{\chi_4}(A_1) \chi_2(A_3) \overline{\chi_2}((B_1,D_3)) \chi_4((B_3,D_1))}{(B_1,D_3)(B_3,D_1) (A_1A_3)^s L(2,\overline{\chi_1\chi_4} \chi_2 \chi_3 )} \nonumber \\
&\qquad \times   \prod_{\substack{p|(B_3,D_1) }} C_p^{+}(s)  \prod_{\substack{p|B_3 \\ p \nmid D_1D_2  D_4}} E_p^{+}(s)   \Big(1 - \frac{\overline{\chi_1\chi_4} \chi_2 \chi_3 (p)}{p^2} \Big)^{-1} \prod_{\substack{p|(B_3,D_4) }} F_p^{+}(s)  \nonumber  \\
& \qquad\times  \prod_{\substack{p|(B_1,D_3) }} H_p^{+}(s) 
\prod_{\substack{p|B_1 \\ p \nmid D_3 D_2 D_4}} J_p^{+}(s)  \Big(1 - \frac{\overline{\chi_1\chi_4} \chi_2 \chi_3 (p)}{p^2} \Big)^{-1} \prod_{\substack{p|(B_1,D_2) }} K_p^{+}(s)     . \label{cs}
\end{align}

\subsection{Off-diagonal main terms - Contour shifts}

We shall need to move the contours in \eqref{ODMbefore} multiple times to move the sums over $r$ and $\ell$ inside the integrals, and then make use of the result of the previous section.

We choose
\[
c= c_1=2\varepsilon\qquad\text{and}\qquad c_2=\varepsilon
\]
in \eqref{ODMbefore}.  We first move the $s$-contour to the right to Re$(s)=1/2-4\varepsilon/3$, crossing a simple pole at $s=1/2-u-it$ from $H(s,u,v)$. By Cauchy's theorem we get
\begin{align*}
\mathcal{M}_{1,3}^+(M,N)&=\mathcal{M}_{1,3}^{{\scriptscriptstyle '}}(M,N)-\text{Res}_{s=1/2-u-it}\\
&=\mathcal{M}_{1,3}^{{\scriptscriptstyle '}}(M,N)-R_1,
\end{align*}
say, where $\mathcal{M}_{1,3}^{{\scriptscriptstyle '}}(M,N)$ is the new integral. We have
\begin{align*}
&R_1   =\sqrt{D_1D_3}\epsilon(\chi_1)\epsilon(\overline{\chi_3})\chi_1\overline{\chi_3}(q)\overline{\chi_2}(D_1)\chi_4(D_3)L(1,\overline{\chi_1}\chi_2)L(1,\chi_3\overline{\chi_4})\nonumber\\
&\qquad\times \sum_{r\geq 1}\sum_{\substack{\ell\geq 1\\D_1|\ell_1''\\D_3|\ell_2''}}\frac{\overline{\chi_1}(\ell_1')\chi_3(-\ell_2')\chi_2(\ell_1'')\overline{\chi_4}(\ell_2'')}{\ell_1''\ell_2''}\, \sideset{}{^*}\sum_{k \pmod \ell}\overline{\chi_1}\chi_3(k)e\Big(\frac{-rk}{\ell}\Big)\nonumber\\
&\qquad\times\frac{1}{(2\pi i)^2}\int_{(2\varepsilon)}\int_{(\varepsilon)}G(\tfrac12-u-it)g(\tfrac12-u-it,t)\widetilde{\omega}(u)\widetilde{\omega}(v)\widehat{q}^{\,1-2u-2it}q^{-(1-u+v-2it)}M^uN^v\\
&\qquad\times \frac{1}{\widetilde{\ell_2}^{u-v+2it}r^{1-u+v-2it}}\frac{dvdu}{1/2-u-it}.
\end{align*}
Moving the $v$-contour here to the right to Re$(v)=3\varepsilon$ encounters no pole. Moreover, we can now move the sums over $r$ and $\ell$ inside the integrals. So we obtain that 
\begin{align}\label{formulaR1}
R_1 & =\frac{(\widetilde{\ell_1},D_3)(\widetilde{\ell_2},D_1)}{D_1D_3}\chi_1\overline{\chi_3}(q)\overline{\chi_1}(D_3)\chi_3(D_1)L(1,\overline{\chi_1}\chi_2)L(1,\chi_3\overline{\chi_4})\nonumber\\
&\qquad\times\frac{1}{(2\pi i)^2}\int_{(2\varepsilon)}\int_{(3\varepsilon)}c^+(1-u+v-2it;\widetilde{\ell_1},\widetilde{\ell_2},D_1,D_3)G(\tfrac12-u-it)g(\tfrac12-u-it,t)\nonumber\\
&\qquad\times \widetilde{\omega}(u)\widetilde{\omega}(v)\widehat{q}^{\,1-2u-2it}q^{-(1-u+v-2it)}M^uN^v\\
&\qquad\times\frac{L(1-u+v-2it,\chi_1\overline{\chi_3})L(2-u+v-2it,\chi_2\overline{\chi_4})}{\widetilde{\ell_2}^{u-v+2it}}\frac{dvdu}{1/2-u-it},\nonumber
\end{align}
by \eqref{doublesum}.

With $\mathcal{M}_{1,3}^{{\scriptscriptstyle '}}(M,N)$, the sums over $r$ and $\ell$ can be moved inside the integrals and can be written using \eqref{doublesum}.
We then move the $s$-contour back to Re$(s)=2\varepsilon$, crossing a simple pole at $s=1/2-u-it$ again. We write
\begin{align*}
\mathcal{M}_{1,3}^{{\scriptscriptstyle '}}(M,N)&=\mathcal{M}_{1,3}^{{\scriptscriptstyle ''}}(M,N)+\text{Res}_{s=1/2-u-it}'\\
&=\mathcal{M}_{1,3}^{{\scriptscriptstyle ''}}(M,N)+R_1',
\end{align*}
say, where
\begin{align*}
&\mathcal{M}_{1,3}^{{\scriptscriptstyle ''}}(M,N)=\frac{(\widetilde{\ell_1},D_3)(\widetilde{\ell_2},D_1)}{D_1D_3}\chi_1\overline{\chi_3}(q)\overline{\chi_1}(D_3)\chi_3(D_1)L(1,\overline{\chi_1}\chi_2)L(1,\chi_3\overline{\chi_4})\nonumber\\
&\quad\times\frac{1}{(2\pi i)^3}\int_{(2\eps)}\int_{(2\eps)}\int_{(\eps)}c^+(2s+u+v;\widetilde{\ell_1},\widetilde{\ell_2},D_1,D_3)G(s)g(s,t)\widetilde{\omega}(u)\widetilde{\omega}(v)\widehat{q}^{\,2s}q^{-(2s+u+v)}M^uN^v\\
&\quad\times \frac{L(2s+u+v,\chi_1\overline{\chi_3})L(1+2s+u+v,\chi_2\overline{\chi_4})}{\widetilde{\ell_1}^{1/2-s-u-it}\widetilde{\ell_2}^{1/2-s-v+it}}H(s,u,v)dvdu\frac{ds}{s}.
\end{align*}
Observe that $R_1'$ has the same expression as $R_1$ in \eqref{formulaR1} but with the $v$-integral over Re$(v)=\varepsilon$. So $R_1=R_1'$, and hence
\begin{align}\label{M13formula}
&\mathcal{M}_{1,3}^+(M,N)=\frac{(\widetilde{\ell_1},D_3)(\widetilde{\ell_2},D_1)}{D_1D_3}\chi_1\overline{\chi_3}(q)\overline{\chi_1}(D_3)\chi_3(D_1)L(1,\overline{\chi_1}\chi_2)L(1,\chi_3\overline{\chi_4})\nonumber\\
&\quad\times\frac{1}{(2\pi i)^3}\int_{(2\eps)}\int_{(2\eps)}\int_{(\eps)}c^+(2s+u+v;\widetilde{\ell_1},\widetilde{\ell_2},D_1,D_3)G(s)g(s,t)\widetilde{\omega}(u)\widetilde{\omega}(v)\widehat{q}^{\,2s}q^{-(2s+u+v)}M^uN^v\\
&\quad\times \frac{L(2s+u+v,\chi_1\overline{\chi_3})L(1+2s+u+v,\chi_2\overline{\chi_4})}{\widetilde{\ell_1}^{1/2-s-u-it}\widetilde{\ell_2}^{1/2-s-v+it}}H(s,u,v)dvdu\frac{ds}{s}.\nonumber
\end{align}

\subsection{Off-diagonal main terms - Summing over partition of unity}

We note that the off-diagonal main terms $\mathcal{M}_{1,3}^+(M,N)$ above come from the pairs $(M,N)\in\mathcal{S}_1$. We wish to add all the missing pairs $(M, N)$.

By moving the contours to Re$(s)=\sigma$, Re$(u)=c_1$ and Re$(v)=c_2$ we get
\begin{align*}
\mathcal{M}_{1,3}^+(M,N)&\ll_\eps \widetilde{\ell_1}^{1/2-(\sigma+c_2)}\widetilde{\ell_2}^{1/2-(\sigma+c_1)}q^{-(c_1+c_2)+\eps}\big(D(1+|t|)\big)^{2\sigma}M^{c_1}N^{c_2}\\
&\ll_\eps q^{-(c_1+c_2)+\eps}\big(D(1+|t|)\big)^{2\sigma}L^{1-(2\sigma+c_1+c_2)}M^{c_1}N^{c_2},
\end{align*}
provided that
\[
\sigma>0,\qquad 0<2\sigma+c_1+c_2<1,\qquad -\frac12<\sigma+c_1<\frac12\qquad\text{and}\qquad -\frac12<\sigma+c_2<\frac12.
\]
With suitable values of $\sigma,c_1,c_2$ we then obtain that
\[
\mathcal{M}_{1,3}^+(M,N)\ll_{\eps,C} q^\eps\min\bigg\{L\Big(\frac{MN}{q^2D^2(1+|t|)^2}\Big)^{-C},\frac{\sqrt{MN}}{q},L\sqrt{\frac{N}{M}}\bigg\}
\]
for any $C>0$. Hence
\begin{equation}\label{addingmissingpairs}
\mathcal{M}_{1,3}^{+,OD}(\ell_1,\ell_2):=\sum_{M,N\in\mathcal{S}_1}\mathcal{M}_{1,3}^+(M,N)=\sum_{M,N}\mathcal{M}_{1,3}^+(M,N)+O_\eps(q^{-\eta+\eps}).
\end{equation}

We now use the following result in \cite[p.30]{Y} to assemble the partition of unity.
\begin{lem}\label{sumunity}
Let $F(u,v)$ be an entire function of rapid decay in each variable in a fixed strip $|\emph{Re}(u)|$, $|\emph{Re}(v)|\leq C$. Then we have
\[
\sum_{M,N}\frac{1}{(2\pi i)^2}\int_{(c_1)}\int_{(c_2)}F(u,v)\widetilde{\omega}(u)\widetilde{\omega}(v)dvdu=F(0,0).
\]
\end{lem}

Applying Lemma \ref{sumunity} to \eqref{addingmissingpairs} we get
\begin{align*}
\mathcal{M}_{1,3}^{+,OD}(\ell_1,\ell_2)&=\frac{(\widetilde{\ell_1},D_3)(\widetilde{\ell_2},D_1)}{D_1D_3}\chi_1\overline{\chi_3}(q)\overline{\chi_1}(D_3)\chi_3(D_1)L(1,\overline{\chi_1}\chi_2)L(1,\chi_3\overline{\chi_4})\\
&\qquad\times\frac{1}{2\pi i}\int_{(2\eps)}c^+(2s;\widetilde{\ell_1},\widetilde{\ell_2},D_1,D_3)G(s)g(s,t)\Big(\frac{(D_1D_2D_3D_4)^{1/4}}{\pi}\Big)^{2s}\\
&\qquad\times \frac{L(2s,\chi_1\overline{\chi_3})L(1+2s,\chi_2\overline{\chi_4})}{\widetilde{\ell_1}^{1/2-s-it}\widetilde{\ell_2}^{1/2-s+it}}\pi^{1/2}\frac{\Gamma(s)\Gamma(\frac{1/2-s-it}{2})\Gamma(\frac{1/2-s+it}{2})}{\Gamma(\frac{1-2s}{2})\Gamma(\frac{1/2+s+it}{2})\Gamma(\frac{1/2+s-it}{2})}\frac{ds}{s}+O_\eps(q^{-\eta+\eps}).
\end{align*}
By the functional equation, we have
\[
\Big(\frac{D_1D_3}{\pi}\Big)^{2s-1/2}\frac{\Gamma(s)}{\Gamma(\frac{1-2s}{2})}L(2s,\chi_1\overline{\chi_3})=\chi_1(D_3)\overline{\chi_3}(D_1)\epsilon(\chi_1)\epsilon(\overline{\chi_3})L(1-2s,\overline{\chi_1}\chi_3).
\]
Hence
\begin{align*}
\mathcal{M}_{1,3}^{+,OD}(\ell_1,\ell_2)&=\frac{(\widetilde{\ell_1},D_3)(\widetilde{\ell_2},D_1)}{\sqrt{D_1D_3}}\chi_1\overline{\chi_3}(q)\epsilon(\chi_1)\epsilon(\overline{\chi_3})L(1,\overline{\chi_1}\chi_2)L(1,\chi_3\overline{\chi_4})\nonumber\\
&\qquad\times\frac{1}{2\pi i}\int_{(2\eps)}c^+(2s;\widetilde{\ell_1},\widetilde{\ell_2},D_1,D_3)G(s)g(s,t)\Big(\frac{(D_2D_4)^{1/2}}{(D_1D_3)^{3/2}}\Big)^{s}\\
&\qquad\times \frac{L(1-2s,\overline{\chi_1}\chi_3)L(1+2s,\chi_2\overline{\chi_4})}{\widetilde{\ell_1}^{1/2-s-it}\widetilde{\ell_2}^{1/2-s+it}}\frac{\Gamma(\frac{1/2-s-it}{2})\Gamma(\frac{1/2-s+it}{2})}{\Gamma(\frac{1/2+s+it}{2})\Gamma(\frac{1/2+s-it}{2})}\frac{ds}{s}+O_\eps(q^{-\eta+\eps}).
\end{align*}
As $$(\widetilde{\ell_1},D_3)(\widetilde{\ell_2},D_1)=(D_1\widetilde{\ell_1},D_3\widetilde{\ell_2})=\frac{(D_1\ell_1,D_3\ell_2)}{(\ell_1,\ell_2)},$$ 
we can write this as
\begin{align}\label{M13+}
\mathcal{M}_{1,3}^{+,OD}(\ell_1,\ell_2)&=\frac{(D_1\ell_1,D_3\ell_2)}{\sqrt{D_1D_3}\ell_1^{1/2-it}\ell_2^{1/2+it}}\chi_1\overline{\chi_3}(q)\epsilon(\chi_1)\epsilon(\overline{\chi_3})L(1,\overline{\chi_1}\chi_2)L(1,\chi_3\overline{\chi_4})\nonumber\\
&\qquad\times\frac{1}{2\pi i}\int_{(2\eps)}c^+\Big(2s;\frac{\ell_1}{(\ell_1,\ell_2)},\frac{\ell_2}{(\ell_1,\ell_2)},D_1,D_3\Big)G(s)g(s,t)\Big(\frac{(D_2D_4)^{1/2}\ell_1\ell_2}{(D_1D_3)^{3/2}(\ell_1,\ell_2)^2}\Big)^{s}\nonumber\\
&\qquad\times L(1-2s,\overline{\chi_1}\chi_3)L(1+2s,\chi_2\overline{\chi_4})\frac{\Gamma(\frac{1/2-s-it}{2})\Gamma(\frac{1/2-s+it}{2})}{\Gamma(\frac{1/2+s+it}{2})\Gamma(\frac{1/2+s-it}{2})}\frac{ds}{s}+O_\eps(q^{-\eta+\eps}).
\end{align}

Notice that $I^{-}(\ell_1,\ell_2)$ is the same as $I^{+}(\ell_1,\ell_2)$ but with the swaps $\chi_1\longleftrightarrow\chi_3$ and $\chi_2\longleftrightarrow\chi_4$. So similarly to \eqref{expandIOD}, we can write
\begin{align*}
I^{-}_{OD}(\ell_1,\ell_2,M,N)&=\mathcal{M}_{3,1}^-(M,N)+\mathcal{M}_{3,2}^-(M,N)+\mathcal{M}_{4,1}^-(M,N)+\mathcal{M}_{4,2}^-(M,N)\\
&\qquad+\sum_{i=1}^{24}\mathcal{E}_{i,\ell_1,\ell_2}^-,\nonumber
\end{align*}
and that
\begin{align*}
\mathcal{M}_{4,2}^{-,OD}(\ell_1,\ell_2)&:=\sum_{M,N}\mathcal{M}_{4,2}^-(M,N)+O_\eps(q^{-\eta+\eps})\\
&=\frac{(D_4\ell_1,D_2\ell_2)}{\sqrt{D_4D_2}\ell_1^{1/2-it}\ell_2^{1/2+it}}\chi_4\overline{\chi_2}(q)\epsilon(\chi_4)\epsilon(\overline{\chi_2})L(1,\overline{\chi_4}\chi_3)L(1,\chi_2\overline{\chi_1})\nonumber\\
&\qquad\times\frac{1}{2\pi i}\int_{(2\eps)}c^-\Big(2s;\frac{\ell_1}{(\ell_1,\ell_2)},\frac{\ell_2}{(\ell_1,\ell_2)},D_4,D_2\Big)G(s)g(s,t)\Big(\frac{(D_3D_1)^{1/2}\ell_1\ell_2}{(D_4D_2)^{3/2}(\ell_1,\ell_2)^2}\Big)^{s}\\
&\qquad\times L(1-2s,\overline{\chi_4}\chi_2)L(1+2s,\chi_3\overline{\chi_1})\frac{\Gamma(\frac{1/2-s-it}{2})\Gamma(\frac{1/2-s+it}{2})}{\Gamma(\frac{1/2+s+it}{2})\Gamma(\frac{1/2+s-it}{2})}\frac{ds}{s}+O_\eps(q^{-\eta+\eps}),
\end{align*}
where the corresponding dual factor $c^-$ is similar to $c^+$ but with 
\[
\chi_1\longleftrightarrow \chi_4\qquad\text{and}\qquad \chi_2\longleftrightarrow \chi_3.
\]
This implies that
\begin{align}\label{M42-}
& \Big(\frac{D_1D_2}{D_3D_4}\Big)^{-it}\epsilon \mathcal{M}_{4,2}^{-,OD}(D_1D_2\ell_1,D_3D_4\ell_2)\nonumber\\
&\ =\frac{(D_1\ell_1,D_3\ell_2)}{\sqrt{D_1D_3}\ell_1^{1/2-it}\ell_2^{1/2+it}}\chi_1\overline{\chi_3}(q)\epsilon(\chi_1)\epsilon(\overline{\chi_3})L(1,\overline{\chi_4}\chi_3)L(1,\chi_2\overline{\chi_1})\nonumber\\
&\ \ \ \times\frac{1}{2\pi i}\int_{(2\eps)}c^-\Big(2s;\frac{D_1D_2\ell_1}{(D_1D_2\ell_1,D_3D_4\ell_2)},\frac{D_3D_4\ell_2}{(D_1D_2\ell_1,D_3D_4\ell_2)},D_4,D_2\Big)G(s)g(s,t)\\
&\ \ \ \times \Big(\frac{(D_1D_3)^{3/2}\ell_1\ell_2}{(D_4D_2)^{1/2}(D_1D_2\ell_1,D_3D_4\ell_2)^2}\Big)^{s}L(1-2s,\overline{\chi_4}\chi_2)L(1+2s,\chi_3\overline{\chi_1})\frac{\Gamma(\frac{1/2-s-it}{2})\Gamma(\frac{1/2-s+it}{2})}{\Gamma(\frac{1/2+s+it}{2})\Gamma(\frac{1/2+s-it}{2})}\frac{ds}{s}\nonumber\\
&\qquad\qquad+O_\eps(q^{-\eta+\eps}).\nonumber
\end{align}

We will need the following identity.

\begin{lem}
\label{identity}
We have 
\begin{align*}
&c^-\Big(-2s;\frac{D_1D_2\ell_1}{(D_1D_2\ell_1,D_3D_4\ell_2)},\frac{D_3D_4\ell_2}{(D_1D_2\ell_1,D_3D_4\ell_2)},D_4,D_2\Big)\Big(\frac{\ell_1\ell_2}{(D_1D_2\ell_1,D_3D_4\ell_2)^2}\Big)^{-s}\\
&\qquad=c^+\Big(2s;\frac{\ell_1}{(\ell_1,\ell_2)},\frac{\ell_2}{(\ell_1,\ell_2)},D_1,D_3\Big)\Big(\frac{\ell_1\ell_2}{(\ell_1,\ell_2)^2}\Big)^{s}.
\end{align*}
\end{lem}
\begin{proof}
Writing
\[
\frac{D_1D_2\ell_1}{(D_1D_2\ell_1,D_3D_4\ell_2)}=A_4B_4\qquad\text{and}\qquad \frac{D_3D_4\ell_2}{(D_1D_2\ell_1,D_3D_4\ell_2)}=A_2B_2,
\]
where $A_j|D_j^{\infty}$, and $(B_j,D_j)=1$ for $j=2,4,$
we note that
\begin{align*}
& c^-\Big(s;\frac{D_1D_2\ell_1}{(D_1D_2\ell_1,D_3D_4\ell_2)},\frac{D_3D_4\ell_2}{(D_1D_2\ell_1,D_3D_4\ell_2)},D_4,D_2\Big) \\
&\quad= \frac{\overline{\chi_4}(B_4) \chi_2(B_2) \overline{\chi_1}(D_4) \chi_3(D_2) \overline{\chi_1}(A_4) \chi_3(A_2) \overline{\chi_3}((B_4,D_2)) \chi_1((B_2,D_4))}{(B_4,D_2)(B_2,D_4) (A_2A_4)^s L(2,\overline{\chi_1\chi_4} \chi_2 \chi_3)} \nonumber \\
&\qquad\qquad \times   \prod_{\substack{p|(B_2,D_4) }} C_p^{-}(s)  \prod_{\substack{p|B_2 \\ p \nmid D_1D_3 D_4}} E_p^{-}(s)   \Big(1 - \frac{\overline{\chi_1\chi_4} \chi_2 \chi_3 (p)}{p^2} \Big)^{-1} \prod_{\substack{p|(B_2,D_1) }} F_p^{-}(s)   \nonumber  \\
&\qquad\qquad \times  \prod_{\substack{p|(B_4,D_2)}} H_p^{-}(s) 
\prod_{\substack{p|B_4 \\ p \nmid D_1 D_3 D_2}} J_p^{-}(s)  \Big(1 - \frac{\overline{\chi_1\chi_4} \chi_2 \chi_3 (p)}{p^2} \Big)^{-1} \prod_{\substack{p|(B_4,D_3) }} K_p^{-}(s)     . 
\end{align*}
We rearrange the identity we need to prove as 
\begin{align} 
\label{tocheck}
\frac{c^+\Big(2s;\frac{\ell_1}{(\ell_1,\ell_2)},\frac{\ell_2}{(\ell_1,\ell_2)},D_1,D_3\Big)}{c^-\Big(-2s;\frac{D_1D_2\ell_1}{(D_1D_2\ell_1,D_3D_4\ell_2)},\frac{D_3D_4\ell_2}{(D_1D_2\ell_1,D_3D_4\ell_2)},D_4,D_2\Big)} = \frac{(B_3,D_1D_2)^{2s} (B_1,D_3D_4)^{2s}}{(A_1A_3B_1B_3)^{2s}}. 
\end{align}
We will check this identity prime by prime. We again use the convention that $x= v_p(X)$ for any integer $X$,

We look at the Euler factor when $p|(B_3,D_1, B_2)$. In this case $b_2 = b_3-1$ and we have that
\begin{equation}C_p^{+}(2s) = F_p^{-}(-2s) \chi_2(p^{b_2}) \overline{\chi_3}(p^{1+b_2}) p^{1-2b_2s}.\label{cp}
\end{equation}
 Then the Euler factor corresponding to $p|(B_3,D_1,B_2)$ on the left-hand side of \eqref{tocheck} becomes  
$$  \frac{ \frac{C_p^{+}(2s) \chi_3(p^{b_3}) \overline{\chi_4}(p) \chi_4(p)}{p}}{F_p^{-}(-2s) \chi_2(p^{b_2})}=p^{-2b_2s},$$ where we used equation \eqref{cp} to simplify the expression. Note that the corresponding Euler factor on the right-hand side of \eqref{tocheck} is equal to 
$$p^{2s- 2b_3s} = p^{-2b_2s},$$ where we used again the fact that $b_2=b_3-1$.

When $p|(B_3,D_1)$ but $p \nmid B_2$, notice that $b_3=1$ and in this case, $C_p^{+}(2s)$ simplifies as
\begin{equation}
\label{cp2}
C_p^{+}(2s) = p \overline{\chi_3}(p).
\end{equation}
Using \eqref{cp2}, we have that the corresponding factor on the left-hand side of \eqref{tocheck} is equal to
$$ \frac{ \chi_3(p) \overline{\chi_4}(p) \chi_4(p) C_p^{+}(2s)}{p} =1, $$ which matches the Euler factor on the right-hand side of \eqref{tocheck} (again using the fact that $b_3=1$).

If $p|(B_2,B_3), p \nmid D_1 D_4$, note that $b_2=b_3$ and in this case
\begin{equation*}
E_p^{+}(2s) = E_p^{-}(-2s) \chi_2 (p^{b_2}) \overline{\chi_3}(p^{b_2}) p^{-2b_2s}.
\end{equation*}
The factor on the left-hand side of \eqref{tocheck} is equal to
$$\frac{E_p^{+}(2s)   \Big(1 - \frac{\overline{\chi_1\chi_4} \chi_2 \chi_3 (p)}{p^2} \Big)^{-1} \chi_3(p^{b_3})}{E_p^{-}(-2s)  \Big(1 - \frac{\overline{\chi_1\chi_4} \chi_2 \chi_3 (p)}{p^2} \Big)^{-1} \chi_2(p^{b_2})} =p^{-2b_3s}.$$
Note that this matches the corresponding Euler factor on the right-hand side of \eqref{tocheck}.

If $p|(B_3,D_4)$, note that this implies that $p|B_2$ and in this case, $b_2=1+b_3$ and we have
\begin{equation} 
\label{fp} F_p^{+}(2s) = C_p^{-}(-2s) \chi_2(p^{1+b_3}) \overline{\chi_3}(p^{b_3}) p^{-1-2b_3s}.\end{equation}
Using \eqref{fp}, we have that the Euler factor on the left-hand side of \eqref{tocheck} is equal to 
$$ \frac{F_p^{+}(2s) \chi_3(p^{b_3})}{\frac{C_p^{-}(-2s) \chi_2(p^{b_2}) \overline{\chi_1}(p) \chi_1(p)}{p}} = p^{-2b_3s},$$ which again matches the Euler factor on the right-hand side of \eqref{tocheck}.

Similar expressions hold for $H_p^{+}, J_p^{+}, K_p^{+}$, and we similarly check that the corresponding Euler factors match up.

Finally, if $p|A_1$, then note that $p|B_4$, and in this case $b_4= 1+a_1$. The factor corresponding to $p|A_1$ on the left-hand side of \eqref{tocheck} is equal to 
$$ \frac{\overline{\chi_4}(p^{1+a_1}) p^{-2 a_1 s}}{\overline{\chi_4}(p^{b_4})} = p^{-2 a_1 s},$$ and this matches the Euler factor on the right-hand side of \eqref{tocheck}. 

The Euler factor when $p|A_2$ is equal to (in this case, $p|B_3$ and $b_3=1+a_2$)
$$ \frac{\chi_3(p^{b_3})}{\chi_3(p^{1+a_2}) p^{2a_2s}} = p^{-2a_2s}.$$ The corresponding Euler factor on the right-hand side of \eqref{tocheck} is equal to $p^{2s-2b_3s}= p^{-2a_2s}$, where we use again the fact that $b_3=1+a_2$. 
By symmetry, we get that the remaining Euler factors over $p|A_3$ and $p|A_4$ match up, which finishes the proof of \eqref{tocheck}, and hence of the lemma.
\end{proof}

Notice that
\[
g(-s,t)\frac{\Gamma(\frac{1/2+s-it}{2})\Gamma(\frac{1/2+s+it}{2})}{\Gamma(\frac{1/2-s+it}{2})\Gamma(\frac{1/2-s-it}{2})}=g(s,t)\frac{\Gamma(\frac{1/2-s-it}{2})\Gamma(\frac{1/2-s+it}{2})}{\Gamma(\frac{1/2+s+it}{2})\Gamma(\frac{1/2+s-it}{2})}.
\]
So by changing the variables $s \mapsto -s$ in \eqref{M42-} and using Lemma \ref{identity} we obtain that
\begin{align}\label{M42-2}
& \Big(\frac{D_1D_2}{D_3D_4}\Big)^{-it}\epsilon \mathcal{M}_{4,2}^{-,OD}(D_1D_2\ell_1,D_3D_4\ell_2)\nonumber\\
&\qquad=\frac{(D_1\ell_1,D_3\ell_2)}{\sqrt{D_1D_3}\ell_1^{1/2-it}\ell_2^{1/2+it}}\chi_1\overline{\chi_3}(q)\epsilon(\chi_1)\epsilon(\overline{\chi_3})L(1,\overline{\chi_1}\chi_2)L(1,\chi_3\overline{\chi_4})\nonumber\\
&\qquad\qquad\times\frac{1}{2\pi i}\int_{(-2\eps)}c^+\Big(2s;\frac{\ell_1}{(\ell_1,\ell_2)},\frac{\ell_2}{(\ell_1,\ell_2)},D_1,D_3\Big)G(s)g(s,t)\Big(\frac{(D_2D_4)^{1/2}\ell_1\ell_2}{(D_1D_3)^{3/2}(\ell_1,\ell_2)^2}\Big)^{s}\\
&\qquad\qquad\times L(1-2s,\overline{\chi_1}\chi_3)L(1+2s,\chi_2\overline{\chi_4})\frac{\Gamma(\frac{1/2-s-it}{2})\Gamma(\frac{1/2-s+it}{2})}{\Gamma(\frac{1/2+s+it}{2})\Gamma(\frac{1/2+s-it}{2})}\frac{ds}{s}+O_\eps(q^{-\eta+\eps}).\nonumber
\end{align}
Thus combining \eqref{M13+} and \eqref{M42-2} gives
\begin{align}\label{M13M42}
&\mathcal{M}_{1,3}^{+,OD}(\ell_1,\ell_2)+ \Big(\frac{D_1D_2}{D_3D_4}\Big)^{-it}\epsilon \mathcal{M}_{4,2}^{-,OD}(D_1D_2\ell_1,D_3D_4\ell_2)\nonumber\\
&\qquad=\text{Res}_{s=0}+O_\eps(q^{-\eta+\eps})\nonumber\\
&\qquad =c^+\Big(0;\frac{\ell_1}{(\ell_1,\ell_2)},\frac{\ell_2}{(\ell_1,\ell_2)},D_1,D_3\Big)\frac{(D_1\ell_1,D_3\ell_2)}{\sqrt{D_1D_3}\ell_1^{1/2-it}\ell_2^{1/2+it}}\chi_1\overline{\chi_3}(q)\epsilon(\chi_1)\epsilon(\overline{\chi_3})\nonumber\\
&\qquad\qquad\times L(1,\overline{\chi_1}\chi_2)L(1,\chi_3\overline{\chi_4})L(1,\overline{\chi_1}\chi_3)L(1,\chi_2\overline{\chi_4})+O_\eps(q^{-\eta+\eps}).
\end{align}

Our next lemma relates the Euler product $c^+$ with the Euler product $A$ in Theorem \ref{twistedfirst}.


\begin{lem}\label{secondidentity}
We have that \begin{equation}
\label{to_check}
c^+\Big(0;\frac{\ell_1}{(\ell_1,\ell_2)},\frac{\ell_2}{(\ell_1,\ell_2)},D_1,D_3\Big)=A_{\chi_3,\chi_2,\chi_1,\chi_4}\Big(\frac{D_1\ell_1}{(D_1\ell_1,D_3\ell_2)},\frac{D_3\ell_2}{(D_1\ell_1,D_3\ell_2)}\Big).
\end{equation}
\end{lem} 
\begin{proof}
We have that
\begin{align*}
A_{\chi_3,\chi_2,\chi_1,\chi_4} & \Big(\frac{D_1\ell_1}{(D_1\ell_1,D_3\ell_2)},\frac{D_3\ell_2}{(D_1\ell_1,D_3\ell_2)}\Big) = \prod_{\substack{ p^{\lambda_1} || \frac{D_1A_1B_1}{ (D_1,B_3)(D_3,B_1)} \\ p^{\lambda_2} || \frac{D_3A_3B_3}{(D_1,B_3)(D_3,B_1)}}} \bigg((\chi_3*\chi_2)(p^{\lambda_2})(\overline{\chi_1}*\overline{\chi_4})(p^{\lambda_1}) \\ 
& +\sum_{j\geq1}\frac{(\chi_3*\chi_2)(p^{j+\lambda_2})(\overline{\chi_1}*\overline{\chi_4})(p^{j+\lambda_1})}{p^{j}}\bigg)\bigg(1+\sum_{j\geq1}\frac{(\chi_3*\chi_2)(p^j)(\overline{\chi_1}*\overline{\chi_4})(p^j)}{p^{j}}\bigg)^{-1}\\
&\times \prod_p\bigg(1+\sum_{j\geq1}\frac{(\chi_3*\chi_2)(p^j)(\overline{\chi_1}*\overline{\chi_4})(p^j)}{p^{j}}\bigg)\\
&\times\bigg(1-\frac{\chi_3\overline{\chi_1}(p)}{p}\bigg)\bigg(1-\frac{\chi_3\overline{\chi_4}(p)}{p}\bigg)\bigg(1-\frac{\chi_2\overline{\chi_1}(p)}{p}\bigg)\bigg(1-\frac{\chi_2\overline{\chi_4}(p)}{p}\bigg).
\end{align*}
Note that the above simplifies to 
\begin{align*}
 &A_{\chi_3,\chi_2,\chi_1,\chi_4}  \Big(\frac{D_1\ell_1}{(D_1\ell_1,D_3\ell_2)},\frac{D_3\ell_2}{(D_1\ell_1,D_3\ell_2)}\Big) = \frac{1}{L(2, \overline{\chi_1} \chi_2 \chi_3 \overline{\chi_4})} \\
&\qquad \times  \prod_{\substack{ p^{\lambda_1} || \frac{D_1A_1B_1}{ (D_1,B_3)(D_3,B_1)} \\ p^{\lambda_2} || \frac{D_3A_3B_3}{(D_1,B_3)(D_3,B_1)}}} \bigg((\chi_3*\chi_2)(p^{\lambda_2})(\overline{\chi_1}*\overline{\chi_4})(p^{\lambda_1}) +\sum_{j\geq1}\frac{(\chi_3*\chi_2)(p^{j+\lambda_2})(\overline{\chi_1}*\overline{\chi_4})(p^{j+\lambda_1})}{p^{j}}\bigg)\\ 
&\qquad\times \bigg(1+\sum_{j\geq1}\frac{(\chi_3*\chi_2)(p^j)(\overline{\chi_1}*\overline{\chi_4})(p^j)}{p^{j}}\bigg)^{-1} .
\end{align*}
The Euler factors when $p|B_3, p \nmid D_1D_2D_4$ simplify to (in this case $\lambda_2=b_3$ and $\lambda_1=0$)
\begin{align*}
 \prod_{\substack{p|B_3 \\ p \nmid D_1D_2D_4}}  E_p^+(0)   \Big(1 - \frac{\overline{\chi_1} \chi_2 \chi_3 \overline{\chi_4}(p)}{p^2} \Big)^{-1} \chi_3(p^{b_3}).
\end{align*} 
When $p|(B_3, D_1)$, the Euler factors simplify to (in this case, $\lambda_2 = b_3-1$ and $\lambda_1=0$):
\begin{align*}
  \prod_{p|(B_3,D_1)} C_p^+(0) \frac{\chi_3(p^{b_3})}{p}.
\end{align*}
The Euler factors when $p|(B_3,D_4)$ simplify to ($\lambda_1= 0$ and $\lambda_2 = b_3$) 
$$ \prod_{p|(B_3,D_4)} F_p^{+}(0) \chi_3(p^{b_3}).$$
When $p|(D_2,B_3)$ ($\lambda_1=0$ and $\lambda_2=b_3$), we have
$$ \prod_{p|(D_2,B_3)} \chi_3(p^{b_3}).$$
The Euler factors when $p|A_1$ give (in this case, $\lambda_1 = 1+a_1$ and $\lambda_2=0$)
$$ \prod_{p|A_1} \overline{\chi_4}(p^{1+a_1}).$$
The Euler factor when $p|D_1, p\nmid A_1B_3$ is equal to $\overline{\chi_4}(p)$. 
Similar expressions hold when we replace $D_1 \longleftrightarrow D_3$ and $B_1 \longleftrightarrow B_3$ (and $\chi_1 \longleftrightarrow \overline{\chi_3}$ and $\chi_2 \longleftrightarrow \overline{\chi_4}$). Putting things together, we get that
\begin{align*}
& A_{\chi_3,\chi_2,\chi_1,\chi_4}  \Big(\frac{D_1\ell_1}{(D_1\ell_1,D_3\ell_2)},\frac{D_3\ell_2}{(D_1\ell_1,D_3\ell_2)}\Big) \nonumber \\
 &\qquad  = \frac{\overline{\chi_1}(B_1) \chi_3(B_3) \overline{\chi_4}(D_1) \chi_2(D_3) \overline{\chi_4}(A_1) \chi_2(A_3) \overline{\chi_2}((B_1,D_3)) \chi_4((D_1,B_3))}{(B_1,D_3)(B_2,D_1) L(2,\overline{\chi_1} \chi_2 \chi_3 \overline{\chi_4})} \nonumber \\
&\qquad\qquad \times   \prod_{\substack{p|(B_3,D_1)}} C_p^+(0)  \prod_{\substack{p|B_3 \\ p \nmid D_1 D_2D_4}} E_p^+(0)  \Big(1 - \frac{\overline{\chi_1} \chi_2 \chi_3 \overline{\chi_4}(p)}{p^2} \Big)^{-1} \prod_{p|(B_3,D_4)} F_p^{+}(0) \nonumber  \\
&\qquad\qquad \times   \prod_{\substack{p|B_1 \\ p|D_3}} H_p^{+}(0) 
\prod_{\substack{p|B_1 \\ p \nmid D_3 D_2 D_4}} J_p^{+}(0)  \Big(1 - \frac{\overline{\chi_1\chi_4} \chi_2 \chi_3 (p)}{p^2} \Big)^{-1} \prod_{\substack{p|B_1 \\ p|D_2}} K_p^{+}(0).
\end{align*}
and using this together with \eqref{cs}, equation \eqref{to_check} follows.
\end{proof}

Going back to the evaluation of the off-diagonal main term, from \eqref{M13M42} and Lemma \ref{secondidentity} we get
\begin{align*}
&\mathcal{M}_{1,3}^{+,OD}(\ell_1,\ell_2)+ \Big(\frac{D_1D_2}{D_3D_4}\Big)^{-it}\epsilon \mathcal{M}_{4,2}^{-,OD}(D_1D_2\ell_1,D_3D_4\ell_2)\nonumber\\
&\qquad=\frac{A_{\chi_3,\chi_2,\chi_1,\chi_4}\big(\frac{D_1\ell_1}{(D_1\ell_1,D_3\ell_2)},\frac{D_3\ell_2}{(D_1\ell_1,D_3\ell_2)}\big)}{\big(\frac{D_1\ell_1}{(D_1\ell_1,D_3\ell_2)}\big)^{1/2-it}\big(\frac{D_3\ell_2}{(D_1\ell_1,D_3\ell_2)}\big)^{1/2+it}}\Big(\frac{D_1}{D_3}\Big)^{-it}\chi_1\overline{\chi_3}(q)\epsilon(\chi_1)\epsilon(\overline{\chi_3})\\
&\qquad\qquad\times L(1,\chi_3\overline{\chi_1})L(1,\chi_3\overline{\chi_4})L(1,\chi_2\overline{\chi_1})L(1,\chi_2\overline{\chi_4})+O_\eps(q^{-\eta+\eps}),
\end{align*}
which gives the first one-swap term in Theorem \ref{twistedfirst}.

The other one-swap terms in Theorem \ref{twistedfirst} are obtained similarly, and this completes the proof of Theorem \ref{twistedfirst}.

\section{The lower bound for the mollified moment}
\label{section_lb}
Here we will prove Theorem \ref{thm_lb}.
We have
\begin{align}
&  \frac{1}{\varphi^+(q)} \,\, \sumplus_{\substack{\chi\pmod q}}\prod_{j=1}^2LM(\tfrac12+it,\chi\chi_j)\prod_{k=3}^4LM(\tfrac12-it,\overline{\chi\chi_k}) \label{mollifiedmoment} \\
& =  \sum_{\substack{m = m_0 \cdot \ldots \cdot m_K \\ p|m_j \Rightarrow p \in I_j }} \frac{a(m;K)  }{m^{\frac{1}{2}+it}} \sum_{\substack{m_j = m_{j_1} m_{j_2} \\ \Omega(m_{j_1}) \leq \ell_j \\ \Omega(m_{j_2}) \leq \ell_j}} \chi_1(m_{j_1}) \chi_2(m_{j_2}) \nu(m_{j_1}) \nu(m_{j_2}) \mu(m_{j_1}) \mu(m_{j_2}) \nonumber  \\
 & \times \sum_{\substack{n = n_0 \cdot \ldots \cdot n_K \\ p|n_j \Rightarrow p \in I_j }} \frac{a(n;K)  }{n^{\frac{1}{2}-it}} \sum_{\substack{n_j = n_{j_1} n_{j_2} \nonumber  \\ \Omega(n_{j_1}) \leq \ell_j \\ \Omega(n_{j_2}) \leq \ell_j}}  \overline{\chi_3(n_{j_1})} \overline{\chi_4(n_{j_2})} \nu(n_{j_1}) \nu(n_{j_2}) \mu(n_{j_1}) \mu(n_{j_2})\\
& \times \frac{1}{\varphi^+(q)} \, \,  \sumplus_{\chi \pmod q}  L(\tfrac12+it,\chi\chi_1)L(\tfrac12+it,\chi\chi_2)L(\tfrac12-it,\overline{\chi\chi_3})L(\tfrac12-it,\overline{\chi\chi_4})\chi(m)\overline{\chi}(n). \nonumber 
\end{align}
Let $r=(m,n)$. We write $m=rM$ and $n=rN$ with $(M,N)=1$. 

We use Theorem \ref{twistedfirst}, and we write
\begin{align*} 
\eqref{mollifiedmoment} &= \sum_{i=1}^6 A_i + O_\varepsilon\Big(q^{\beta_K+\varepsilon}\big(q^{-11/16+96 \beta_K}D^{272}(1+|t|)^{10}\big)^{1/28}\Big)  \nonumber \\
&\qquad+O_\varepsilon\Big(q^{\beta_K+\varepsilon}\big(q^{-25/32+50 \beta_K}D^{188}(1+|t|)\big)^{1/80}\Big),
\end{align*}
where the terms $A_i$ correspond to the six main terms in Theorem \ref{twistedfirst}. Using condition \eqref{further_condition_c} in Theorem \ref{thm_lb}, we have
\begin{align} 
\label{expression_lb}
\eqref{mollifiedmoment} &= \sum_{i=1}^6 A_i + o(1).
\end{align}
Using the expression \eqref{diagonal_mt} of the diagonal term in the twisted moment computation,  we have that $A_1$ in the expression above is equal to
\begin{align}
A_1=& H(1) \sum_{\substack{r=r_0 \ldots r_K \\ M=M_0 \ldots M_K \\ N=N_0 \ldots N_K \\ (M,N)=1}} \frac{a(r;K)^2a(M;K) a(N;K) }{rMN}\label{mainterm} \\
& \times  \sum_{\substack{r_j =r_{j_1} r_{j_2} \\M_j =M_{j_1} M_{j_2} \\ \Omega(r_{j_1}M_{j_1}) \leq \ell_j \\ \Omega(r_{j_2}M_{j_2}) \leq \ell_j}} \chi_1(r_{j_1} M_{j_1}) \chi_2(r_{j_2}M_{j_2}) \nu(r_{j_1} M_{j_1}) \nu(r_{j_2}M_{j_2}) \mu(r_{j_1} M_{j_1}) \mu(r_{j_2}M_{j_2}) \nonumber \\
& \times  \sum_{\substack{r_j =r_{j_3} r_{j_4} \\N_j =N_{j_1} N_{j_2} \\ \Omega(r_{j_3}N_{j_1}) \leq \ell_j \\ \Omega(r_{j_4}N_{j_2}) \leq \ell_j}} \overline{\chi_3}(r_{j_3} N_{j_1}) \overline{\chi_4}(r_{j_4}N_{j_2}) \nu(r_{j_3} N_{j_1}) \nu(r_{j_4}N_{j_2}) \mu(r_{j_3} N_{j_1}) \mu(r_{j_4} N_{j_2}) \nonumber \\
& \times F(M,N;1)
L(1,\chi_1\overline{\chi_3})L(1,\chi_1\overline{\chi_4})L(1,\chi_2\overline{\chi_3})L(1,\chi_2\overline{\chi_4}). \nonumber 
\end{align}
We similarly obtain $5$ other terms, coming from the off-diagonal main terms in the twisted moment asymptotic (Theorem \ref{twistedfirst}). Theorem \ref{thm_lb} will follow once we prove the following two propositions.

\begin{prop}
We have
 \label{a1_big}
 \begin{equation}
 A_1 \gg 1.
 \end{equation} 
 \end{prop}
 
 \begin{prop}
 \label{aismall}
 For $2 \leq i \leq 6$, we have that
 $$A_i=o(1).$$
 \end{prop}

 \begin{proof}[Proof of Proposition \ref{a1_big}]
Using equation \eqref{mainterm}, we rewrite
 \begin{align} \label{simplification}
A_1 &= H(1)
L(1,\chi_1\overline{\chi_3})L(1,\chi_1\overline{\chi_4})L(1,\chi_2\overline{\chi_3})L(1,\chi_2\overline{\chi_4})\prod_{j=0}^K T(j) ,
 \end{align}
 where 
 \begin{align}
& T(j) = \sum_{\substack{p|r_jM_jN_j \Rightarrow p \in I_j\\ (M_j,N_j)=1}} \frac{a(r_j;K)^2a(M_j;K) a(N_j;K) }{r_jM_jN_j}  \nonumber \\
& \times  \sum_{\substack{r_j =r_{j_1} r_{j_2} \\M_j =M_{j_1} M_{j_2} \\ \Omega(r_{j_1}M_{j_1}) \leq \ell_j \\ \Omega(r_{j_2}M_{j_2}) \leq \ell_j}} \chi_1(r_{j_1} M_{j_1}) \chi_2(r_{j_2}M_{j_2}) \nu(r_{j_1} M_{j_1}) \nu(r_{j_2}M_{j_2}) \mu(r_{j_1}M_{j_1}) \mu(r_{j_2}M_{j_2}) \nonumber  \\
& \times  \sum_{\substack{r_j =r_{j_3} r_{j_4} \\N_j =N_{j_1} N_{j_2} \\ \Omega(r_{j_3}N_{j_1}) \leq \ell_j \\ \Omega(r_{j_4}N_{j_2}) \leq \ell_j}} \overline{\chi_3}(r_{j_3} N_{j_1}) \overline{\chi_4}(r_{j_4}N_{j_2}) \nu(r_{j_3} N_{j_1}) \nu(r_{j_4}N_{j_2}) \mu(r_{j_3}N_{j_1}) \mu(r_{j_4} N_{j_2})  \nonumber \\
& \times  F(M_j,N_j;1). \label{tj}
 \end{align}
For $j \geq 0$, we further write 
 \begin{align*}
 T(j) = U(j)-V(j),
 \end{align*}
 where $U(j)$ is the term with no restrictions on the number of prime factors, and $V(j)$ is the term where at least one of $\Omega(r_{j_1} M_{j_1}),\Omega(r_{j_2}M_{j_2}), \Omega(r_{j_3}N_{j_1}), \Omega(r_{j_4} N_{j_2})$ is greater than $\ell_j$. We will prove the following lemma.
 
 \begin{lem}
 \label{tjlb}
 We have that
 $$ \Big|  \prod_{j=0}^K T(j) \Big| \gg \Big| \prod_{j=0}^K U(j) \Big|.$$
 \end{lem}
 
 \begin{proof}
 
We deal with the cases $j=0$ and $j \geq 1$ differently. We first consider the case $j  \geq 1$. For the term $V(j)$ and $j \geq 1$, we have that $2^{\ell_j} < 2^{\Omega(r_j^2M_jN_j)}$. We then get that
 \begin{align*}
 |T(j)| \geq |U(j)| - \frac{1}{2^{\ell_j}} \sum_{\substack{p|r_jM_jN_j \Rightarrow p \in I_j\\ (M_j,N_j)=1}} \frac{2^{\Omega(r_j^2 M_j N_j)}\tau(r_j)^2 \tau(M_j)\tau(N_j)}{r_jM_jN_j}  |F(M_j,N_j;1)|,
 \end{align*}
 where for the second term above, we used the bounds $|a(b;K)|\leq 1, |\nu(b)|\leq 1$ for any integer $b$, and where $\tau$ denotes the divisor function. We further have that
 
 \begin{align*}
|T(j)| & \geq |U(j)|-\frac{1}{2^{\ell_j}} \prod_{p \in I_j} \Bigg( 1+ c_0 \sum_{\substack{i+h+k \geq 1 \\ hk=0}} \frac{2^{2i+h+k}(i+1)^2(h+1)(k+1)}{p^{i+h+k}} \\
& \times \Big( 1+h+k+ \sum_{\ell=1}^{\infty} \frac{ (\ell+1)(\ell+1+h+k)}{p^{\ell}}\Big) \Bigg),
 \end{align*}
 where $$c_0 = \Big(  1 - \sum_{\ell=1}^{\infty} \frac{ (\ell+1)^2}{p^{\ell}}\Big)^{-1},$$ and where we used the fact that $(M_j,N_j)=1$.
Note that for $p \in I_j$ for $j \geq 1$ (hence $p$ large enough), we can bound
$c_0 \leq \frac{4}{3}.$
 It then follows that 
 \begin{align*}
 |T(j)| & \geq |U(j)|-\frac{1}{2^{\ell_j}} \prod_{p \in I_j} \Bigg(1+ \frac{32c_0}{p} +c_0 \sum_{\substack{i+h+k \geq 2 \\ hk=0}} \frac{2^{2i+h+k}(i+1)^2(h+1)(k+1)}{p^{i+h+k}} \\
& \times \Big( 1+h+k+ \sum_{\ell=1}^{\infty} \frac{ (\ell+1)(\ell+1+h+k)}{p^{\ell}}\Big)  + \frac{24c_0}{p} \sum_{\ell=1}^{\infty} \frac{ (\ell+1)(\ell+1+h+k)}{p^{\ell}} \Bigg),
 \end{align*}
 so
 \begin{align}
 \prod_{j=1}^K &|T(j)| \geq \prod_{j=1}^K |U(j)| \prod_{j=1}^K \Bigg( 1 - \frac{1}{2^{\ell_j}|U(j)|} \prod_{p \in I_j}\Bigg(1+ \frac{32c_0}{p} +c_0 \sum_{\substack{i+j+k \geq 2 \\ hk=0}} \frac{2^{2i+h+k}(i+1)^2(h+1)(k+1)}{p^{i+h+k}} \nonumber \\
& \times \Big( 1+h+k+ \sum_{\ell=1}^{\infty} \frac{ (\ell+1)(\ell+1+h+k)}{p^{\ell}}\Big)  + \frac{24c_0}{p} \sum_{\ell=1}^{\infty} \frac{ (\ell+1)(\ell+1+h+k)}{p^{\ell}} \Bigg) \Bigg).\label{inequality}
 \end{align}
 For $j \geq 1$, similarly as above, we get that
\begin{align*}
|U(j)| & \geq  \prod_{p \in I_j} \Big( 1 - \frac{4}{p} - \frac{4 c_0}{p} \Big(  2 +\sum_{\ell=1}^{\infty} \frac{(\ell+1)(\ell+2)}{p^{\ell}}\Big)  + O \Big( \frac{1}{p^2} \Big) \Big) \\
& \geq \prod_{p \in I_j} \Big( 1 - \frac{15}{p}+ O \Big(  \frac{1}{p^2}\Big) \Big).
\end{align*} 
 
 It follows that for the second product in \eqref{inequality}, we have
 \begin{align}
 \prod_{j=1}^K & \Bigg( 1 - \frac{1}{2^{\ell_j}|U(j)|} \prod_{p \in I_j}\Bigg(1+ \frac{32c_0}{p} +c_0 \sum_{\substack{i+j+k \geq 2 \\ hk=0}} \frac{2^{2i+h+k}(i+1)^2(h+1)(k+1)}{p^{i+h+k}} \nonumber \\
& \times \Big( 1+h+k+ \sum_{\ell=1}^{\infty} \frac{ (\ell+1)(\ell+1+h+k)}{p^{\ell}}\Big) + \frac{24c_0}{p} \sum_{\ell=1}^{\infty} \frac{ (\ell+1)(\ell+1+h+k)}{p^{\ell}} \Bigg) \Bigg)\nonumber  \\
& \geq \prod_{j=1}^K \Big( 1 - \frac{1}{2^{\ell_j}} \prod_{p \in I_j} \Big(1 + \frac{59}{p} + O \Big( \frac{1}{p^2} \Big) \Big)\Big) \nonumber \\
& \geq \prod_{j=1}^K \Big( 1 - \frac{e^{60}}{2^{\ell_j}} \Big) \gg 1, \label{onebound}
 \end{align}
 where in the above we use the second bound in \eqref{condition_c} for $\beta_K$.
 
 It follows that
 $$ \prod_{j=1}^K |T(j)| \gg \prod_{j=1}^K |U(j)|.$$

 For $j=0$, we use the fact that
$$ \mathds{1}_{\Omega(b) \leq \ell_j}= \frac{1}{2 \pi i} \oint_{|z|<1} \frac{z^{\Omega(n)}}{(1-z)z^{\ell_j}}  \, \frac{dz}{z},$$ 
and hence
 \begin{align*}
 T(0) &=  \frac{1}{(2 \pi i)^4}\oint_{|z_{1}|=r} \oint_{|z_{2}|=r} \oint_{|z_{3}|=r} \oint_{|z_{4}|=r} \sum_{\substack{p|r_0 M_0 N_0 \Rightarrow p \in I_0 \\ (M_0,N_0)=1}} \frac{a(r_0;K)^2a(M_0;K)a(N_0;K) }{r_0 M_0N_0}  \\
& \times \sum_{r_0=r_{0_1}r_{0_2}} z_{1}^{\Omega(r_{0_1})}z_{2}^{\Omega(r_{0_2})}  \chi_1(r_{0_1}) \chi_2(r_{0_2}) \sum_{\substack{M_0 = M_{0_1} M_{0_2} }}z_{1}^{\Omega(M_{0_1})} z_{2}^{\Omega(M_{0_2})} \chi_1(M_{0_1}) \chi_2(M_{0_2}) \\
& \times \nu(r_{0_1} M_{0_1}) \nu(r_{0_2} M_{0_2}  \mu(r_{0_1} M_{0_1}) \mu(r_{0_2} M_{0_2}) \sum_{r_0=r_{0_3}r_{0_4}} z_{3}^{\Omega(r_{0_3})}z_{4}^{\Omega(r_{0_4})}  \overline{\chi_3}(r_{0_3}) \overline{\chi_4}(r_{0_4}) \\
& \times  \sum_{\substack{N_0 = N_{0_1} N_{0_2} }}z_{3}^{\Omega(N_{0_1})} z_{4}^{\Omega(N_{0_2})} \overline{\chi_3}(N_{0_1}) \overline{\chi_4}(N_{0_2}) \nu(r_{0_3} N_{0_1}) \nu(r_{0_4} N_{0_2})  \mu(r_{0_3} N_{0_1}) \mu(r_{0_4} N_{0_2})\\
& \times F(M_0,N_0;1) \prod_{i=1}^4 \frac{1}{(1-z_{i}) z_{i}^{\ell_0}}  \frac{dz_{1}}{z_{1}} \frac{dz_{2}}{z_{2}} \frac{dz_{3}}{z_{3}} \frac{dz_{4}}{z_{4}},
 \end{align*}
 for some $r<1$. For ease of notation, let
$$  x = \sum_{\ell \geq1}\frac{(\chi_1*\chi_2)(p^{\ell})(\overline{\chi_3}*\overline{\chi_4})(p^{\ell})}{p^{\ell }}, \, \, Y = 1+x.$$
 We rewrite the above as
 \begin{align}
&  T(0) = \frac{1}{(2 \pi i)^4} \oint_{|z_{1}|=r} \oint_{|z_{2}|=r} \oint_{|z_{3}|=r} \oint_{|z_{4}|=r} \prod_{p \in I_0} \Bigg( 1+ \frac{a(p;K)^2 (z_1 \chi_1(p)+z_2 \chi_2(p))(z_3 \overline{\chi_3}(p)+ z_4 \overline{\chi_4}(p))}{p} \nonumber  \\
 &- \frac{a(p;K)(z_1 \chi_1(p)+z_2 \chi_2(p))}{p}\Big((\overline{\chi_3}*\overline{\chi_4})(p)+ \sum_{\ell=1}^{\infty} \frac{(\chi_1*\chi_2)(p^{\ell})(\overline{\chi_3}*\overline{\chi_4})(p^{\ell+1})}{p^{\ell}}\bigg) Y^{-1}\nonumber \\
& - \frac{a(p;K)(z_3 \overline{\chi_3}(p)+z_4 \overline{\chi_4}(p))}{p}  \Big((\chi_1*\chi_2)(p)+ \sum_{\ell=1}^{\infty} \frac{(\chi_1*\chi_2)(p^{\ell+1})(\overline{\chi_3}*\overline{\chi_4})(p^{\ell})}{p^{\ell}}\bigg) Y^{-1}\nonumber \\
&+ \sum_{\substack{i+h+k \geq 2 \\ hk=0}} \frac{a (p^{2i+h+k};K) (-1)^{h+k}}{p^{i+h+k}} \sum_{n=0}^i z_1^{n} z_2^{i-n} \chi_1(p^{n}) \chi_2(p^{i-n})\sum_{m=0}^h z_1^m z_2^{h-m} \chi_1(p^m) \chi_2(p^{h-m}) \nonumber \\
& \times \frac{1}{ (n+m)! (i+h-n-m)!} \sum_{v=0}^i z_3^v z_4^{i-v} \overline{\chi_3}(p^v) \overline{\chi_4}(p^{i-v}) \sum_{t=0}^k z_3^t z_4^{k-t} \overline{\chi_3}(p^t) \overline{\chi_4}(p^{k-t}) \nonumber \\ 
& \times \frac{1}{(v+t)! (i+k-t-v)!} \mu(p^{n+m}) \mu(p^{i+h-n-m}) \mu(p^{v+t}) \mu(p^{i+k-t-v}) \Big( \mathds{1}_{hk>0}Y^{-1} \Big) \nonumber \\
& \times  \Bigg( \mathds{1}_{h>0} \bigg((\overline{\chi_3}*\overline{\chi_4})(p^{h})+\sum_{\ell\geq1}\frac{(\chi_1*\chi_2)(p^{\ell})(\overline{\chi_3}*\overline{\chi_4})(p^{\ell+h})}{p^{\ell}}\bigg) \Bigg)\Bigg) \nonumber  \\
& \times \Bigg( \mathds{1}_{k>0} \bigg(\chi_1*\chi_2)(p^{k})+\sum_{\ell \geq1}\frac{(\chi_1*\chi_2)(p^{\ell+k})(\overline{\chi_3}*\overline{\chi_4})(p^{\ell})}{p^{\ell}}\bigg) \Bigg)\Bigg) \prod_{i=1}^4 \frac{1}{(1-z_{i}) z_{i}^{\ell_0}}  \frac{dz_{1}}{z_{1}} \frac{dz_{2}}{z_{2}} \frac{dz_{3}}{z_{3}} \frac{dz_{4}}{z_{4}}.\label{t0_integral}
 \end{align}
In the four-fold integral above, we can pick $|r|=1/2$. In the integral over $z_1$, we shift the contour of integration to $|z_1|=10$, and we encounter the pole at $z_1=1$. The integral over the new contour will be bounded by
\begin{equation}
\label{et_z1}
\frac{1}{(5/4)^{\ell_0}} \prod_{p \in I_0} \Big(1+\frac{A}{p} \Big) \ll \exp\Big( -B (\log \log q)^{15/4} \Big),
\end{equation}
for some constants $A,B>0$. 

We evaluate the residue at the pole $z_1=1$, and then shift the $z_2$ contour to $|z_2|=10$, encountering the pole at $z_2=1$ and bounding the integral over the new contour by \eqref{et_z1}. Repeating the same argument for the integrals over $z_3$ and $z_4$, we get that
\begin{align}
T(0) &= U(0) +  O \Big( \exp\Big( -B (\log \log q)^{15/4} \Big) \Big),\label{t0}
\end{align}
where $U(0)$ is the main term coming from the integral \eqref{t0_integral} by evaluating the residues at $z_0=z_1=z_2=z_3=z_4=1$, and which corresponds to the term in \eqref{tj} with no restrictions on the number of prime factors (i.e., $U(j)$). 
We have that for $j \geq 0$,
 \begin{align}
 U(j) &=  \prod_{p \in I_j} \Bigg( 1+ \frac{a(p;K)^2 ( \chi_1(p)+ \chi_2(p))( \overline{\chi_3}(p)+  \overline{\chi_4}(p))}{p} \nonumber  \\
 &- \frac{a(p;K)( \chi_1(p)+ \chi_2(p))}{p}\Big((\overline{\chi_3}*\overline{\chi_4})(p)+ \sum_{\ell=1}^{\infty} \frac{(\chi_1*\chi_2)(p^{\ell})(\overline{\chi_3}*\overline{\chi_4})(p^{\ell+1})}{p^{\ell}}\bigg)Y^{-1} \nonumber \\
&  - \frac{a(p;K)( \overline{\chi_3}(p)+ \overline{\chi_4}(p)}{p}  \Big((\chi_1*\chi_2)(p))+ \sum_{\ell=1}^{\infty} \frac{(\chi_1*\chi_2)(p^{\ell+1})(\overline{\chi_3}*\overline{\chi_4})(p^{\ell})}{p^{\ell}}\bigg) Y^{-1} \nonumber \\
&+ \sum_{\substack{i+h+k \geq 2 \\ hk=0}} \frac{a (p^{2i+h+k};K) (-1)^{h+k}}{p^{i+h+k}} \sum_{n=0}^i  \sum_{v=0}^i \sum_{m=0}^h  \sum_{t=0}^k  \chi_1(p^{n+m}) \chi_2(p^{i-n+h-m}) \overline{\chi_3}(p^{v+t}) \overline{\chi_4}(p^{i-v+k-t})  \nonumber  \\
& \times \frac{\mathds{1}(m+n \leq 1) \mathds{1}(i+h-m-n \leq 1) \mathds{1}(v+t \leq 1) \mathds{1}(i+k-v-t \leq 1)}{ (n+m)! (i+h-n-m)!(v+t)! (i+k-t-v)!}   \Big( \mathds{1}_{hk>0}Y^{-1} \Big)\nonumber \\
& \times \Bigg(\mathds{1}_{h>0} \bigg((\overline{\chi_3}*\overline{\chi_4})(p^{h})+\sum_{\ell\geq1}\frac{(\chi_1*\chi_2)(p^{\ell})(\overline{\chi_3}*\overline{\chi_4})(p^{\ell+h})}{p^{\ell}}\bigg) \Bigg)  \Bigg) \nonumber\\
& \times \Bigg( \mathds{1}_{k>0} \bigg(\chi_1*\chi_2)(p^{k})+\sum_{\ell\geq1}\frac{(\chi_1*\chi_2)(p^{\ell+k})(\overline{\chi_3}*\overline{\chi_4})(p^{\ell})}{p^{\ell}}\bigg) \Bigg)\Bigg) . \label{uj}
 \end{align}

 In the expression \eqref{uj} above, we write
\begin{align*}
U(j) &= \prod_{p \in I_j} U_p,
\end{align*}
where 
$$U_p =   1+ \frac{a(p;K)^2 ( \chi_1(p)+ \chi_2(p))( \overline{\chi_3}(p)+  \overline{\chi_4}(p))}{p} - \frac{2a(p;K)( \chi_1(p)+ \chi_2(p))(\overline{\chi_3}+\overline{\chi_4})(p)}{p} +E_p  ,$$
with $E_p = O \Big(  \frac{1}{p^2} \Big)$. Also denote by $F_p$ the term in the expression above which corresponds to $1/p$. Note that
$$|F_p| \leq \frac{12}{p}.$$We rewrite 
\begin{align}
E_p &= \frac{2a(p;K) x(\chi_1(p)+\chi_2(p))(\overline{\chi_3}(p)+\overline{\chi_4}(p))}{pY} - \frac{a(p;K) (\chi_1(p)+\chi_2(p))}{pY} \sum_{\ell=1}^{\infty} \frac{(\chi_1*\chi_2)(p^{\ell})(\overline{\chi_3}*\overline{\chi_4})(p^{\ell+1})}{p^{\ell}} \nonumber \\
&- \frac{a(p;K) (\overline{\chi_3}(p)+\overline{\chi_4}(p))}{pY} \sum_{\ell=1}^{\infty} \frac{(\overline{\chi_3}*\overline{\chi_4}_2)(p^{\ell})(\chi_1*\chi_2)(p^{\ell+1})}{p^{\ell}} + \sum_{\substack{i+h+k \geq 2 \\ hk=0}} \frac{a (p^{2i+h+k};K) (-1)^{h+k}}{p^{i+h+k}} \nonumber \\
& \times  \Big( \mathds{1}_{hk>0} Y^{-1} \Big) \sum_{n=0}^i  \sum_{v=0}^i \sum_{m=0}^h  \sum_{t=0}^k  \chi_1(p^{n+m}) \chi_2(p^{i-n+h-m}) \overline{\chi_3}(p^{v+t}) \overline{\chi_4}(p^{i-v+k-t})  \nonumber  \\
& \times \frac{\mathds{1}(m+n \leq 1) \mathds{1}(i+h-m-n \leq 1) \mathds{1}(v+t \leq 1) \mathds{1}(i+k-v-t \leq 1)}{ (n+m)! (i+h-n-m)!(v+t)! (i+k-t-v)!} \nonumber \\
& \times \Bigg( \mathds{1}_{h>0} \bigg((\overline{\chi_3}*\overline{\chi_4})(p^{h})+\sum_{\ell\geq1}\frac{(\chi_1*\chi_2)(p^{\ell})(\overline{\chi_3}*\overline{\chi_4})(p^{\ell+h})}{p^{\ell}}\bigg)\Bigg) \Bigg(  \mathds{1}_{k>0} \bigg(\chi_1*\chi_2)(p^{k}) \nonumber \\
&+\sum_{\ell\geq1}\frac{(\chi_1*\chi_2)(p^{\ell+k})(\overline{\chi_3}*\overline{\chi_4})(p^{\ell})}{p^{\ell}}\bigg)\Bigg). \label{ep_expression}
\end{align}

Note that 
$$|x| \leq \sum_{\ell \geq 1} \frac{ (\ell+1)^2}{p^{\ell}} = \frac{4p^2-3p+1}{(p-1)^3} \leq \frac{32}{p},$$ and for $p \geq 64$,
$$ \frac{1}{|Y|} \leq \frac{1}{1- \frac{32}{p}} \leq 2.$$
Further note that in the expression for $E_p$, the term involving $i+h+k \leq 2$ only includes the factor corresponding to $i+h+k=2$ (since the conditions on $m,n,v,t$ force $i+h+k \geq 2$). Using the inequalities above, we then get that
\begin{align}
|E_p| & \leq \frac{16 \cdot 32}{p^2} + \frac{8}{p} \sum_{\ell=1}^{\infty} \frac{ (\ell+1)(\ell+2)}{p^{\ell}} + \frac{2}{p^2} \Bigg( 1+ 8 \Big( 2 + \sum_{\ell \geq 1} \frac{(\ell+1)(\ell+2)}{p^{\ell}} \Big)  \nonumber \\ &\qquad\qquad+ 2 \Big( 3 + \sum_{\ell \geq 1} \frac{ (\ell+1)(\ell+3)}{p^{\ell}} \Big)\Bigg) \leq \frac{1838}{p^2}.\label{ep}
\end{align}
 It follows that
 \begin{align}
 \label{u0bound}
 |U(0)| \geq \Big| \prod_{p<64} U_p \Big| \prod_{64<p \leq q^{\beta_0}} \Big(1 - \frac{12}{p}-\frac{1838}{p^2} \Big).
 \end{align}

In the case where $p <64$, note that we have
 $$a(p;K) = 1+ O \Big( \frac{1}{\log q} \Big).$$
In this case, we use the expression for $F_p$ and the expression \eqref{ep_expression} for $E_p$, and note that the term involving $p^{i+h+k}$ must have $i+h+k=2$. In this case, we have $$(i,h,k) \in \{  (2,0,0), (1,1,0), (1,0,1), (0,2,0), (0,0,2)\}.$$ When $(i,h,k)=(2,0,0)$, we must have $m=t=0, n=v=1$ in order for the term to be nonzero. Similarly, when $(i,h,k)=(0,2,0)$, we must have $n=v=t=0$ and $m=1$, and when $(i,h,k)=(0,0,2)$, we must have $n=v=m=0$ and $t=1$. When $(i,h,k)= (1,1,0)$, we must have $t=0, v\leq 1, m+n=1$, and similarly when $(i,h,k)=(1,0,1)$, we must have $m=0, n \leq 1, v+t=1$. Putting all the terms above together, the term $U_p$ greatly simplifies as

\kommentar{we get that $U_p$ simplifies as 
\begin{align*}
U_p &= 1 - \frac{ (\chi_1(p)+\chi_2(p))(\overline{\chi_3}(p)+\overline{\chi}_4(p))}{p} +\frac{2 x(\chi_1(p)+\chi_2(p))(\overline{\chi_3}(p)+\overline{\chi_4}(p)}{pY} \\
&- \frac{\chi_1(p)+\chi_2(p)}{pY} \sum_{\ell=1}^{\infty} \frac{(\chi_1*\chi_2)(p^{\ell})(\overline{\chi_3}*\overline{\chi_4})(p^{\ell+1})}{p^{\ell}} \nonumber \\
&- \frac{ \overline{\chi_3}(p)+\overline{\chi_4}(p)}{pY} \sum_{\ell=1}^{\infty} \frac{(\overline{\chi_3}*\overline{\chi_4}_2)(p^{\ell})(\chi_1*\chi_2)(p^{\ell+1})}{p^{\ell}}\nonumber \\
& + \sum_{\substack{i+h+k \geq 2 \\ hk=0}} \frac{ (-1)^{h+k}}{p^{i+h+k}} \Big( \mathds{1}_{hk>0} \frac{1}{Y} \Big) \sum_{n=0}^i  \sum_{v=0}^i \sum_{m=0}^h  \sum_{t=0}^k  \chi_1(p^{n+m}) \chi_2(p^{i-n+h-m}) \overline{\chi_3}(p^{v+t}) \overline{\chi_4}(p^{i-v+k-t})  \nonumber  \\
& \times \frac{\mathds{1}(m+n \leq 1) \mathds{1}(i+h-m-n \leq 1) \mathds{1}(v+t \leq 1) \mathds{1}(i+k-v-t \leq 1)}{ (n+m)! (i+h-n-m)!(v+t)! (i+k-t-v)!} \nonumber \\
& \times  \mathds{1}_{h>0} \bigg((\overline{\chi_3}*\overline{\chi_4})(p^{h})+\sum_{\ell\geq1}\frac{(\chi_1*\chi_2)(p^{\ell})(\overline{\chi_3}*\overline{\chi_4})(p^{\ell+h})}{p^{\ell}}\bigg) \nonumber  \\
& \times \mathds{1}_{k>0} \bigg(\chi_1*\chi_2)(p^{k})+\sum_{\ell\geq1}\frac{(\chi_1*\chi_2)(p^{\ell+k})(\overline{\chi_3}*\overline{\chi_4})(p^{\ell})}{p^{\ell}}\bigg) + O \Big(  \frac{1}{\log q}\Big).
\end{align*}
In the expression above, as previously noted, }
\begin{equation}
\label{upfinal}
U_p = 1 - \frac{(\chi_1(p)+\chi_2(p))(\overline{\chi_3}(p)+\overline{\chi_4(p)})}{p}- \frac{\chi_1 \chi_2 \overline{\chi_3 \chi_4}(p)}{p^2}+O \Big( \frac{1}{\log q} \Big). 
\end{equation}

\kommentar{{\color{red}After LOTS of by hand and Mathematica computations, I am getting that the very awful Euler product coming from the mollifier (after redefining the mollifier a bit, we can talk more about this on Monday), is equal to
$$ \prod_{p \leq q^{\beta_K}} \Big( 1 - \frac{(\chi_1(p)+\chi_2(p))(\overline{\chi_3}(p)+\overline{\chi_4(p)})}{p}- \frac{\chi_1 \chi_2 \overline{\chi_3 \chi_4}(p)}{p^2}\Big).$$
Now this seems too nice not to be true after all the computations. The final answer will be the above multiplied by the product of the $4$ L-functions. But I am concerned about the fact that for small $p$, the Euler factor above could be $0$. For example, for $p=2$, if we look at the equation 
$$ 2(a+b)(c+d)+abcd=4,$$ this does have solutions (for example, $a= b=1$, $c=e^{i\theta}$ and $d=e^{-i\theta}$, where $\theta=\arccos(3/8)$.), so the Factor corresponding to $p=2$ is $0$. Now of course this must hold for $a,b,c,d$ roots of unity, and I don't know how to show that in that case the Euler factor is not $0$. I asked ChatGPT and it claims that indeed, we can't have the equation above for roots of unity. It even came up with a solution, but I don't understand it and I'm afraid it's gibberish... 
 } }
 
We now need to prove the following lemma, which is a straightforward consequence of the arithmetic of algebraic integers in cyclotomic number fields. 
\begin{lem} 
\label{cyclotomic}
Let $m\ge 2$ be an integer. If $a,b,c,d$ are roots of unity, then 
\[
1-\frac{(a+b)(c+d)}{m}-\frac{abcd}{m^2}\ne 0.
\]
\end{lem}
\begin{proof}
The proof below works generally for all $m\ge 2$. However, if $m\ge 5$, then the lemma is clear since
\[
\left| 1-\frac{(a+b)(c+d)}{m}-\frac{abcd}{m^2} \right| \ge 1 - \frac{4}{m} - \frac{1}{m^2} \ge \frac{4}{25}.
\]
For sake of contradiction, suppose that
\[
1-\frac{(a+b)(c+d)}{m}-\frac{abcd}{m^2}= 0.
\]
Multiplying both sides of the equation by $m^2$ gives
\[
m^2-m(a+b)(c+d)-abcd=0,
\]
so that
\[
r:=abcd = m^2-m(a+b)(c+d) = m(m-xy),
\]
where we have set $x=a+b$ and $y=c+d$. Note that $m-xy \ne0$, since $r$ is a root of unity. 

If $K=\mathbb{Q}(a,b,c,d)$, then $K$ is a cyclotomic (and hence Galois) number field and $r \in K$ is a root of unity, so it lies in the ring of integers $\mathcal{O}_K^\times$ and satisfies
\[
N_{K/\mathbb{Q}}(r)=\pm1.
\]
Taking norms of the relation $r = m(m-xy)$ gives
\[
\pm1 = N_{K/\mathbb{Q}}(r)
     = N_{K/\mathbb{Q}}\big(m(m-xy)\big)
     = m^{[K:\mathbb{Q}]}\,N_{K/\mathbb{Q}}(m-xy).
\]
Hence
\[
N_{K/\mathbb{Q}}(m-xy) = \pm m^{-[K:\mathbb{Q}]}.
\]
This is impossible since $m-xy \in \mathcal{O}_K\setminus \{0\}$, so its norm $N_{K/\mathbb{Q}}(m-xy) \in \mathbb{Z}\setminus \{0\}$.
Therefore
\[
1-\frac{(a+b)(c+d)}{m}-\frac{abcd}{m^2}\ne 0,
\]
proving the lemma. 
\end{proof}

Using Lemma \ref{cyclotomic} and equation \eqref{u0bound}, we have that
\begin{equation}
\label{u0lb2}|U(0)| \gg \prod_{64<p \leq q^{\beta_0}} \Big(1- \frac{12}{p}-\frac{1838}{p^2} \Big).
\end{equation}
From equations \eqref{inequality}, \eqref{onebound} and \eqref{t0}, we get that
 \begin{align}
&  \bigg|  \prod_{j=0}^K T(j) \bigg|  \gg \bigg|  \prod_{j=0}^K U(j) \bigg| \bigg(1 - \frac{C \exp \big( -B (\log \log q)^{15/4} \big) }{|U(0)|} \bigg) , \label{ineq_prod_tj}
 \end{align}
 for some $C>0$.  
 
 Using equation \eqref{u0lb2}, it also follows that 
 $$ 1 - \frac{C \exp \big( -B (\log \log q)^{15/4}\big) }{|U(0)|} \geq 1 - C \exp \Big( -B (\log \log q)^{15/4} \Big) (\log q^{\beta_0})^{100} \geq 1/2.$$  
 
Combining the above and equation \eqref{ineq_prod_tj} finishes the proof of Lemma \ref{tjlb}.

\end{proof}

Now we go back to the proof of Proposition \ref{a1_big}. 
Using Lemma \ref{tjlb} and \eqref{simplification}, we have that
\begin{align}
\label{lb_mt}
A_1 \gg  \Big| L(1,\chi_1\overline{\chi_3})L(1,\chi_1\overline{\chi_4})L(1,\chi_2\overline{\chi_3})L(1,\chi_2\overline{\chi_4})  \prod_{p \leq q^{\beta_K}} U_p \Big|.
\end{align}
Let $\mathcal{M}$ denote the term in absolute value above.

Using Lemma $8.2$ in \cite{gs}, under GRH, we have that
\begin{equation}
\label{logl} \log L(1,\chi) = \sum_{n \leq q^{\beta_K}} \frac{\Lambda(n) \chi(n)}{n \log n}+ O \Big((\log q)^3 q^{-\beta_K/2} \Big),
\end{equation}
for $\chi \in \{ \chi_1\overline{\chi_3}, \chi_1\overline{\chi_4}, \chi_2\overline{\chi_3}, \chi_2\overline{\chi_4}\}$.

Exponentiating the above, it follows that for some constant $C_1>0$, we have
\begin{align*}
\mathcal{M} & = C_1 \Big(1 + O \Big((\log q)^3 q^{-\beta_K/2} \Big)   \bigg| \prod_{p \leq q^{\beta_K}} \Big( 1 + \frac{\chi_1 \overline{\chi_3}(p)}{p} \Big) \Big( 1 + \frac{\chi_1 \overline{\chi_4}(p)}{p} \Big) \Big( 1 + \frac{\chi_2 \overline{\chi_3}(p)}{p} \Big) \Big( 1 + \frac{\chi_2 \overline{\chi_4}(p)}{p} \Big) \bigg| \\
&\qquad\quad \times \bigg|  \prod_{p \leq q^{\beta_K}} U_p\bigg|.
\end{align*}
We further rewrite 
\begin{align}
\label{main_product}
\mathcal{M} &= C_1 \Big(1 + O \Big((\log q)^3 q^{-\beta_K/2} \Big)  \bigg| \prod_{p \leq q^{\beta_K}} \Big(1+F_p+E_p \Big)\Big( 1+ B_p+ D_p \Big) \bigg|,
\end{align}
where 
$$B_p = \frac{\chi_1 \overline{\chi_3}(p)}{p}+ \frac{\chi_1 \overline{\chi_4}(p)}{p}+ \frac{\chi_2 \overline{\chi_3}(p)}{p} + \frac{\chi_2 \overline{\chi_4}(p)}{p},$$ and  
$D_p= O(1/p^2)$. More precisely, we have
\begin{equation}
|D_p| \leq \frac{6}{p^2}+ \frac{4}{p^3}+\frac{1}{p^4} \leq \frac{9}{p^2}. \label{dp}
\end{equation}

From \eqref{main_product}, we have that
\begin{align*}
\mathcal{M} &= C_1 \Big(1 + O \big((\log q)^3 q^{-\beta_K/2} \big)\Big)  \bigg| \prod_{p \leq q^{\beta_K}} \Big( 1 + F_p+B_p+E_p+D_p+F_pB_p+ F_pD_p+E_pB_p+E_pD_p\Big) \bigg|,
\end{align*}
and note that
$$F_p+B_p = \frac{(1-a(p;K))^2 (\chi_1(p)+\chi_2(p))(\overline{\chi_3}(p)+\overline{\chi_4}(p))}{p}.$$
Now we use the fact that
$$ 1-a(p;K) \leq \frac{(1+\lambda) \log p}{\beta_K \log q},$$
and we get that
$$|F_p+B_p| \leq \frac{4(1+\lambda)^2 (\log p)^2}{p \beta_K^2 (\log q)^2}.$$
Using \eqref{dp}, \eqref{ep} and the fact that $|F_p| \leq \frac{12}{p}$ and $|B_p| \leq \frac{4}{p}$, we also have that for $p>64$,
\begin{align*}
\big|   E_p+D_p+F_pB_p+ F_pD_p+E_pB_p+E_pD_p \big| & \leq \frac{1838+9+48}{p^2}+  \frac{1838 \cdot 4+ 9 \cdot 12}{p^3}+ \frac{1838 \cdot 9}{p^4} \\
& \leq \frac{10^3}{p^2}.
\end{align*}
Hence, we have that  
\begin{align}
\bigg|  \prod_{64<p\leq p^{\beta_{K}}} &   \Big( 1 + F_p+B_p+E_p+D_p+F_pB_p+ F_pD_p+E_pB_p+E_pD_p\Big) \bigg|  \nonumber \\
& \geq  \prod_{64<p\leq p^{\beta_{K}}} \Big( 1 - \frac{4(1+\lambda)^2(\log p)^2}{p \beta_K^2 (\log q)^2} - \frac{10^3}{p^2} \Big). \label{ineq5}
\end{align}
Note that for $p > 64$, we have
$$ 1 - \frac{4(1+\lambda)^2(\log p)^2}{p \beta_K^2 (\log q)^2} - \frac{10^3}{p^2}  >1- \frac{4(1+\lambda)^2 (\log 64)^2}{64 \beta_K^2 (\log q)^2} - \frac{10^3}{64^2} > \frac{1}{2}.$$
Using the fact that $\log(1-x) > - \frac{x}{1-x}$ for $0<x<1$, we have that
\begin{align*}
\eqref{ineq5} \geq \exp \bigg(  -2 \Big( \sum_{64<p<q^{\beta_K}} \frac{4(1+\lambda)^2(\log p)^2}{p \beta_K^2 (\log q)^2} + \frac{10^3}{p^2}\Big) \bigg).
\end{align*}
Since
$$ \sum_{64<p<q^{\beta_K}} \frac{4(1+\lambda)^2(\log p)^2}{p \beta_K^2 (\log q)^2}  \leq 4(1+\lambda)^2,$$ and the sum involving $10^3/p^2$ converges, we get that
\begin{equation} 
\label{trunc}
\eqref{ineq5} \gg 1.
\end{equation}
Combining \eqref{main_product}, \eqref{ineq5} and \eqref{trunc}, we have that
$$ \mathcal{M}_1 \gg  \bigg| \prod_{p < 64} (1+F_p+E_p)(1+B_p+D_p) \bigg| .$$

In order to show that the main term above is bounded below by a constant, we are left with showing that each of the factors above is not equal to $0$. Recall that 
$$U_p=1+F_p+E_p,$$ and from equation \eqref{upfinal} and Lemma \ref{cyclotomic}, it follows that for $p<64$, $1+F_p+E_p \neq 0$.
We also have
$$\Big( 1 + \frac{\chi_1 \overline{\chi_3}(p)}{p} \Big) \Big( 1 + \frac{\chi_1 \overline{\chi_4}(p)}{p} \Big) \Big( 1 + \frac{\chi_2 \overline{\chi_3}(p)}{p} \Big) \Big( 1 + \frac{\chi_2 \overline{\chi_4}(p)}{p} \Big) = 1+ B_p+D_p,$$
and since for $\chi$ any Dirichlet character,
$$ \Big|  1+ \frac{\chi(p)}{p} \Big| \geq \frac{1}{2},$$ it follows that
$1+B_p+D_p \neq 0$.
This finishes the proof of Proposition \ref{a1_big}.
 \end{proof}

\begin{proof}[Proof of Proposition \ref{aismall}]
We will only focus on the term $A_2$,  which corresponds to the second term in Theorem \ref{twistedfirst}, since all the other ones are similar.  We have 
\begin{align}
A_2 &:= \frac{\varepsilon \overline{H(1)}}{\sqrt{D_1D_2D_3D_4}} \sum_{\substack{r=r_0 \ldots r_K \\ M=M_0 \ldots M_K \\ N=N_0 \ldots N_K \\ (M,N)=1}} \frac{a(r;K)^2a(M;K) a(N;K) (D_3D_4, M)(D_1D_2,N) }{rMN} \label{a2_term}\\
& \times  \sum_{\substack{r_j =r_{j_1} r_{j_2} \\M_j =M_{j_1} M_{j_2} \\ \Omega(r_{j_1}M_{j_1}) \leq \ell_j \\ \Omega(r_{j_2}M_{j_2}) \leq \ell_j}} \chi_1(r_{j_1} M_{j_1}) \chi_2(r_{j_2}M_{j_2}) \nu(r_{j_1} M_{j_1}) \nu(r_{j_2}M_{j_2}) \mu(r_{j_1} M_{j_1}) \mu(r_{j_2}M_{j_2}) \nonumber \\
& \times  \sum_{\substack{r_j =r_{j_3} r_{j_4} \\N_j =N_{j_1} N_{j_2} \\ \Omega(r_{j_3}N_{j_1}) \leq \ell_j \\ \Omega(r_{j_4}N_{j_2}) \leq \ell_j}} \overline{\chi_3}(r_{j_3} N_{j_1}) \overline{\chi_4}(r_{j_4}N_{j_2}) \nu(r_{j_3} N_{j_1}) \nu(r_{j_4}N_{j_2}) \mu(r_{j_3} N_{j_1}) \mu(r_{j_4} N_{j_2}) \nonumber \\
& \prod_{\substack{p^{\lambda_1}||\widetilde{D_1D_2M}\\p^{\lambda_2}||\widetilde{D_3D_4N}}}\bigg((\chi_3*\chi_4)(p^{\lambda_2})(\overline{\chi_1}*\overline{\chi_2})(p^{\lambda_1})+\sum_{j\geq1}\frac{(\chi_3*\chi_4)(p^{j+\lambda_2})(\overline{\chi_1}*\overline{\chi_2})(p^{j+\lambda_1})}{p^{j}}\bigg) \nonumber \\
&\times\bigg(1+\sum_{j\geq1}\frac{(\chi_3*\chi_4)(p^j)(\overline{\chi_1}*\overline{\chi_2})(p^j)}{p^{j}}\bigg)^{-1}
 L(1,\chi_3\overline{\chi_1})L(1,\chi_3\overline{\chi_2})L(1,\chi_4\overline{\chi_1})L(1,\chi_4\overline{\chi_2}) ,\label{a2}
\end{align} where $\widetilde{D_1D_2M} = D_1D_2M/(D_1D_2M,D_3D_4N)$ and $\widetilde{D_3D_4N} = D_3D_4N/(D_1D_2M,D_3D_4N)$.
We remove the conditions on the number of prime factors of $r_j,M_j,N_j$ in the same way we did for $A_1$, and use equation \eqref{logl} to approximate the $L$--functions as we did for $A_1$.  We then get that for some constant $C>0$, we have 
\begin{align*}
A_2 &= \frac{\varepsilon C  \overline{H(1)}}{\sqrt{D_1D_2D_3D_4}} \prod_{\substack{p \leq q^{\beta_K} \\ p \nmid D_1 D_2D_3D_4}} \Big(1+ \frac{a(p;K)^2 (\chi_1(p)+\chi_2(p))(\overline{\chi_3}(p)+\overline{\chi_4}(p))}{p}\\
& - \frac{a(p;K) (\chi_1(p)+\chi_2(p))(\overline{\chi_1}(p)+\overline{\chi_2}(p))}{p} -  \frac{a(p;K) (\chi_3(p)+\chi_4(p))(\overline{\chi_3}(p)+\overline{\chi_4}(p))}{p}  + O \Big( \frac{1}{p^2} \Big)  \Big) \\
& \times \prod_{\substack{p \leq q^{\beta_K} \\ p|D_1 D_2}} \Big( 1  - a(p;K) (\overline{\chi_3}(p)+\overline{\chi_4}(p))+ \frac{a(p;K)^2 (\chi_1(p)+\chi_2(p))(\overline{\chi_3}(p)+\overline{\chi_4}(p))}{p} \\
&- \frac{a(p;K) (\chi_1(p)+\chi_2(p))(\overline{\chi_1}(p^2)+ \overline{\chi_1 \chi_2}(p)+ \overline{\chi_2}(p^2))}{p} + O \Big( \frac{1}{p^2} \Big)\Big) \\
& \times \prod_{\substack{p \leq q^{\beta_K} \\ p|D_3D_4}}  \Big( 1  - a(p;K) (\chi_1(p)+\chi_2(p))+ \frac{a(p;K)^2 (\chi_1(p)+\chi_2(p))(\overline{\chi_3}(p)+\overline{\chi_4}(p))}{p} \\
&- \frac{a(p;K) (\overline{\chi_3}(p)+\overline{\chi_4}(p))(\chi_3(p^2)+ \chi_3 \chi_4(p)+ \chi_4(p^2))}{p} + O \Big( \frac{1}{p^2} \Big)\Big) \\
& \times \prod_{p \leq q^{\beta_K}} \Big( 1+ \frac{\chi_3 \overline{\chi_1}(p)}{p}  \Big)  \Big( 1+ \frac{\chi_3 \overline{\chi_2}(p)}{p}  \Big) \Big( 1+ \frac{\chi_4 \overline{\chi_1}(p)}{p}  \Big) \Big( 1+ \frac{\chi_4 \overline{\chi_2}(p)}{p}  \Big) (1+o(1)).
\end{align*}

We get that
\begin{align}
|A_2| & \ll \frac{1}{\sqrt{D_1D_2D_3D_4}} \prod_{p|D_1D_2D_3D_4} \Big(3 + O\Big(\frac{1}{p} \Big) \Big)  \exp \Bigg(  \Re \Big(\sum_{p \leq q^{\beta_K}} \Big(\frac{ (\chi_1(p)+\chi_2(p))(\overline{\chi_3}(p)+\overline{\chi_4}(p))}{p^{1+\frac{2\lambda}{\beta_K \log q}}} \nonumber \\
& -\frac{ (\chi_1(p)+\chi_2(p))(\overline{\chi_3}(p)+\overline{\chi_4}(p))}{p^{1+\frac{\lambda}{\beta_K \log q}}}-\frac{4}{p^{1+\frac{\lambda}{\beta_K \log q}}} + \frac{\chi_3 \overline{\chi_1}(p)}{p} + \frac{\chi_3\overline{\chi_2}(p)}{p}+ \frac{\chi_4 \overline{\chi_1}(p)}{p}+\frac{\chi_4 \overline{\chi_2}(p)}{p} \nonumber \\
& -\frac{2(\log p) (\chi_1(p)+\chi_2(p))(\overline{\chi_3}(p)+\overline{\chi_4}(p))}{\beta_K (\log q) p^{1+\frac{2\lambda}{\beta_K \log q}}} + \frac{ (\log p)(\chi_1(p)+\chi_2(p))(\overline{\chi_3}(p)+\overline{\chi_4}(p))}{\beta_K (\log q) p^{1+\frac{\lambda}{\beta_K \log q}}} \nonumber  \\
&+ \frac{ (\log p)^2  (\chi_1(p)+\chi_2(p))(\overline{\chi_3}(p)+\overline{\chi_4}(p)}{\beta_K^2 (\log q)^2 p^{1+\frac{2\lambda}{\beta_K \log q}}}\Big) \Big)\Bigg).\label{a2again}
\end{align}
Now for $\chi$ a non-principal character modulo $q$ and $\alpha \in \{\frac{\lambda}{\beta_K \log q},\frac{2\lambda}{\beta_K \log q}\}$, we write
\begin{align*}
\sum_{p \leq x} \frac{\chi(p)}{p^{1+\alpha}} = \sum_{p \leq a} \frac{\chi(p) }{p^{1+\alpha}} + \sum_{a < p \leq x} \frac{\chi(p) }{p^{1+\alpha}},
\end{align*}
for some $a<x$, where $x=q^{\beta_K}$. We trivially bound the sum over $p \leq a$ by $\log \log a$, and using GRH, Theorem $13.7$ in  \cite{MVbook} and partial summation, we have $$ \sum_{a < p \leq x} \frac{\chi(p) }{p^{1+\alpha}} \ll \frac{ \log q + \log x}{a^{1/2+\alpha}} .$$
Choosing $a =( \log q)^{2}$, we get that on GRH, 
\begin{equation}
\label{first_sum}
 \sum_{p \leq x} \frac{\chi(p)}{p^{1+\alpha}} \ll \log \log \log q.
 \end{equation}
We also have that for $j \in \{1,2\}$,
\begin{equation} \sum_{p\leq x} \frac{ \chi(p) (\log p)^j}{(\log q)^j p^{1+\alpha}} \ll \sum_{p \leq x} \frac{ (\log p)^j}{(\log q)^j p^{1+\alpha}} \ll 1,
\label{second_sum}
\end{equation}
and
\begin{equation}
\label{third_sum} \sum_{p \leq q^{\beta_K}} \frac{1}{p^{1+\frac{\lambda}{\beta_K \log q}}} = \sum_{p \leq q^{\beta_K}} \frac{1}{p} \Big(1 + O \Big(  \frac{ \log p}{\log q}\Big) \Big) = \log \log q + O(1).
\end{equation} 
Finally, we have that
$$ \prod_{p|D_1D_2D_3D_4} \Big( 3+O \Big(\frac{1}{p} \Big) \Big) \ll (D_1D_2D_3D_4)^{\varepsilon}.$$  By combining the equation above and \eqref{a2again},\eqref{first_sum}, \eqref{second_sum}, \eqref{third_sum}, we get that 
\begin{align*}
|A_2| \ll \frac{ (D_1D_2D_3D_4)^{\varepsilon}}{\sqrt{D_1D_2D_3D_4}} \exp \Big(12 \log \log  \log q  - 4 \log \log q \Big).
\end{align*}
 This shows that 
 $$A_2 = o(1).$$

We similarly show that $A_j = o(1)$, for $j =3,4,5,6.$

\end{proof}

\section{The proof of Theorems \ref{mainthm} and \ref{unconditional}}
\label{final_proofs}
First we will prove Theorem \ref{mainthm}. Let $M(1/2+it,\chi)$ be the mollifier defined in equations \eqref{mj}, \eqref{mollifier}, \eqref{beta_j}, \eqref{condition_c}, \eqref{sj}, \eqref{further_condition_c} (where we take $k=6$). H\"{o}lder's inequality implies that 
\begin{align}
&\bigg(\ \sumplus_{\substack{\chi \pmod q\\\prod_{j=1}^{4}L(\frac12+it,\chi\chi_j)\ne0}}1\bigg)^2\prod_{j=1}^4\bigg(\,\,\,\, \sumplus_{\substack{\chi \pmod q}}|L(\tfrac12+it,\chi\chi_j)M(\tfrac12+it,\chi\chi_j)|^6\bigg) \label{holder} \\
&\qquad\geq \bigg(\, \, \, \, \sumplus_{\substack{\chi \pmod  q}}\prod_{j=1}^2LM(\tfrac12+it,\chi\chi_j)\prod_{k=3}^4LM(\tfrac12-it,\overline{\chi\chi_k})\bigg)^6.\nonumber 
\end{align}
Using Theorem \ref{thm_ub} (with $k=6)$ we have that 
\begin{equation*}
\frac{1}{\varphi^+(q)}  \, \, \,  \sumplus_{\substack{\chi \pmod  q}}|L(\tfrac12+it,\chi\chi_j)M(\tfrac12+it,\chi\chi_j)|^6 \ll 1.
\end{equation*}
Using Theorem \ref{thm_lb}, we also have that 
\begin{equation*}
\frac{1}{\varphi^+(q)} \, \, \,   \sumplus_{\substack{\chi \pmod  q}}\prod_{j=1}^2LM(\tfrac12+it,\chi\chi_j)\prod_{k=3}^4LM(\tfrac12-it,\overline{\chi\chi_k})  \gg 1.
\end{equation*} 
Combining the three equations above, Theorem \ref{mainthm} follows.

Now we focus on proving Theorem \ref{unconditional}. Using the subconvexity bound $L(1/2+it,\chi \chi_j) \ll (qD(1+|t|))^{1/6+\varepsilon}$ (see \cite[Theorem 1.1]{PY}), we have the following.
\begin{align}\label{lowerbd_uncond}
&\sumplus_{\substack{\chi \pmod  q}} L(\tfrac12+it,\chi\chi_1) L(\tfrac12+it,\chi\chi_2 )L(\tfrac12-it,\overline{\chi\chi_3})L(\tfrac12-it,\overline{\chi\chi_4})\\\
&\qquad\ll_\varepsilon \big(qD(1+|t|)\big)^{2/3+\varepsilon}\cdot \#   \Big\{  \chi \pmod q \text{ even primitive} : \prod_{j=1}^{4}L(\tfrac12+it,\chi\chi_j)\ne0\Big\}. \nonumber 
\end{align}
Using Theorem \ref{twistedfirst} and simplifying the expression for $A_{\chi_1,\chi_2,\chi_3,\chi_4}$, we have that the (untwisted) fourth moment above is equal to
\begin{align}
&L(1,\chi_1\overline{\chi_3})L(1,\chi_1\overline{\chi_4})L(1,\chi_2\overline{\chi_3})L(1,\chi_2\overline{\chi_4}) L(2,\chi_1 \chi_2\overline{\chi_3}\overline{\chi_4}) \label{longexpression} \\
& + \frac{\epsilon \overline{\chi_1}(D_2) \overline{\chi_2}(D_1) \chi_3(D_4) \chi_4(D_3)}{\sqrt{D_1D_2D_3D_4}}L(1,\chi_3\overline{\chi_1})L(1,\chi_3\overline{\chi_2})L(1,\chi_4\overline{\chi_1})L(1,\chi_4\overline{\chi_2}) L(2,\chi_3\chi_4\overline{\chi_1}\overline{\chi_2}) \nonumber  \\
&+ \frac{\chi_1\overline{\chi_3}(q)\epsilon(\chi_1)\epsilon(\overline{\chi_3}) \chi_3(D_2) \overline{\chi_4}(D_1)}{\sqrt{D_1D_3}}L(1,\chi_3\overline{\chi_1})L(1,\chi_3\overline{\chi_4})L(1,\chi_2\overline{\chi_1})L(1,\chi_2\overline{\chi_4}) L(2,\chi_3\chi_2\overline{\chi_1}\overline{\chi_4}) \nonumber   \\
&+ \frac{\chi_1\overline{\chi_4}(q)\epsilon(\chi_1)\epsilon(\overline{\chi_4})\chi_2(D_4) \overline{\chi_3}(D_1)}{\sqrt{D_1D_4}}L(1,\chi_4\overline{\chi_3})L(1,\chi_4\overline{\chi_1})L(1,\chi_2\overline{\chi_3})L(1,\chi_2\overline{\chi_1}) L(2,\chi_4\chi_2\overline{\chi_3}\overline{\chi_1}) \nonumber \\
&+\frac{\chi_2\overline{\chi_3}(q)\epsilon(\chi_2)\epsilon(\overline{\chi_3})\chi_1(D_3) \overline{\chi_4}(D_2)}{\sqrt{D_2D_3}} L(1,\chi_1\overline{\chi_2})L(1,\chi_1\overline{\chi_4})L(1,\chi_3\overline{\chi_2})L(1,\chi_3\overline{\chi_4}) L(2,\chi_1\chi_3\overline{\chi_2}\overline{\chi_4}) \nonumber \\
&+ \frac{\chi_2\overline{\chi_4}(q)\epsilon(\chi_2)\epsilon(\overline{\chi_4}) \chi_1(D_4) \overline{\chi_3}(D_2)}{\sqrt{D_2D_4}}L(1,\chi_1\overline{\chi_3})L(1,\chi_1\overline{\chi_2})L(1,\chi_4\overline{\chi_3})L(1,\chi_4\overline{\chi_2}) L(2,\chi_1\chi_4\overline{\chi_3}\overline{\chi_2}) \nonumber \\
&\qquad+O_\varepsilon\Big(q^{\varepsilon}\big(q^{-11/16}D^{272}(1+|t|)^{10}\big)^{1/28}\Big)+O_\varepsilon\Big(q^{\varepsilon}\big(q^{-25/32}D^{188}(1+|t|)\big)^{1/80}\Big).\nonumber 
\end{align}
For ease of notation, let
$$R(\chi_1,\chi_2,\overline{\chi_3},\overline{\chi_4}) := L(1,\chi_1\overline{\chi_3})L(1,\chi_1\overline{\chi_4})L(1,\chi_2\overline{\chi_3})L(1,\chi_2\overline{\chi_4}) L(2,\chi_1 \chi_2\overline{\chi_3}\overline{\chi_4}).$$
Note that, without loss of generality, we can assume that 
\begin{equation}
\label{wlog}
|R(\chi_1,\chi_2,\overline{\chi_3},\overline{\chi_4})| \geq \max \Big\{ |R(\chi_3,\chi_2,\overline{\chi_1},\overline{\chi_4})|, |R(\chi_4, \chi_2,\overline{\chi_3},\overline{\chi_1})| \Big\}.
\end{equation}
Indeed, if the above is not the case, in \eqref{lowerbd_uncond} we can start with a different permutation of characters $\chi_1,\ldots, \chi_4$ such that the first term in the asymptotic formula \eqref{longexpression} is the biggest in absolute value among the $|R(\chi_{i_1},\chi_{i_2},\overline{\chi_{i_3}},\overline{\chi_{i_4}})|$, where $\{i_1,\ldots i_4\}=\{1,\ldots,4\}. $
Under the assumption \eqref{wlog}, it follows that the sum of main terms in \eqref{longexpression} is 
\begin{align*}
    \geq |R(\chi_1,\chi_2,\overline{\chi_3},\overline{\chi_4})| \Big(1 - \frac{1}{\sqrt{D_1D_2D_3D_4}}- \frac{1}{\sqrt{D_1D_3}}- \frac{1}{\sqrt{D_1D_4}}-\frac{1}{\sqrt{D_2D_3}}-\frac{1}{\sqrt{D_2D_4}}  \Big).
\end{align*}
Now if $D_1,\ldots,D_4$ are such that
\begin{equation}
\label{conddi}
    1 - \frac{1}{\sqrt{D_1D_2D_3D_4}}- \frac{1}{\sqrt{D_1D_3}}- \frac{1}{\sqrt{D_1D_4}}-\frac{1}{\sqrt{D_2D_3}}-\frac{1}{\sqrt{D_2D_4}} > 0,
\end{equation} then using the ineffective Siegel bound for the $L$--functions  (see, for instance, \cite[Theorem 11.14]{MV}), it follows that the main for the untwisted moment is 
$$\eqref{longexpression} \gg q^{-\varepsilon},$$ and combining this with \eqref{lowerbd_uncond}, the conclusion follows. 

Note that condition \eqref{conddi} is satisfied for all but a (small) finite number of discriminants $D_i$; namely, we have to consider the cases $D_1=1, D_2=3, D_3=5, D_4 \leq 37$ (there are $9$ possibilities for the value of $D_4$), $D_1=1, D_2=3, D_3=7, D_4 \leq 19$ ($4$ cases in total) and $D_1=1, D_2=5, D_3=7, D_4=11$.

To check the remaining small cases, we use inequality \eqref{lowerbd_uncond} five more times, corresponding to the one-swaps $\chi_1 \longleftrightarrow \chi_3$, $\chi_1 \longleftrightarrow \chi_4$, $\chi_2 \longleftrightarrow \chi_3$, $\chi_2 \longleftrightarrow \chi_4$ and the two-swap $\chi_1 \longleftrightarrow \chi_3, \chi_2 \longleftrightarrow \chi_4$, and we use Theorem \ref{twistedfirst} for all these moment expressions. We want to show that among the six expressions we get (each consisting of six main terms), at least one of them is $\gg q^{-\varepsilon}$. We assume for the sake of contradiction that all the main terms we obtain are $o(q^{-\varepsilon})$. 
We have the following:

$$ A \begin{pmatrix}
R(\chi_1,\chi_2,\overline{\chi_3},\overline{\chi_4}) \\[4pt]
\overline{R(\chi_1,\chi_2,\overline{\chi_3},\overline{\chi_4})} \\[4pt]
R(\chi_3,\chi_2,\overline{\chi_1},\overline{\chi_4}) \\[4pt]
\overline{R(\chi_3,\chi_2,\overline{\chi_1},\overline{\chi_4})} \\[4pt]
R(\chi_4,\chi_2,\overline{\chi_3},\overline{\chi_1}) \\[4pt]
\overline{R(\chi_4,\chi_2,\overline{\chi_3},\overline{\chi_1})}
\end{pmatrix}
=
\begin{pmatrix}
o(q^{-\varepsilon})\\
o(q^{-\varepsilon})\\
o(q^{-\varepsilon})\\
o(q^{-\varepsilon})\\
o(q^{-\varepsilon})\\
o(q^{-\varepsilon})
\end{pmatrix},$$ where $A$ is the $6 \times 6$ matrix given in Appendix $2$. 

A Mathematica computation shows that the above $6\times 6$ matrix has nonzero determinant for all the possible primitive even Dirichlet characters modulo $D_1,\ldots, D_4$, with the $D_i$ in the finite list for which  \eqref{conddi} is not satisfied. This would imply that
$$ L(1,\chi_1\overline{\chi_3})L(1,\chi_1\overline{\chi_4})L(1,\chi_2\overline{\chi_3})L(1,\chi_2\overline{\chi_4}) L(2,\chi_1 \chi_2\overline{\chi_3}\overline{\chi_4})= o(q^{\varepsilon}), $$ which contradicts the Siegel lower bound on the $L$--functions. 

It follows that in all cases,
$$\eqref{longexpression} \gg q^{-\varepsilon},$$ and again combining this with \eqref{lowerbd_uncond} finishes the proof of Theorem \ref{unconditional}.

\section{Appendix - Voronoi summation formula}
\label{appendix}

Let $D_1,D_2\geq 1$ be square-free and co-prime, and let $\chi_1$, $\chi_2$ be primitive Dirichlet characters modulo $D_1$, $D_2$, respectively. Let 
\[
D_j=D_j'(c,D_j), \qquad c=c_j(c,D_j).
\]
Since $D_1,D_2$ are square-free, $(D_j',(c,D_j))=1$. By the Chinese Remainder Theorem we may factor $\chi_j=\chi_j'\chi_j''$, where $\chi_j'$ is a primitive character modulo $D_j'$ and $\chi_j''$ is a primitive character modulo $(c,D_j)$.

We shall prove the following summation formula, which is a generalization of Theorem A.1 in \cite{BPZ}.

\begin{thm}\label{Voronoithm}
Let $(a,c)=1$. 
Then for any smooth function $g$ compactly supported in $\mathbb{R}_{>0}$ we have
\begin{align}\label{eqn:summation formula}
\sum_n (\chi_1*\chi_2)(n) e \Big(\frac{an}{c} \Big) g(n)&=\mathds{1}_{D_1|c}\frac{\sqrt{D_1}\epsilon(\chi_1)\overline{\chi_1}(a)\chi_2(c\overline{D_1})}{c}L(1,\overline{\chi_1}\chi_2)\int_0^\infty g(x)dx\nonumber\\
&\qquad+\mathds{1}_{D_2|c}\frac{\sqrt{D_2}\epsilon(\chi_2)\chi_1(c\overline{D_2})\chi_2(\overline{a})}{c}L(1,\chi_1\overline{\chi_2})\int_0^\infty g(x)dx\nonumber\\
&\qquad+  \frac{\epsilon(\chi_1')\epsilon(\chi_2')\chi_1'\chi_2'(c)\chi_1''(-\overline{aD_{2}'})\chi_2''(-\overline{aD_{1}'})}{c\sqrt{D_1'D_2'}}  T_{D_1,D_2}(a,c),
\end{align}
where $T_{D_1,D_2}(a,c)$ is given by
\begin{align*}
& -\big(1+\chi_1\chi_2(-1)\big)\pi \sum_{n \geq 1}(\overline{\chi_1'}\chi_2''*\chi_1''\overline{\chi_2'})(n)e\Big(-\frac{\overline{aD_1'D_2'}n}{c}\Big)\int_0^\infty g(x) Y_0 \Big(\frac{4\pi\sqrt{nx}}{c\sqrt{D_1'D_2'}} \Big)dx \\
&\quad -\big(1-\chi_1\chi_2(-1)\big)i\pi \sum_{n \geq 1}(\overline{\chi_1'}\chi_2''*\chi_1''\overline{\chi_2'})(n)e\Big(-\frac{\overline{aD_1'D_2'}n}{c}\Big)\int_0^\infty g(x) J_0 \Big(\frac{4\pi\sqrt{nx}}{c\sqrt{D_1'D_2'}} \Big)dx\\
&\quad+2\big(\overline{\chi_1'}\chi_2''(-1)+\chi_1''\overline{\chi_2'}(-1)\big) \sum_{n \geq 1}(\overline{\chi_1'}\chi_2''*\chi_1''\overline{\chi_2'})(n)e\Big(\frac{\overline{aD_1'D_2'}n}{c}\Big)\int_0^\infty g(x) K_0 \Big(\frac{4\pi\sqrt{nx}}{c\sqrt{D_1'D_2'}} \Big)dx.
\end{align*}
\end{thm}
\begin{proof}
For technical reasons it is convenient to study the sum
\[
S_\nu=\sum_n \tau_\nu(n,\chi_1,\chi_2) e \Big(\frac{an}{c} \Big) g(n),
\]
where
\begin{align*}
\tau_\nu(n,\chi_1,\chi_2) := \sum_{n_1n_2 = n} \chi_1(n_1)\chi_2(n_2) \Big(\frac{n_1}{n_2} \Big)^{-\nu}.
\end{align*}

We open the function $\tau_\nu(n,\chi_1,\chi_2)$ and split the sum over $n_j$ into arithmetic progressions modulo $c$ and $D_j$. This gives
\begin{align*}
S_\nu &= \sum_{\substack{u_1,u_2(\textup{mod }c)}} e \Big(\frac{au_1u_2}{c} \Big)\sum_{\substack{v_1(\textup{mod }D_1)\\v_2(\textup{mod }D_2)}}\chi_1(v_1)\chi_2(v_2)\sum_{\substack{n_1 \equiv u_1 (\textup{mod }c) \\ n_1 \equiv v_1 (\textup{mod }D_1)}}\sum_{\substack{n_2 \equiv u_2 (\textup{mod }c) \\ n_2 \equiv v_2 (\textup{mod }D_2)}}\Big(\frac{n_1}{n_2} \Big)^{-\nu} g(n_1n_2).
\end{align*}
Notice that for the congruences on $n_j$ to be compatible, we must have $u_j\equiv v_j(\text{mod}\ (c,D_j))$. We
may use the Chinese Remainder Theorem to combine the congruences on $n_j$ into a single
congruence $n_j\equiv w_j(\text{mod}\ [c,D_j])$.
Applying the Poisson summation to the sums over $n_1,n_2$ we obtain that
\begin{align}\label{sumh1h2}
S_\nu &= \frac{1}{[c,D_1][c,D_2]} \sum_{h_1,h_2 \in \mathbb{Z}}\sum_{\substack{v_1(\textup{mod }D_1)\\v_2(\textup{mod }D_2)}}\chi_1(v_1)\chi_2(v_2)\sum_{\substack{u_1,u_2(\textup{mod }c)\\u_j\equiv v_j(\text{mod}\ (c,D_j))}}e \Big(\frac{au_1u_2}{c} \Big) e \Big(\frac{h_1w_1}{[c,D_1]}+\frac{h_2w_2}{[c,D_2]} \Big)\nonumber\\
&\qquad\qquad \times\int_0^\infty \int_0^\infty \Big(\frac{x_1}{x_2} \Big)^{-\nu} e \Big(-\frac{h_1x_1}{[c,D_1]}-\frac{h_2x_2}{[c,D_2]} \Big)g(x_1x_2)dx_1 dx_2.
\end{align}
We have $[c,D_j]=cD_{j}'$, and since $D_j$ is square-free, $(c,D_j')=1$. So by reciprocity we obtain that
\[
\frac{h_jw_j}{[c,D_j]}\equiv \frac{h_j\overline{D_{j}'}u_j}{c}+\frac{h_j\overline{c}v_j}{D_{j}'}\ (\text{mod}\ 1),
\]
and hence the sums over $u_j,v_j$ in \eqref{sumh1h2} become
\begin{align}\label{sumuv}
\sum_{\substack{v_1(\textup{mod }D_1)\\v_2(\textup{mod }D_2)}}\chi_1(v_1)\chi_2(v_2)e\Big(\frac{h_1\overline{c}v_1}{D_{1}'}+\frac{h_2\overline{c}v_2}{D_{2}'}\Big)\sum_{\substack{u_1,u_2(\textup{mod }c)\\u_j\equiv v_j(\text{mod}\ (c,D_j))}}e \Big(\frac{au_1u_2+h_1\overline{D_{1}'}u_1+h_2\overline{D_{2}'}u_2}{c} \Big).
\end{align}

We first consider the case $h_1=0$, in which the sums over $u_1, v_1$ become
\begin{equation}\label{1200}
\sum_{\substack{u_1(\textup{mod }c)}}e \Big(\frac{au_1u_2}{c} \Big)\sum_{\substack{v_1(\textup{mod }D_1)\\v_1\equiv u_1(\text{mod}\ (c,D_1))}}\chi_1(v_1).
\end{equation}
Recall that $\chi_1=\chi_1'\chi_1''$, where $\chi_1'$ is a primitive character modulo $D_1'$ and $\chi_1''$ is a primitive character modulo $(c,D_1)$, so
\begin{align*}
\sum_{\substack{v_1(\textup{mod }D_1)\\v_1\equiv u_1(\text{mod}\ (c,D_1))}}\chi_1(v_1)&=\sideset{}{^*}\sum_{\substack{v_1(\textup{mod }D_1')}}\chi_1'\chi_1''\big((c,D_1)v_1+D_1'\overline{D_1'}u_1\big)\\
&=\chi_1'\big((c,D_1)\big)\chi_1''(u_1)\ \sideset{}{^*}\sum_{\substack{v_1(\textup{mod }D_1')}}\chi_1'\big(v_1\big)=
\begin{cases}
0 & \text{if }D_1'>1,\\
\chi_1(u_1) & \text{if }D_1'=1,
\end{cases}
\end{align*}
by the Chinese Remainder Theorem. Hence \eqref{1200} vanishes unless $D_1|c$, and in this case equals
\[
\sum_{\substack{u_1(\textup{mod }c)}}\chi_1(u_1)e \Big(\frac{au_1u_2}{c} \Big)=\begin{cases}
	0 & \text{if }u_2\not\equiv 0(\text{mod}\ c_1),\\
	c_1\sqrt{D_1}\epsilon(\chi_1)\overline{\chi_1}(au_2/c_1) & \text{if }u_2\equiv 0(\text{mod}\ c_1).
\end{cases}
\]
Thus \eqref{sumuv} is 
\begin{align*}
&c_1\sqrt{D_1}\epsilon(\chi_1)\chi_1(\overline {a})\sum_{\substack{v_2(\textup{mod }D_2)\\}}\chi_2(v_2)e\Big(\frac{h_2\overline{c}v_2}{D_{2}'}\Big)\sum_{\substack{u_2(\textup{mod }c)\\u_2\equiv v_2(\text{mod}\ (c,D_2))\\u_2\equiv 0(\text{mod}\ c_1)}}\overline{\chi_1}(u_2/c_1)e \Big(\frac{h_2\overline{D_{2}'}u_2}{c} \Big).
\end{align*}
The conditions that $u_2\equiv v_2(\text{mod}\ (c,D_2))$ and $u_2\equiv 0(\text{mod}\ c_1)$ force $v_2\equiv 0(\text{mod}\ (c,D_2))$. So the above expression vanishes unless $(c,D_2)=1$, and if this the case then it is equal to
\begin{align*}
	&c_1\sqrt{D_1}\epsilon(\chi_1)\chi_1(\overline{a})\sum_{\substack{v_2(\textup{mod }D_2)\\}}\chi_2(v_2)e \Big(\frac{h_2\overline{c}v_2}{D_2} \Big)\sum_{\substack{u_2(\textup{mod }c)\\u_2\equiv 0(\text{mod}\ c_1)}}\overline{\chi_1}(u_2/c_1)e \Big(\frac{h_2\overline{D_2}u_2}{c} \Big)\\
	&\qquad=c_1\sqrt{D_1D_2}\epsilon(\chi_1)\epsilon(\chi_2)\chi_1(\overline{a})\chi_2(c\overline{h_2})\sum_{\substack{u_2(\textup{mod }D_1)}}\overline{\chi_1}(u_2)e \Big(\frac{h_2\overline{D_2}u_2}{D_1} \Big)\\
	&\qquad=c\sqrt{D_2}\epsilon(\chi_2)\chi_1(-h_2\overline{aD_2})\chi_2(c\overline{h_2}).
\end{align*}
In particular, the contribution of the term $h_1=h_2=0$ to \eqref{sumh1h2} is $0$, and that of the terms with $h_1=0$, $h_2\ne 0$  is
\begin{align}\label{easyS}
&\frac{\epsilon(\chi_2)\chi_1(\overline{aD_2})\chi_2(-c)}{c\sqrt{D_2}} \sum_{h\geq 1}\chi_1\overline{\chi_2}(h)	\bigg(S_\nu\Big(0,\frac{h}{cD_2}\Big)+\chi_1\overline{\chi_2}(-1)S_\nu\Big(0,-\frac{h}{cD_2}\Big)\bigg)
\end{align}
if $D_1|c$ and $(c,D_2)=1$, where we define
\begin{align}\label{Snufor}
	S_\nu(\epsilon_1 A_1,\epsilon_2 A_2) := \int_0^\infty \int_0^\infty  \Big( \frac{x_1}{x_2}\Big)^{-\nu} e \left(\epsilon_1 A_1 x_1 +\epsilon_2A_2x_2 \right) g(x_1x_2)dx_1 dx_2
\end{align}
with $\epsilon_i \in \{\pm 1\}$ and $A_1,A_2 \geq 0$.
 From \cite[p.81]{IK} we have
\begin{align*}
&\sum_{h\geq 1}\chi_1\overline{\chi_2}(h)	\bigg(S_\nu\Big(0,\frac{h}{cD_2}\Big)+\chi_1\overline{\chi_2}(-1)S_\nu\Big(0,-\frac{h}{cD_2}\Big)\bigg)\\
&\qquad\qquad=\Big(\frac{D_1}{c}\Big)^{-2\nu}\sqrt{D_1D_2}\epsilon(\chi_1\overline{\chi_2})L(1-2\nu,\overline{\chi_1}\chi_2)\int_0^\infty g(x)x^\nu dx.
\end{align*}
Letting $\nu=0$ gives the first term on the right of \eqref{eqn:summation formula} as
\[
\epsilon(\chi_1\overline{\chi_2})\epsilon(\chi_2)=\chi_1(D_2)\overline{\chi_2}(-D_1)\epsilon(\chi_1).
\]

We are left to consider the terms with $h_1h_2\ne 0$. 
Writing $u_j(\textup{mod }c)$, $u_j\equiv v_j(\text{mod}\ (c,D_j))$ as
\[
u_j=v_j+u_j'(c,D_j)
\]
with $0\leq u_j'< c_j$, we see that the inner sum in \eqref{sumuv} is equal to
\begin{align*}
&\sum_{\substack{u_1(\textup{mod }c_1)\\u_2(\textup{mod }c_2)}}e \Big(\frac{a(v_1+u_1(c,D_1))(v_2+u_2(c,D_2))+h_1\overline{D_{1}'}(v_1+u_1(c,D_1))+h_2\overline{D_{2}'}(v_2+u_2(c,D_2))}{c} \Big)\\
&\qquad=e\Big(\frac{av_1v_2+h_1\overline{D_{1}'}v_1+h_2\overline{D_{2}'}v_2}{c}\Big)\\
&\qquad\qquad\qquad\times\sum_{\substack{u_1(\textup{mod }c_1)\\u_2(\textup{mod }c_2)}}e\Big(\frac{u_1(a(v_2+u_2(c,D_2))+h_1\overline{D_{1}'})}{c_1}+\frac{au_2(v_1+h_2\overline{aD_{2}'})}{c_2}\Big).
\end{align*}
The sum over $u_1$ vanishes unless $u_2(c,D_2)\equiv-(v_2+h_1\overline{aD_{1}'})(\text{mod}\ c_1)$. In this case we must have $v_2+h_1\overline{aD_{1}'}\equiv 0(\text{mod}\ (c,D_2))$, and the sum over $u_2$ is
\begin{align*}
\sum_{\substack{u_2(\textup{mod }c_2)\\u_2(c,D_2)\equiv-(v_2+h_1\overline{aD_{1}'})(\text{mod}\ c_1)}}e\Big(\frac{au_2(v_1+h_2\overline{aD_{2}'})}{c_2}\Big).
\end{align*}
This sum vanishes unless $v_1+h_2\overline{aD_{2}'}\equiv 0(\text{mod}\ (c,D_1))$, and if so it equals
\begin{align*}
&(c,D_1)e\Big(-\frac{a(v_1+h_2\overline{aD_{2}'})(v_2+h_1\overline{aD_{1}'})}{c}\Big)=(c,D_1)e\Big(-\frac{av_1v_2+h_1\overline{D_{1}'}v_1+h_2\overline{D_{2}'}v_2+h_1h_2\overline{aD_{1}'D_{2}'}}{c}\Big).
\end{align*}
Thus \eqref{sumuv} is
\begin{align*}
&ce\Big(-\frac{h_1h_2\overline{aD_{1}'}\overline{D_{2}'}}{c}\Big)\bigg(\sum_{\substack{v_1(\textup{mod }D_1)\\v_1\equiv -h_2\overline{aD_{2}'}(\text{mod}\ (c,D_1))}}\chi_1(v_1)e\Big(\frac{h_1\overline{c}v_1}{D_{1}'}\Big)\bigg)\bigg(\sum_{\substack{v_2(\textup{mod }D_2)\\v_2\equiv -h_1\overline{aD_{1}'}(\text{mod}\ (c,D_2))}}\chi_2(v_2)e\Big(\frac{h_2\overline{c}v_2}{D_{2}'}\Big)\bigg).\nonumber
\end{align*}

As in \cite[p. 639--640]{BPZ} we have
\begin{align*}
\sum_{\substack{v_1(\textup{mod }D_1)\\v_1\equiv -h_2\overline{aD_{2}'}(\text{mod}\ (c,D_1))}}\chi_1(v_1)e\Big(\frac{h_1\overline{c}v_1}{D_{1}'}\Big)=\sqrt{D_1'}\epsilon(\chi_1')\chi_1'(c\overline{h_1})\chi_1''(-h_2\overline{aD_{2}'}).
\end{align*}
Similarly for the sum over $v_2$ and hence the contribution of the term $h_1h_2\ne 0$ to \eqref{sumh1h2} is
\begin{align*}
& \frac{\epsilon(\chi_1')\epsilon(\chi_2')\chi_1'\chi_2'(c)\chi_1''(-\overline{aD_{2}'})\chi_2''(-\overline{aD_{1}'})}{c\sqrt{D_1'D_2'}} \sum_{h_1,h_2 \ne 0}\overline{\chi_1'}\chi_2''(h_1)\chi_1''\overline{\chi_2'}(h_2)e\Big(-\frac{h_1h_2\overline{aD_{1}'}\overline{D_{2}'}}{c}\Big)\nonumber\\
	&\qquad\qquad \times\int_0^\infty \int_0^\infty \Big(\frac{x_1}{x_2} \Big)^{-\nu} e \Big(-\frac{h_1x_1}{cD_1'}-\frac{h_2x_2}{cD_2'} \Big)g(x_1x_2)dx_1 dx_2.
\end{align*}
We write the sum over $h_1,h_2$ as
\begin{align*}
&=\sum_{h_1,h_2 \geq 1}\overline{\chi_1'}\chi_2''(h_1)\chi_1''\overline{\chi_2'}(h_2)e\Big(-\frac{h_1h_2\overline{aD_{1}'}\overline{D_{2}'}}{c}\Big)\bigg(S_\nu\Big(\frac{-h_1}{cD_1'},\frac{-h_2}{cD_2'} \Big) + \chi_1\chi_2(-1)S_\nu\Big(\frac{h_1}{cD_1'},\frac{h_2}{cD_2'} \Big) \bigg) \\
&\qquad+\sum_{h_1,h_2 \geq 1}\overline{\chi_1'}\chi_2''(h_1)\chi_1''\overline{\chi_2'}(h_2)e\Big(\frac{h_1h_2\overline{aD_{1}'}\overline{D_{2}'}}{c}\Big)\\
&\qquad\qquad\qquad\times\bigg(\chi_1''\overline{\chi_2'}(-1)S_\nu\Big(\frac{-h_1}{cD_1'},\frac{h_2}{cD_2'} \Big) + \overline{\chi_1'}\chi_2''(-1)S_\nu\Big(\frac{h_1}{cD_1'},\frac{-h_2}{cD_2'} \Big) \bigg),
\end{align*}
where $S_\nu$ is defined in \eqref{Snufor}.
 
With the change of variables
$
(x_1,x_2) \mapsto\bigg(y \sqrt{\frac{xA_2}{A_1}}, \frac{1}{y} \sqrt{\frac{xA_1}{A_2}} \bigg)
$
we get
\begin{align*}
S_\nu(\epsilon_1 A_1, \epsilon_2A_2) = \Big(\frac{A_1}{A_2} \Big)^\nu \int_0^\infty  g(x) \int_0^\infty y^{-(1+2\nu)} e \big(\sqrt{xA_1A_2} (\epsilon_1 y + \epsilon_2 y^{-1}) \big) dxdy.
\end{align*}
We can then write $S_\nu(\cdot,\cdot)$ as an integral transform of $g$ by applying the integral identities (4.112)--(4.115) in \cite{IK}. We obtain \eqref{eqn:summation formula} by letting $\nu = 0$.
\end{proof}
\newpage

\section{Appendix 2}
The matrix $A$ is given by
\[
\rotatebox{90}{\scalebox{0.33}
{$
\begin{pmatrix}
1
& \dfrac{\chi_1\chi_2\overline{\chi_3}\overline{\chi_4}(q)\epsilon(\chi_1)\epsilon(\chi_2)\epsilon(\overline{\chi_3})\epsilon(\overline{\chi_4}) \overline{\chi_1}(D_2)\overline{\chi_2}(D_1)\chi_3(D_4)\chi_4(D_3)}{\sqrt{D_1D_2D_3D_4}}
& \dfrac{\chi_1\overline{\chi_3}(q)\epsilon(\chi_1)\epsilon(\overline{\chi_3})\chi_3(D_2)\overline{\chi_4}(D_1)}{\sqrt{D_1D_3}}
& \dfrac{\chi_2\overline{\chi_4}(q)\epsilon(\chi_2)\epsilon(\overline{\chi_4})\chi_1(D_4)\overline{\chi_3}(D_2)}{\sqrt{D_2D_4}}
& \dfrac{\chi_1\overline{\chi_4}(q)\epsilon(\chi_1)\epsilon(\overline{\chi_4})\chi_2(D_4)\overline{\chi_3}(D_1)}{\sqrt{D_1D_4}}
& \dfrac{\chi_2\overline{\chi_3}(q)\epsilon(\chi_2)\epsilon(\overline{\chi_3})\chi_1(D_3)\overline{\chi_4}(D_2)}{\sqrt{D_2D_3}}
\\[50pt] \dfrac{\overline{\chi_1\chi_2}\chi_3\chi_4(q)\epsilon(\overline{\chi_1})\epsilon(\overline{\chi_2})\epsilon(\chi_3)\epsilon(\chi_4) \chi_1(D_2)\chi_2(D_1)\overline{\chi_3}(D_4)\overline{\chi_4}(D_3)}{\sqrt{D_1D_2D_3D_4}}
& 1
& \dfrac{\overline{\chi_2}\chi_4(q)\epsilon(\overline{\chi_2})\epsilon(\chi_4)\overline{\chi_1}(D_4)\chi_3(D_2)}{\sqrt{D_2D_4}}
& \dfrac{\overline{\chi_1}\chi_3(q)\epsilon(\overline{\chi}_1)\epsilon(\chi_3)\overline{\chi_3}(D_2)\chi_4(D_1)}{\sqrt{D_1D_3}}
& \dfrac{\overline{\chi_2}\chi_3(q)\epsilon(\overline{\chi_2})\epsilon(\chi_3)\overline{\chi_1}(D_3)\chi_4(D_2)}{\sqrt{D_2D_3}} &\dfrac{\overline{\chi_1}\chi_4(q)\epsilon(\overline{\chi_1})\epsilon(\chi_4)\overline{\chi_2}(D_4)\chi_3(D_1)}{\sqrt{D_1D_4}}\\[50pt] 
\dfrac{\chi_3\overline{\chi_1}(q)\epsilon(\chi_3)\epsilon(\overline{\chi_1})\chi_1(D_2)\overline{\chi_4}(D_3)}{\sqrt{D_1D_3}} & \dfrac{\chi_2\overline{\chi_4}(q)\epsilon(\chi_2)\epsilon(\overline{\chi_4})\chi_3(D_4)\overline{\chi_1}(D_2)}{\sqrt{D_2D_4}} & 1 & \dfrac{\chi_3\chi_2\overline{\chi_1}\overline{\chi_4}(q)\epsilon(\chi_3)\epsilon(\chi_2)\epsilon(\overline{\chi_1})\epsilon(\overline{\chi_4}) \overline{\chi_3}(D_2)\overline{\chi_2}(D_3)\chi_1(D_4)\chi_4(D_1)}{\sqrt{D_1D_2D_3D_4}} & \dfrac{\chi_3\overline{\chi_4}(q)\epsilon(\chi_3)\epsilon(\overline{\chi_4})\chi_2(D_4)\overline{\chi_1}(D_3)}{\sqrt{D_3D_4}} & \dfrac{\chi_2\overline{\chi_1}(q)\epsilon(\chi_2)\epsilon(\overline{\chi_1})\chi_3(D_1)\overline{\chi_4}(D_2)}{\sqrt{D_2D_1}} \\[50pt] 
\dfrac{\overline{\chi_2}\chi_4(q)\epsilon(\overline{\chi_2})\epsilon(\chi_4)\overline{\chi_3}(D_4)\chi_1(D_2)}{\sqrt{D_2D_4}} & \dfrac{\overline{\chi_3}\chi_1(q)\epsilon(\overline{\chi_3})\epsilon(\chi_1)\overline{\chi_1}(D_2)\chi_4(D_3)}{\sqrt{D_1D_3}} & \dfrac{\overline{\chi_3}\overline{\chi_2}\chi_1\chi_4(q)\epsilon(\overline{\chi_3})\epsilon(\overline{\chi_2})\epsilon(\chi_1)\epsilon(\chi_4) \chi_3(D_2)\chi_2(D_3)\overline{\chi_1}(D_4)\overline{\chi_4}(D_1)}{\sqrt{D_1D_2D_3D_4}} & 1 & \dfrac{\overline{\chi_2}\chi_1(q)\epsilon(\overline{\chi_2})\epsilon(\chi_1)\overline{\chi_3}(D_1)\chi_4(D_2)}{\sqrt{D_2D_1}} & \dfrac{\overline{\chi_3}\chi_4(q)\epsilon(\overline{\chi_3})\epsilon(\chi_4)\overline{\chi_2}(D_4)\chi_1(D_3)}{\sqrt{D_3D_4}} \\[50pt] 
\dfrac{\chi_4\overline{\chi_1}(q)\epsilon(\chi_4)\epsilon(\overline{\chi_1})\chi_2(D_1)\overline{\chi_3}(D_4)}{\sqrt{D_1D_4}}
& \dfrac{\chi_2\overline{\chi_3}(q)\epsilon(\chi_2)\epsilon(\overline{\chi_3})\chi_4(D_3)\overline{\chi_1}(D_2)}{\sqrt{D_2D_3}} & \dfrac{\chi_4\overline{\chi_3}(q)\epsilon(\chi_4)\epsilon(\overline{\chi_3})\chi_3(D_2)\overline{\chi_1}(D_4)}{\sqrt{D_4D_3}}
& \dfrac{\chi_2\overline{\chi_1}(q)\epsilon(\chi_2)\epsilon(\overline{\chi_1})\chi_4(D_1)\overline{\chi_3}(D_2)}{\sqrt{D_2D_4}} & 1 & \dfrac{\chi_4\chi_2\overline{\chi_3}\overline{\chi_1}(q)\epsilon(\chi_4)\epsilon(\chi_2)\epsilon(\overline{\chi_3})\epsilon(\overline{\chi_1}) \overline{\chi_4}(D_2)\overline{\chi_2}(D_4)\chi_3(D_1)\chi_1(D_3)}{\sqrt{D_1D_2D_3D_4}} \\[50pt] 
\dfrac{\overline{\chi_2}\chi_3(q)\epsilon(\overline{\chi_2})\epsilon(\chi_3)\overline{\chi_4}(D_3)\chi_1(D_2)}{\sqrt{D_2D_3}} & \dfrac{\overline{\chi_4}\chi_1(q)\epsilon(\overline{\chi_4})\epsilon(\chi_1)\overline{\chi_2}(D_1)\chi_3(D_4)}{\sqrt{D_1D_4}} & \dfrac{\overline{\chi_2}\chi_1(q)\epsilon(\overline{\chi_2})\epsilon(\chi_1)\overline{\chi_4}(D_1)\chi_3(D_2)}{\sqrt{D_2D_4}} & \dfrac{\overline{\chi_4}\chi_3(q)\epsilon(\overline{\chi_4})\epsilon(\chi_3)\overline{\chi_3}(D_2)\chi_1(D_4)}{\sqrt{D_4D_3}}& \dfrac{\overline{\chi_4\chi_2}\chi_3\chi_1(q)\epsilon(\overline{\chi_4})\epsilon(\overline{\chi_2})\epsilon(\chi_3)\epsilon(\chi_1) \chi_4(D_2)\chi_2(D_4)\overline{\chi_3}(D_1)\overline{\chi_1}(D_3)}{\sqrt{D_1D_2D_3D_4}} & 1 
\end{pmatrix}
$}}
\]

 \bibliographystyle{amsalpha}

\end{document}